\documentclass[11pt,oneside,reqno,letterpaper]{article}

\usepackage[top=2.5cm, bottom=2.5cm, right=2.5cm, left=2.5cm]{geometry}
\usepackage[utf8]{inputenc}
\usepackage[UKenglish]{babel}
\usepackage{changepage}
\usepackage{amsmath, mathtools, amsthm, amssymb}
\usepackage{bbm}
\usepackage{xcolor}
\usepackage{graphicx}
\usepackage{caption, subcaption}
\usepackage[shortlabels]{enumitem}
\usepackage{hyperref}
\usepackage{float}
\usepackage{titling}
\usepackage{cite}
\usepackage[titletoc,toc,title]{appendix}

\DeclareMathOperator*{\esssup}{ess\,sup}

\theoremstyle{plain}
\newtheorem{thm}{Theorem}[section] 
\newtheorem{corollary}{Corollary}[thm]
\newtheorem{lemma}[thm]{Lemma}
\newtheorem{prop}[thm]{Proposition}

\theoremstyle{definition}
\newtheorem{defn}[thm]{Definition} 
\newtheorem{assump}[thm]{Assumption} 

\theoremstyle{remark}
\newtheorem{rmk}{Remark}

\newcommand{\R}{\mathbb{R}}
\newcommand{\N}{\mathbb{N}}
\newcommand{\Prob}{\mathbb{P}}
\newcommand{\E}{\mathbb{E}}
\newcommand{\dx}{\, \mathrm{d}x}
\newcommand{\dy}{\, \mathrm{d}y}
\newcommand{\ds}{\, \mathrm{d}s}
\newcommand{\dt}{\, \mathrm{d}t}
\newcommand{\di}{\, \mathrm{d}}

\usepackage{fancyhdr}
\begin{document}

\title{An interacting particle system for the front of an epidemic \\ advancing through a susceptible population}

\author{Eliana Fausti\thanks{Department of Mathematics, Imperial College London, {\tt eliana.fausti@imperial.ac.uk}.}
		\and
		Andreas S{\o}jmark\thanks{Department of Statistics, London School of Economics, {\tt a.sojmark@lse.ac.uk}.}}

\date{\today}

\maketitle

\begin{abstract}
We introduce an interacting particle system that models the spread of an epidemic in terms of heterogeneous diffusive dynamics, rather than exogenous contact and transmission rates at the population level as in classical compartmental models. Each individual has a one-dimensional level of shielding that evolves according to a stochastic differential equation reflected at the advancing front of the epidemic. The front is driven by cumulative infections, and collisions with it represent at-risk situations which may lead to infection depending on a non-Markovian mechanism that involves the local time, the intrinsic transmissibility, and the current contagiousness within the population. We give a rigorous construction of the system and develop two key technical tools: a compensated martingale property for the infected proportion and a general result on how local time transforms under a random time-dependent bijection of the state space. The former yields a decomposition of the expected number of new infections that parallels a corresponding decomposition in the SIR model. The latter allows us to represent the law of each particle, after suitable conditioning, as a generalised elastic Brownian motion with drift.
\end{abstract}
	
\begin{adjustwidth}{0.6cm}{0.6cm}
\begin{flushleft}
			
\end{flushleft}
\end{adjustwidth}

\section{Introduction}\label{sect:intro}

This paper proposes an interacting particle system to describe the spread of an epidemic through a population of $n$ susceptible individuals. Each individual is represented by a real-valued diffusion process that tracks their level of shielding from the disease. The processes reflect off a moving boundary, which advances with every infection, and the infections are determined by collisions with this boundary in a way that depends on how contagious the system currently is. The precise mechanism is a coupled generalisation of the notion of elastic Brownian motion due to Feller.

In the study of Brownian motion on the half-line \cite{Feller_1952, ito_mckean, mckean_local}, \emph{elastic} Brownian motion is introduced as the natural object between \emph{absorbing} and \emph{reflecting} Brownian motion. The latter can be expressed as $X_t=B_t+\tfrac{1}{2}\ell_t$, where $B$ is a standard Brownian motion and
\[
    \ell_t = \lim_{\varepsilon \searrow 0} \frac{1}{\varepsilon} \int_0^t \mathbbm{1}_{[0,\varepsilon)}(X_s) \di s
\]
is the local time of $X$ at the origin. The constant $1/2$ is sometimes incorporated in the definition of $\ell$, but the above definition is more appropriate for our setting. Following \cite[Section 2.8]{mckean_local}, for any coefficient $\gamma >0$, the elastic Brownian motion $X^\gamma$ can then be obtained by setting $X^\gamma_t := X_t$ for $t < \tau$ and $X^\gamma_t := \dagger$ for $t\geq  \tau$, where $\dagger$ is a cemetery state, subject to the conditional law
\begin{equation}\label{eq:elastic_BM}
    \Prob(\tau \leq t \mid X ) = 1- \exp \{-\gamma \ell_t\},
\end{equation}
which amounts to the first time $\gamma \ell_t> \chi$, where $\chi$ is a standard Exponential variable independent of $X$. Trivially, $\gamma=0$ returns the reflecting Brownian motion, while properties of the local time give that, as $\gamma\rightarrow \infty$, one obtains the absorbing case $\tau = \inf \{t>0 : X_t =0 \}$. Note also that
\begin{equation}\label{eq:elasticBM_prob}
    \Prob\bigl(\tau \leq t +h \mid X , \{t < \tau \}\bigr)=\gamma \cdot (\ell_{t+h}-\ell_{t}) + o(\ell_{t+h}-\ell_{t}),\quad \text{as}
\;h\searrow 0.
\end{equation}
Consequently, the probability of absorption within a small time interval is approximately $\gamma$ times the expected extent to which $X^\gamma$ collides with the origin during that interval, as measured by the local time. The infection mechanism introduced below will amount to a generalisation of \eqref{eq:elasticBM_prob} within a particle system that interacts through the reflected dynamics and the parameter $\gamma$.

\subsection{The epidemic model}\label{sect:explain_particle_sys}

Consider a susceptible population of size $n$ which is exposed to an epidemic. Each individual is assigned an initial position $X_0^i$ on the half-line $[a_0,\infty)$, where $a_0$ is the starting point of what we call the \emph{advancing front} of the epidemic, denoted by $A^n_t$ at time $t$ (to be defined below). The value of $X_0^i$ represents the $i$'th individual's initial level of \emph{shielding} from the disease. This can capture a summary of how far away the $i$'th individual is from the disease geographically, the lifestyle or social demographic of the individual, and any preventive measures the individual is taking to avoid being at risk of contracting the disease (even if the disease has a presence nearby). In this way, the population can be broken up according to different characteristics initially. By relying on the level of shielding as a single index summarising various traits, we only require one spatial dimension. This is of course a simplification of reality, but the reliance on such one-dimensional indices is a typical approach in practice and avoids making the model too complicated.

Starting from $X_0^i$, we let the current level of shielding $X_t^i$ evolve randomly according to some stochastic differential equation (SDE) with drift and volatility coefficients that may depend on the level of shielding. This SDE is taken to reflect off the front $A^n$, and we say that $X_t^i$ is \emph{at-risk} precisely when it is colliding with the front. Being at-risk does not necessarily lead to infection, but it \emph{does} imply a probability of this happening which we take to obey the following principle. Let $\ell^{i}$ denote the local time of $X^i$ along $A^n$.
Then, for a given realisation of the trajectory $X^i$, we want the conditional probability of infection in a small interval $(t,t+h]$ to be of the form
\begin{equation}\label{eq:rate_of_infection}
    \Prob\bigl(\tau^i \in(t,t+h] \mid X^i, \gamma^n, \{t<\tau^i\} \bigr) \approx \int_t^{t+h} \gamma^n(s) \di \ell^i_s,
\end{equation}
where $\gamma^n$ can itself be stochastic. We refer to $\gamma^n(t)$ as the \emph{effective rate of infection} at time $t$ given that an individual is \emph{at-risk}. The instantaneous accumulation of local time $\di \ell^i$ is how we quantify the \emph{extent} to which the $i$'th individual is currently \emph{at-risk}, due to collisions with the front $A^n$.

The infection mechanism \eqref{eq:rate_of_infection} resembles \eqref{eq:elasticBM_prob}, but the form of $A^n$ and $\gamma^n$, defined in  \eqref{eq:defn_adv_front} and \eqref{eq:currently_contagious} below, couple the infection times to the system dynamics: they affect the domain $[A^n_t,\infty)$ and they impact the conditional probability of infection through both the local time and the effective rate of infection. Thus, one cannot proceed as for the elastic Brownian motion in \eqref{eq:elastic_BM}, and \eqref{eq:rate_of_infection} is only heuristic. Its precise specification is a central part of the construction presented in Section \ref{sect:main_results}.

As regards the advancing front $A^n$, we take this to be of the form
\begin{equation}\label{eq:defn_adv_front}
    A^n_t = a_0 + \alpha \int_{t-\bar{d}}^t \varrho(t-s) I^n_{s} \ds \quad \text{with}\quad     I^n_t = \frac{1}{n}\sum_{i=1}^n \mathbbm{1}_{[0,t]}(\tau^i),
\end{equation}
where $I^n$ is the total proportion of infections and $\varrho $ is a non-negative kernel integrating to $1$ on $[0,\bar{d}]$. We refer to $\varrho$ as the \emph{infection-to-recovery kernel} and $\bar{d}>0$ as the corresponding \emph{duration}, while $\alpha>0$ is the \emph{coefficient of proportionality} for how infections expand the reach of the disease. With each infection, the front advances by an amount $\alpha/n$ in a gradual motion determined by $\varrho$, until the effect cedes after $\bar{d}$ units of time. As the front advances, it becomes more likely for individuals with higher levels of shielding to find themselves at risk of infection.

It remains to clarify that the effective rate of infection $\gamma^n(t)$, in at-risk instances, should depend on a notion of the current degree of contagiousness within the system. That is, wherever the front has advanced to, given that an individual is at-risk, the likelihood of infection should be a function of how many individuals were recently infected and how contagious they currently are, together with the intrinsic transmissibility of the disease. Thus, we take $\gamma^n$ to be of the functional form
\begin{equation}\label{eq:currently_contagious}
    \gamma^n(t)=\gamma(t,C^n_t) \quad \text{for}\quad  C^n_t = \int_{t-\bar{d}}^t \varrho(t-s)(I^n_s - I^n_{s-\bar{d}}\,)\ds,
\end{equation}
where we refer to $C^n$ as the \emph{current index of contagiousness}. Note that both $C^n$ and $I^n$ live in $[0,1]$. We define it this way so that (i) very recent infections are not weighted too strongly, as those are only starting to become infectious, and (ii) any infections that were contributing to advance the front within the last $\bar{d}$ units of time are still considered but become less and less important. The definition also has the convenient property that a change of variables gives $C^n_t = (A^n_{t}-A^n_{t-\bar{d}})/\alpha$, so it can be expressed in terms of the change in the moving boundary.

\begin{rmk}
    As a concrete example, the kernel $\varrho$ could be a suitable Weibull or log-normal density which is zero at zero with negligible mass after $\bar{d}$ units of time (cut off and normalised). For the recent coronavirus pandemic, such a choice with around $\bar{d}=14$ days, is in line with how incubation times and infectiousness have been estimated; see e.g.~\cite{covid1, covid2}.
\end{rmk}

\subsection{Connections to the SIR model}\label{subsect:SIR}

It is interesting to compare our framework with the classical SIR model (see \cite[Ch.~2]{BrauerChavezFeng}). To this end, we stress that our particle system endogenises how individuals are at-risk and how they may, or may not, become infected when at-risk. Moreover, we note that, instead of trying to model contact between individuals, we lower the dimensionality by studying collisions of their shielding levels with our notion of the front of the epidemic. In the SIR model, the main variable is the extent of contact between individuals. However, this and the resulting infections are modelled at the population level and prescribed exogenously via, firstly, a general rate of contact between any individuals (infected or not) and, secondly, a general rate of infection given contact.

Departing from \cite{BrauerChavezFeng}, we shall use notation that facilitates our comparison. For the SIR model, we let $I_t$ denote the cumulative proportion of individuals that have been infected up to time $t$, in a population of size $n$, and we then have that $S_t=1-I_t$ is the proportion of susceptible individuals at time $t$. Moreover, we let $C_t$ denote the proportion of individuals that are currently infected at time $t$ (noting that this is typically denoted by $I$ in the SIR model). The SIR model posits that these three variables evolve deterministically, governed by three constants: the intrinsic transmissibility of the disease $\theta$, the general rate of contact within the population $\bar{c}$, and the duration of infectiousness $\bar{d}$. Setting $\beta := \theta \bar{c} $, the SIR model amounts to
\begin{equation}\label{eq:SIR_ODEs}
    \frac{\di}{\dt} I_t = \beta  S_t  C_t, \quad \frac{\di}{\dt}C_t = \frac{\di}{\dt} I_t - C_t / \bar{d},\quad S_t=1-I_t,
\end{equation}
where the term $C_t/ \bar{d}$ accounts for recovery of infected individuals at rate $1/ \bar{d}$.

Two key quantities are the \textit{basic reproduction number} $R_0$ and the \textit{effective reproduction number} $R_t$. These are defined, respectively, as (i) the (average) total number of cases caused by the first infection in a population where everyone is susceptible and (ii) the (average) total number of cases caused by a \emph{single} new infectious individual at time $t$. In the SIR model, holding $C_s=1/n$ and $S_s=S_t$ fixed, one new infection at time $t$ yields
\begin{equation}\label{eq:SIR_reproduction_new_infections}
   R_t = n(I_{t+\bar{d}} - I_t) = \beta S_t \bar{d}=R_0 S_t,\quad R_0=\beta \bar{d}.
\end{equation}

Instead of focusing on the basic reproduction number, as in \eqref{eq:SIR_reproduction_new_infections}, it is instructive to decompose the proportion of new infections on any interval $[t,t+h]$ as the threefold product
\begin{equation}\label{eq:SIR_new_infections}
    I_{t+h} - I_t = \int_t^{t+h}\!\!\!\!\!\!\!\!\!\!\!\!\!\!\!\!\!\!\!\!\! \underset{
    \text{proportion of susceptibles}}{\underbrace{S_{s}}} \cdot \underset{
    \text{rate of infection given contact}}{\underbrace{\theta C_{s}}}\cdot\underset{
    \text{extent of contact}}{\underbrace{\bar{c}\di s}} 
\end{equation}
Here we stress that `contact' just refers to contact with another individual, not necessarily contact with a currently infectious individual. This is why $C_s$ enters in the effective rate of infection given contact (multiplied by the intrinsic transmissibility $\theta$). It will follow from our results (see Section \ref{sect:main_results} below), that the infected proportion in our system admits a related decomposition of the form
\begin{equation}\label{eq:our_model_new_infections}
    \E\bigl[ I^n_{t+h} - I^n_{t} \mid \bar{\mathcal{F}}^n_t \bigr]  =  
    \E \Bigl[ \, \frac{1}{n}\sum_{i=1}^{n} \int_{t}^{t+h} \!\!\! \underset{\text{$i$ susceptible}}{\underbrace{\mathbbm{1}_{\{s<\tau^{i} \}}}} 
     \;\cdot\,\,\!\!\!\underset{\text{rate of infection given at-risk}}{\underbrace{\gamma(s,C_{s}^{n})}}  \!\!\!\!\cdot \,\;\,\underset{\text{extent of $i$ being at-risk}}{\underbrace{\di \ell^i_s }} \!\!\!\!\!\!\!\!\!\!\!\!\!\! \big\vert \;\bar{\mathcal{F}}^n_t\Bigr]
\end{equation}
where the filtration $\bar{\mathcal{F}}^n_t$ records the shielding levels and infection events on $[0,t]$. Since our system is stochastic, in contrast to \eqref{eq:SIR_new_infections}, the corresponding quantity is now the \emph{expected} number of new infections given the current information state. Moreover, we stress that \eqref{eq:our_model_new_infections} considers individuals, while \eqref{eq:SIR_new_infections} is purely macroscopic: the extent to which, and when, an individual is at-risk is an endogenous quantity that changes over time according to that individual's dynamics, whereas the SIR model is specified by exogenous population-wide rates. In particular, the number of new infections and the current contagiousness in our formulation depend on the realisations of the individual dynamics and is affected by the heterogeneity of the population.

As in \eqref{eq:SIR_reproduction_new_infections}, keeping $S\equiv S_t$ and $C\equiv 1/n$ fixed on $[t,t+\bar{d}]$, it follows from \eqref{eq:SIR_new_infections} that the effective reproduction number of the SIR model decomposes as
\begin{equation}\label{eq:eff_reproduction_number_SIR}
    R_t =n ( I_{t+\bar{d}} - I_t )= nS_t\; \cdot\; \frac{\theta}{n} \;\cdot\; \bar{c}\bar{d}.
\end{equation}
The number of susceptibles $nS_t$ is a given, and $R_t$ is then fully determined by the threefold product of this, the fixed rate of infection given contact $\theta/n$, and the fixed amount of contact $\bar{c}\bar{d}$ on $[t,t+\bar{d}]$. In our framework, we can similarly introduce an (expected) \emph{effective reproduction number} $R^n$ given the current state of the system: replacing $\gamma(s,C^n_s)$ by $\gamma(t,1/n)$ on $[t,t+\bar{d}]$ and keeping the set of susceptible individuals fixed, to isolate the effect of a single infection, \eqref{eq:our_model_new_infections} leads to
\begin{equation}\label{eq:our_eff_reptroduction_number}
    R^n_t = n\E[ I^n_{t+\bar{d}} - I^n_{t} \mid \bar{\mathcal{F}}^n_t ] = \sum_{i=1}^n\;\mathbbm{1}_{\{t<\tau^{i} \}} \; \cdot \; \gamma(t,\tfrac{1}{n}) \; \cdot \;\E\bigl[ \ell^i_{t+\bar{d}} -\ell^i_{t}    \mid \bar{\mathcal{F}}^n_t\bigr],
\end{equation}
where the front $A^n$ is either fixed at $A_t^n$ or taken to advance solely as $A^n_{s}=A^n_{t}+\alpha \int_{s-\bar{d}}^s \varrho(s-r)\frac{1}{n}\di r$ on $[t,t+\bar{d}]$ for the accumulation of local time. The rightmost term is total expected extent to which each susceptible will be at-risk during $[t,t+\bar{d}]$. Unlike the fixed quantity $\bar{c}\bar{d}$ in \eqref{eq:eff_reproduction_number_SIR}, this depends on the current state of each individual and their dynamics.

\subsection{Related literature}

Beyond the deterministic tradition, there is a large and growing
literature on stochastic SIR-type models focused on individual-based formulations with asymptotic results for large homogeneously mixing populations; see the survey \cite{britton_survey}. Under homogeneous mixing, \cite{pang_pardoux_2022} establishes functional limit theorems for a wide class of compartmental models in non-Markovian settings with generally distributed infectious periods, obtaining limits given by deterministic or stochastic Volterra equations. In \cite{forien_pang_pardoux}, this is extended to allow for infection-age dependent contagiousness, paralleling how the current index of contagiousness $C^n$ in our model varies over the duration of an infection according to the infection-to-recovery kernel $\varrho$. In a Markovian setting,~\cite{vuong_hauray_pardoux} establishes conditional propagation-of-chaos for a spatial SIR-type model with common noise, formulated in terms of an SDE system where infection occurs at a rate given by a bounded Lipschitz kernel applied to the relative positions of a given individual and all currently infected individuals.

The above works all concern large-population limit theorems for established classes of finite-population stochastic models, whereas our contribution is to construct and analyse a new interacting particle system in which infection arises endogenously through individuals being at-risk exactly when they collide with the front of the epidemic, as discussed in the previous subsections. Closer to our framework, but formulated as a reaction–diffusion system at the macroscopic level, is the model introduced in~\cite{berestycki_etal_plateaus} and further developed in~\cite{berestycki_desjardins_weitz_oury}: it
generalises the SIR model by incorporating behavioural heterogeneity in the susceptible population via a one-dimensional risk variable $a$ that diffuses on the positive half-line, similarly to the level of shielding in our model. In the notation of Section \ref{subsect:SIR}, it amounts to changing the first equation of the SIR dynamics \eqref{eq:SIR_ODEs} to
\begin{equation*}
\frac{\mathrm{d}}{\mathrm{d}t} I_t =\bar{\beta}_t S_t  C_t ,\quad  \bar{\beta}_t := \int_0^\infty \!\beta(a) f(t,a) \mathrm{d}a,
\end{equation*}
where $f(t,a)$ solves the Fokker--Planck equation for the risk variable and $\beta(a)$ is an increasing function which specifies the transmission rate for the susceptible proportion $S_t f(t,a)\mathrm{d}a$ of the population with risk level $[a,a+\mathrm{d}a]$. To account for this preferential infection at higher risk levels, the Fokker-Planck equation includes a sculpting term that moves mass towards lower risk levels in proportion to the share $C_t$ of the population that is currently infected. 

Individuals with low levels of shielding in our particle system correspond to high-risk parts of the population density in their setting. However, instead of a rate of infection $\beta(a)$ that acts across the entire population at all times, growing with the risk level $a$, we model the advance of the epidemic through the population so that it reaches individuals with higher levels of shielding and only then do those individuals become at-risk of infection. The reflection of the diffusing risk variable at zero in~\cite{berestycki_etal_plateaus, berestycki_desjardins_weitz_oury} serves only to confine it to the half-line. By contrast, it is precisely the local time of the collisions with the front of the epidemic, together with the current contagiousness, that determine whether or not an individual gets infected in our system.

In a different direction, \cite{Knight} studied the motion of an inert particle on the real line which is pushed away from a Brownian motion at a rate proportional to the local time of their collisions. Later, \cite{barnes} studied a system of $n$ Brownian motions with the rate of the inert particle now proportional to the empirical average of the local time accumulated along its trajectory. At least heuristically, \eqref{eq:our_model_new_infections} suggests a connection between how the inert particle and the front $A^n$ evolves: the rate of the former is directly proportional to the sum of the local times of the fully reflected Brownian motions, while the rate at which $A^n$ advances is implicitly linked to the sum of the local time of the non-infected particles with a nonlinear coefficient of proportionality $\gamma(t,C^n)$. We shall revisit this connection in Section \ref{sec:Barnes}. Here, we only note that the main focus of \cite{barnes} is the mean-field limit, while our focus is on making sense of the interaction through the infection mechanism \eqref{eq:rate_of_infection} in the finite system and developing some key properties of this, including a suitable martingale machinery which will in particular make the connection to \cite{barnes} precise.

Also related to the above, \cite{becherer, Becherer_finance} study a one-dimensional diffusion process reflected off a moving boundary given by a non-decreasing $C^1$ function of the local time along the boundary. As in \cite{Knight, barnes}, the process is fully reflected differently from our framework. Regarding our infection mechanism and how it is coupled to the front $A^n$ via $I^n$, we stress that \cite{baker_shkolnikov_undercooling} and \cite{hambly_meier} study the classical notion of elastic killing \eqref{eq:elastic_BM} for a Brownian motion that is shifted towards zero in proportion to $\mathbb{P}(\tau \leq t)$ with reflection at zero, where $\tau$ is the elastic killing time. This is motivated by the probabilistic analysis of the Stefan problem with kinetic undercooling on a half-line.

\subsection{Overview of the rest of the paper}

In Section \ref{sect:main_results} we present our main results on the well-posedness of the particle system and two key properties. The construction of the particle system is given in Section \ref{sec:existence_sys} with some technical details deferred to Appendix~\ref{appendix:particle_system}. The properties are established in Sections \ref{sect:lamperti} and \ref{sect:martingale_prop}.

In Section \ref{sect:lamperti}, we derive some general results on how the local time $\ell^\lambda_{t}(X)$ at $\lambda$ of a continuous semimartingale $X$ behaves under a time-dependent and possibly random bijection of the state $X_t \mapsto \Upsilon(t,X_t)$. For diffusion processes, the Lamperti transform is a prominent example of a useful such bijection (see \cite[Section 3]{luschgy_pages_06}). Utilising these results and a conditioning argument, we can represent each particle $X^{i,n}$ as a form of generalised elastic Brownian motion with drift.

In Section \ref{sect:martingale_prop}, we show that the infected proportion $I^n$ enjoys a compensated martingale property which gives us the decompositions \eqref{eq:our_model_new_infections} and \eqref{eq:our_eff_reptroduction_number} discussed above in Section \ref{subsect:SIR}. Moreover, it yields an asymptotic statement about the infected proportion, as the number of particles becomes large, which reveals a precise connection to a new variant of the system studied in \cite{barnes}.

\section{The particle system and its properties}\label{sect:main_results}

In this section, we introduce our assumptions and give our main results on the well-posedness of the particle system (Theorem \ref{thm:existence}) and its key properties (Theorems \ref{Thm:Brownian_transformation} and \ref{Thm:convergence_Barnes}).

\subsection{Well-posedness and key properties}

Beyond the notation introduced in Section \ref{sect:explain_particle_sys}, we let the drift and diffusion coefficients of the current level of shielding $X^i_t$ be denoted by $b(t,X_0^i,X_t^i)$ and $\sigma(t,X_0^i,X_t^i)$, respectively.

\begin{assump}[Structural conditions]\label{assump:coefficient_assumptions} Beyond joint measurability of the coefficients, on any time interval $[0,T]$ we assume that:
\begin{itemize}
    \item $b(t,x_0,x)$ is Lipschitz continuous in $x$ uniformly across $(t,x_0)$, and has at most linear growth in $(x_0,x)$ uniformly across $t$,
    \item $\sigma(t,x_0,x)$ is bounded, non-degenerate, and Lipschitz continuous in $(t,x)$ uniformly across $x_0$,
    \item $\gamma(t,x)$ is non-negative and jointly continuous in $(t,x)$, and
    \item $\varrho(t)$ is supported on $[0,\bar{d}]$, non-negative and right-continuous with $\Vert \varrho \Vert_{L^1(0,\bar{d})}=1$, for a given $\bar{d}>0$.
\end{itemize}
\end{assump}

Next, we specify the starting points and the random inputs that will drive the system dynamics. Throughout, we take as given a probability space $(\Omega, \Prob , \mathcal{F})$ that is large enough to support these.

\begin{assump}[Underlying inputs]\label{assump:inputs} We take as given the following random inputs:
\begin{itemize}
    \item a family of independent starting points $\{X_0^i\}_{i=1}^n$, which may or may not be random,
    \item a family of independent Brownian motions $\{B^i\}_{i=1}^n$, and
    \item a family of independent standard Exponential random variables $\{\chi^{j,(k)}\}_{j,k=1}^n$,
\end{itemize}
where all three families are mutually independent.
\end{assump}

Given the above, we define a filtration $(\mathcal{F}^n_t)_{t\geq 0}$ by
\begin{equation}\label{eq:full_filtration}
    \mathcal{F}^n_t:= \sigma\bigl({X_0^j,B_s^j, \{\chi^{j,(k)}\}_{k=1}^n } : s\in[0,t], j=1,\ldots,n  \bigr).
\end{equation}
In line with \cite{Knight} and the related works \cite{burdzy_nualart, burdzy_sylvester} on reflected Brownian motion with moving boundaries, we define the local time $\ell$ of a real-valued continuous semimartingale $X$ along a continuous (possibly random) curve $t \mapsto h(t)$ as
\begin{equation}\label{eq:local_time_defn}
    \ell_t(X) = \lim_{\varepsilon \searrow 0} \frac{1}{\varepsilon} \int_0^t \mathbbm{1}_{[h(s), h(s) + \varepsilon)} (X_s) \di \langle X\rangle_s,
\end{equation}
where $\langle X\rangle$ is the quadratic variation of $X$. We can now state our precise well-posedness result for the particle system discussed informally in the introduction.

\begin{thm}[Well-posed particle system]\label{thm:existence} Let Assumptions \ref{assump:coefficient_assumptions} and \ref{assump:inputs} hold, and denote by $\ell^{i,n}$ the local time of $X^{i,n}$ along the front $t\mapsto A_t^n$ on $[0,\tau^i)$ in the below. Then, Section \ref{sec:existence_sys} constructs an $\mathcal{F}^n_t$-adapted solution $\mathbf{X}^n_t = (X_t^{1,n}, \dots, X_t^{n,n})$ to
\begin{equation}\label{eq:sys_X}
    \left\{ \begin{array}{@{}l@{}l}
    \di X_t^{i,n} = b(t, X_0^{i,n}, X_t^{i,n}) \dt + \sigma (t, X_0^{i,n}, X_t^{i,n}) \di B^i_t + \tfrac{1}{2} \di \ell_t^{i,n}, & \quad t \in [0, \tau^i), \vspace{3pt}\\
    A_t^{n} = a_0 + \alpha\int_{t-\bar{d}}^t \varrho(t-s) I^{n}_{s} \ds, & \quad t \ge 0, \vspace{4pt}\\ 
    I_t^{n} = \frac{1}{n} \sum_{j=1}^n \mathbbm{1}_{[0,t]}(\tau^j), & \quad t \ge 0, \vspace{4pt}\\ 
    X_{t}^{i,n} = \dagger, &\quad t\geq \tau^{i}, 
    \end{array} \right.
\end{equation}
 living in $\left([A_t^n, \infty) \cup \{\dagger\} \right)^{n}$, 
    with initial conditions $X_0^{i,n} = X^i_0$ such that, for each  $i\in \{1,\ldots,n\}$, $\ell^{i,n}$ is flat off $\{t\in [0,\tau^i) :X^{i,n}_t = A_t^n\}$ and the infection time $\tau^i = \inf \{t \ge 0 \, : \, X^{i,n}_t = \dagger \}$ satisfies
\begin{equation}\label{eq:tau_i_prob}
    \Prob (\tau^i \le t \mid \hat{\mathcal{F}}^{i,n}_t) = 1 - \exp \left\{ - \int_0^t \gamma(s, \hat{C}^{n, (-i)}_s) \di \hat{\ell}_s^{i,n} \right\}, \quad t \ge 0,
\end{equation}
with respect to the reduced-information filtration
\begin{equation}\label{eq:reduced_filtration}
    \hat{\mathcal{F}}^{i,n}_t:= \sigma \bigl( (X_0^i,B_s^i), (X_0^j,B_s^j, \{\chi^{j,(k)}\}_{k=1}^n)  : s\in[0,t] , j\in \{ 1,\ldots,n  \}\! \setminus \! \{i\}    \bigr).
    \end{equation}
Here, each $\hat{\ell}^{i,n}$ in \eqref{eq:tau_i_prob} denotes the local time of the $i^{\text{th}}$ particle along the corresponding advancing front $\hat{A}^{n, (-i)}$ in the auxiliary particle system $\mathbf{\hat{X}}^{n,(-i)}=(\hat{X}^{1,n,(-i)},\ldots,\hat{X}^{n,n,(-i)})$ for which the dynamics and interactions are as in $\mathbf{X}^n$, except that $\hat{X}^{i,n,(-i)}$ is fully reflected with no effect on the other particles indexed by $j\neq i$. $\hat{C}^{n, (-i)}$ correspondingly denotes the current contagiousness in this system, defined as in \eqref{eq:currently_contagious} but with $I^n$ replaced by the infected proportion $\hat{I}^{n, (-i)}$ for the auxiliary system $\mathbf{\hat{X}}^{n,(-i)}$. Finally, the construction of $\mathbf{X}^n$ is uniquely determined by the inputs, and, in particular, $\mathbf{X}^n$ is unique in law for any inputs satisfying Assumption \ref{assump:inputs}.
\end{thm}

The proof of Theorem \ref{thm:existence} is given in Section \ref{sec:existence_sys} and the precise definition of the auxiliary systems $\mathbf{\hat{X}}^{n,(-i)}$, for $i = 1,\ldots,n$, can be found in \eqref{eq:sys_X_minus_i} of Proposition~\ref{prop:artificial_tagged_i}. The main intricacies pertain to the coupling of the infection times and the particle dynamics through the moving boundary and the effective rate of infection which are themselves determined by the infection times. In particular, the specification \eqref{eq:tau_i_prob}--\eqref{eq:reduced_filtration} is essential: it ensures there is no unintended temporal circularity in how $\tau^i$ is determined, as it should not depend on the trajectories after the infection time. Moreover, we note that the specifics of the construction are central to the proofs of the conditional elastic representation in Theorem \ref{Thm:Brownian_transformation} and the martingale property in Theorem \ref{Thm:convergence_Barnes}, which underpins the observations about the infected proportion discussed in Section \ref{subsect:SIR}.

\begin{rmk}[Dependence on initial values]
In Theorem~\ref{thm:existence} the coefficients can depend on $X_0^{i,n}$, allowing for different dynamics depending on the initial level of shielding. For example, one may wish to divide the population into subgroups according to the initial shielding and let each group mean-revert around a given level of shielding (see Section \ref{subsect:illustration_of_model}). To simplify the notation, we suppress the $X_0^{i,n}$-dependence throughout, but our arguments of course account for it.
\end{rmk}

It will be useful to decompose the reduced-information filtrations $	\hat{\mathcal{F}}^{i,n}$ from \eqref{eq:reduced_filtration} as
\[
    \hat{\mathcal{F}}^{i,n}_t = \mathcal{G}_{t}^{i,n} \lor \sigma(B^i_r,X_0^i : r\leq t),
\]
\begin{equation*}
    \mathcal{G}_{t}^{i,n}:=\sigma((X_{0}^{j},B_{r}^{j},\chi^{j,(k)}):j\neq i,\;k\leq n,\,r\leq t).
\end{equation*}
 Based on this, the next result highlights a crucial consequence of the structure of the particle system. Namely that, for any given particle $X^{i,n}$, we can, in a suitable sense, `freeze' its interactions with the other particles, by conditioning on $\mathcal{G}^{i,n}_t$, and one can then recast its conditional law in terms of a generalised elastic Brownian motion with drift in the frame of the advancing front.

\begin{thm}[Generalised elastic Brownian motion]\label{Thm:Brownian_transformation}
Fix any $i\in\{1,\ldots,n\}$ and let $\hat{A}^{n, (-i)}$ denote the advancing front of the auxiliary system $\mathbf{\hat{X}}^{n,(-i)}$. There is a random time-dependent bijection $\Upsilon^{i,n}(t,\omega , \cdot):[0,\infty) \rightarrow [0,\infty)$ and a jointly measurable function $(t,\omega,z)\mapsto \tilde{b}^{i,n}(t,\omega,z)$, both of which are adapted to $\mathcal{G}^{i,n}$, such that, almost surely, we have
\begin{equation}
    \Prob(X_{t}^{i,n}\in(a,b)\mid\mathcal{G}_{t}^{i,n})=\Prob(Z_{t}^{i,n}\in(\Upsilon^{i,n}(t,a-\hat{A}_{t}^{n, (-i)}),\Upsilon^{i,n}(t,b-\hat{A}_{t}^{n, (-i)})),\;t<\tau_{Z}^{i,n}\mid\mathcal{G}_{t}^{i,n}),\label{eq:Prob_Brownian}
\end{equation}
for any $(a,b)\subseteq [A_t^n,\infty)$ and all $t\geq0$, with 
\begin{equation}\label{eq:Z_i_n_dynamics}
    \di Z_{t}^{i,n} = \tilde{b}^{i,n}(t,\omega,Z_{t}^{i,n})\dt + \di B_{t}^{i,n} + \tfrac{1}{2}\di \ell_{t}^{0}(Z^{i,n}),\quad t\geq0,
\end{equation}
where $\ell^{0}(Z^{i,n})$ is the local time of $Z^{i,n}$ at the origin, so that $Z^{i,n}$ lives in $\R^+$, $B^{i,n}$ is a Brownian motion  independent of $\mathcal{G}^{i,n}$, and the infection time $\tau_{Z}^{i,n}$ is given by
\begin{equation}
    \tau_{Z}^{i,n}=\inf\Bigl\{ t\geq0:\int_{0}^{t} \sigma(r,\hat{A}_{r}^{n, (-i)}) \gamma(r,\hat{C}_{r}^{n, (-i)}) \di  \ell_{r}^{0} (Z^{i,n}) > \chi^{i,n} \Bigr\} \label{eq:tau_Upsilon}
\end{equation}
for a standard exponential random variable $\chi^{i,n}$ independent
of $\hat{\mathcal{F}}^{i,n}$.
\end{thm}

The proof is given in Section \ref{sect:lamperti_for_particle_sys}, where one can find the constructions of $\Upsilon^{i,n}$ and $\tilde{b}^{i,n}$. The result is an important analytical device: the tower law with conditioning on $\mathcal{G}^{i,n}$ allows one to untangle correlations between the particles and describe the infection time by \eqref{eq:tau_Upsilon} akin to a standard elastic Brownian motion. Up to having control over the drift $\tilde{b}^{i,n}$ (which is of at most linear growth by Proposition \ref{prop:Lamperti}), it also allows for Girsanov arguments. Moreover, it can aid the implementation and analysis of numerical schemes by avoiding discretisation of the multiplicative noise.

The key quantity in our epidemic model is the infected proportion $I^n$. The next results shows that its evolution may be compensated, in terms of the local times of the remaining susceptible individuals, to form a martingale. Furthermore, we make an insightful observation about the asymptotic behaviour of $I^n$ as the population tends to infinity.

\begin{thm}[Martingale property and large $n$ asymptotics]\label{Thm:convergence_Barnes}
For every $n\geq 1$, we define
\begin{equation*}
    V_t^n:=\frac{1}{n}\sum_{i=1}^{n}\int_0^{t}\mathbbm{1}_{s<\tau^i}\gamma(s,C_s^n) \di \ell^{i,n}_s,\quad t\geq0,
\end{equation*}
where we recall that $C^n_t = \int_{t-\bar{d}}^t \varrho(t-s)(I^n_s - I^n_{s-\bar{d}}\,)\ds$. Then, the difference $I^n-V^n$ is a martingale for the subfiltration $\mathcal{\bar{{F}}}_{t}^{n}$ of $\mathcal{{F}}_{t}^{n}$ given by
\[
\mathcal{\bar{{F}}}_{t}^{n}=\sigma(X_{0}^{j},B_{s}^{j},\{s<\tau^{j}\}:s\in[0,t],\,j\in\{1,\ldots,n\}).
\]
Furthermore, as $n\rightarrow \infty$, $\sup_{s\leq t}|I^n_s-V^n_s|$ vanishes in $L^2(\Omega, \Prob)$ at the rate $O(1/\sqrt{n})$, for every $t>0$. In particular, $I^n-V^n$ vanishes uniformly on compacts in probability as $n\rightarrow \infty$. 
\end{thm}

The proof is given in Section \ref{sect:martingale_prop}, where the martingality of $I^n-V^n$ is derived as a consequence of the infection mechanism \eqref{eq:tau_i_prob}. This result justifies the decompositions \eqref{eq:our_model_new_infections} and \eqref{eq:our_eff_reptroduction_number} in our comparison with the number of new infections and the effective reproduction number in the SIR model. Moreover, the asymptotic part of the result says that, for large enough $n$, these decompositions hold as almost sure approximations: on any compact time interval, we have
\[
    I^n_{t+h}-	I^n_{t} \approx \frac{1}{n}\sum_{i=1}^{n}\int_t^{t+h}\mathbbm{1}_{s<\tau^i}\gamma(s,C_s^n) \di \ell^{i,n}_s
\]
for large $n\geq 1$. This asymptotic observation is explored further in our companion paper \cite{Fausti-Soj_26}, where we identify the mean-field limit of the front and the empirical measure flow associated to the susceptible particles. Both Theorems \ref{Thm:Brownian_transformation} and \ref{Thm:convergence_Barnes} play a critical role in this analysis.

\subsection{Illustration of the epidemic model}\label{subsect:illustration_of_model}

Our framework can be used to simulate possible evolutions of a spreading epidemic and examine their likelihood. Moreover, such simulations can allow one to analyse the effect of various interventions, and they can provide an estimate of the expected evolution.

The output of the model will be a function of the following four inputs, which one will need to either estimate or take an informed view on:
\begin{enumerate}[i.]
    \item the population heterogeneity as characterised by the current distribution of the individual levels of shielding (i.e., the initial conditions $X_0^{i,n}$, or the current positions if restarting the system at some later time),
    \item the rate at which the front advances per new infection (i.e.,~the proportionality constant $\alpha$ and the infection-to-recovery kernel $\varrho$),
    \item the population dynamics (i.e., the drift and diffusion coefficients $b$ and $\sigma$), and
    \item the disease's transmissibility as a function of the current index of contagiousness (i.e., the effective rate of infection $\gamma$ as a function of $C^n$). 
\end{enumerate}

To illustrate the epidemic model, we will show a few examples of how the particle system may evolve for certain choices of the aforementioned inputs (chosen primarily for illustrative purposes, but also with a view to practically relevant values). In Figures~\ref{fig:10_part} and \ref{fig:20_part_gamma_effect}, particles diffuse on $[A^n_t, \infty)$ for $t \in [0,T]=[0, 100]$, with no drift ($b = 0$) and constant diffusion coefficient ($\sigma = 0.25$).  They are initially distributed according to a mixture of two (in Figure~\ref{fig:10_part}) and three (in Figure~\ref{fig:20_part_gamma_effect}) disjoint uniform densities, representing the individuals' initial levels of shielding $X_0^{i,n}$ (representing two or three initially disjoint, but interacting, groups). The moving boundary $A^n_t$ begins at $A^n_0 = a_0 = 0$, and we have taken the kernel $\varrho$ to be given by a Weibull density supported on $[0,14]$. We model the effective rate of infection $\gamma$ as $\gamma(t, C^n_t) = \gamma_0 + k_1 \tanh(k_2 C^n_t)$. Throughout, the parameter $\alpha$ is simply taken so that infection of the whole population is theoretically possible (that is, we consider the particle that reaches the highest level of shielding in the system, and set $\alpha$ so that $a_0 + \alpha$ exceeds this level).

\begin{figure}[h]
	\centering
	\includegraphics[width=0.33\linewidth]{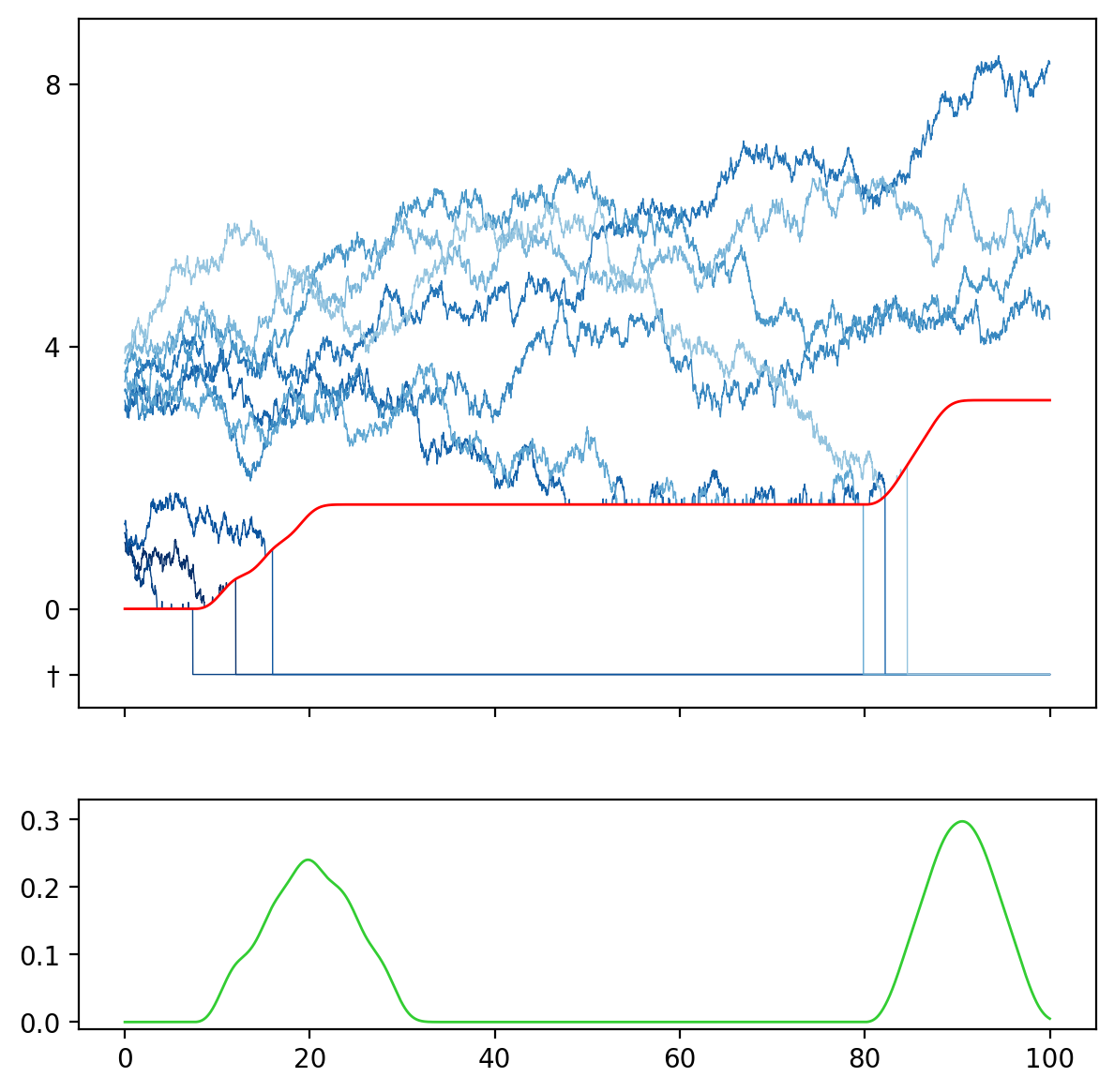}
    \includegraphics[width=0.315\linewidth]{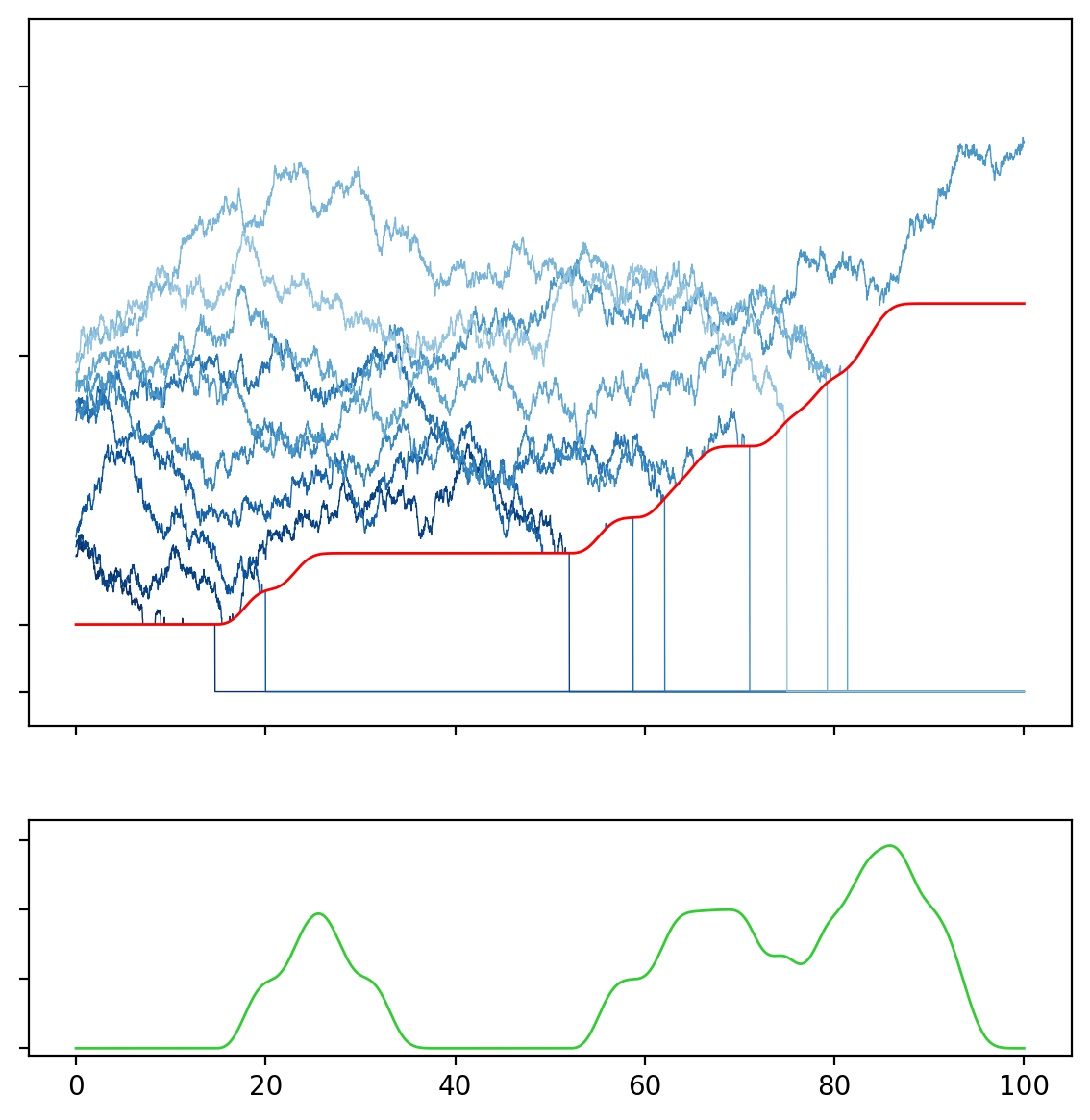}
	\includegraphics[width=0.315\linewidth]{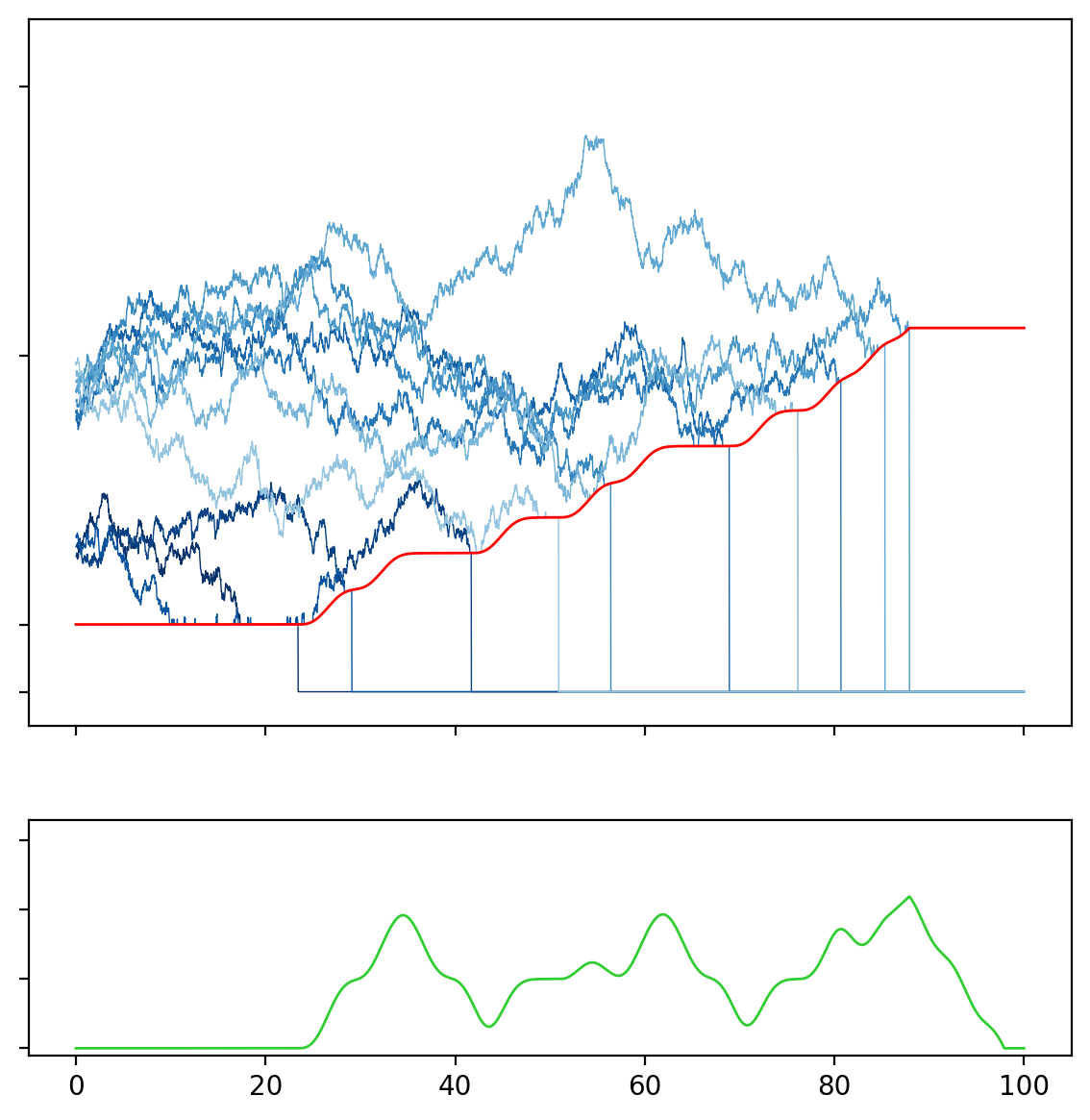}
	\caption{Dynamics of the levels of shielding in a population of 10 individuals, initially distributed according to a mixture of two disjoint uniform densities, representing low and high initial shielding. The vertical axis is the level shielding. The horizontal axis is time. The red curve shows the advancing front of the epidemic $A^n$. The green curve shows the current index of contagiousness $C^n$. Across the three plots, parameters are fixed while realisations of the $X^i$'s change.}
	\label{fig:10_part}
\end{figure}

\begin{figure}[h]
	\centering
	\includegraphics[width=0.33\linewidth]{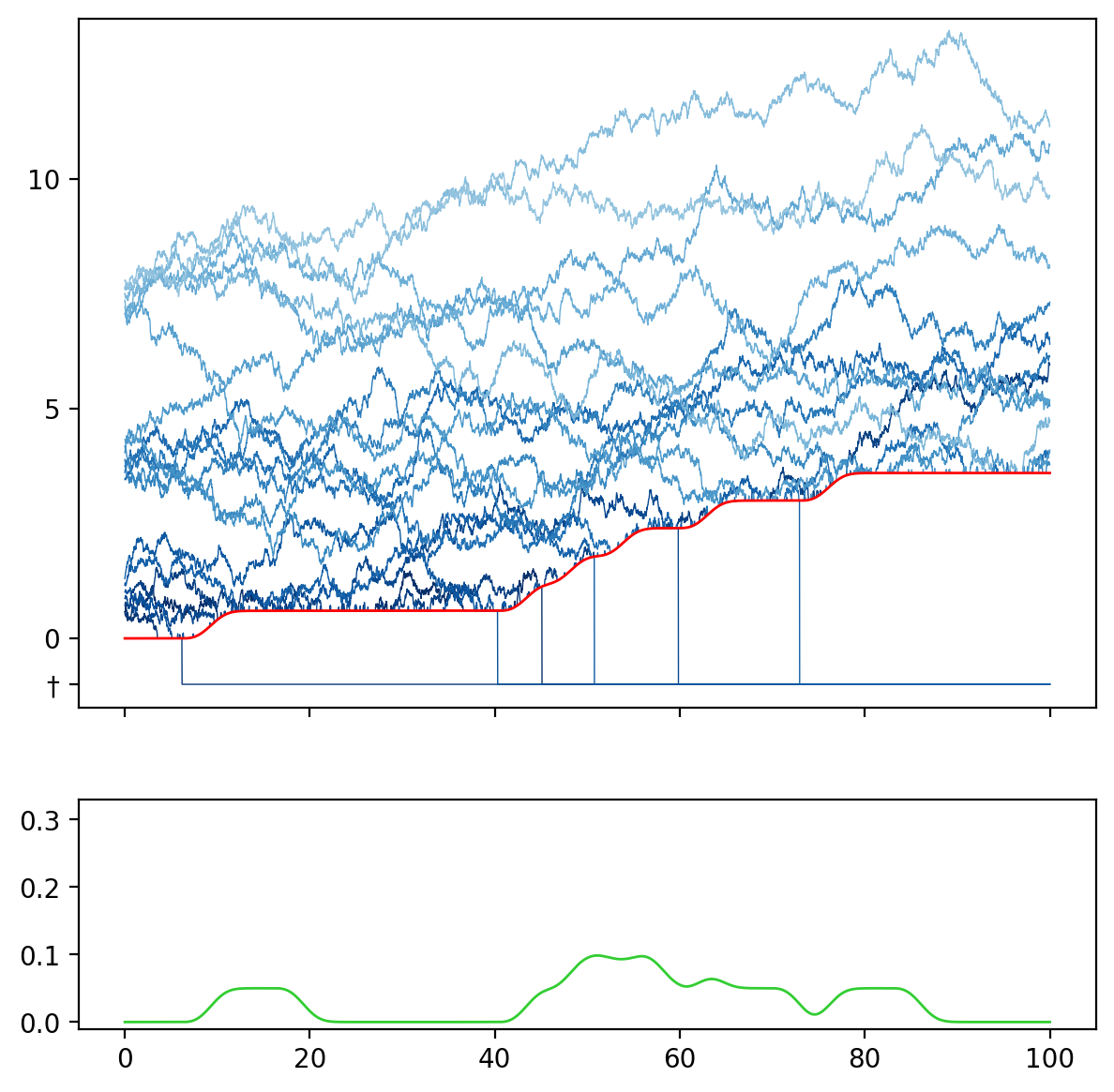}
    \includegraphics[width=0.315\linewidth]{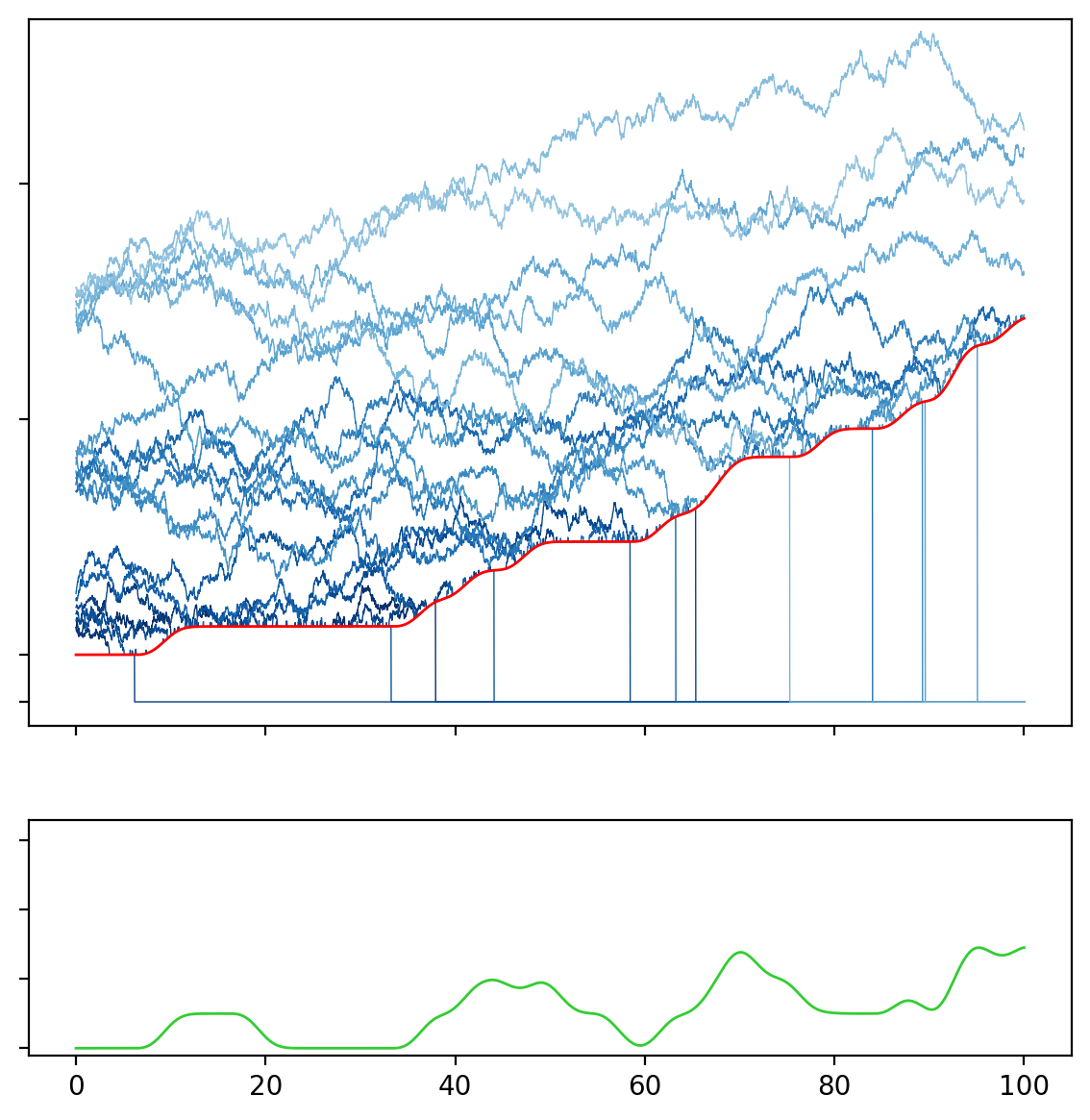}
	\includegraphics[width=0.315\linewidth]{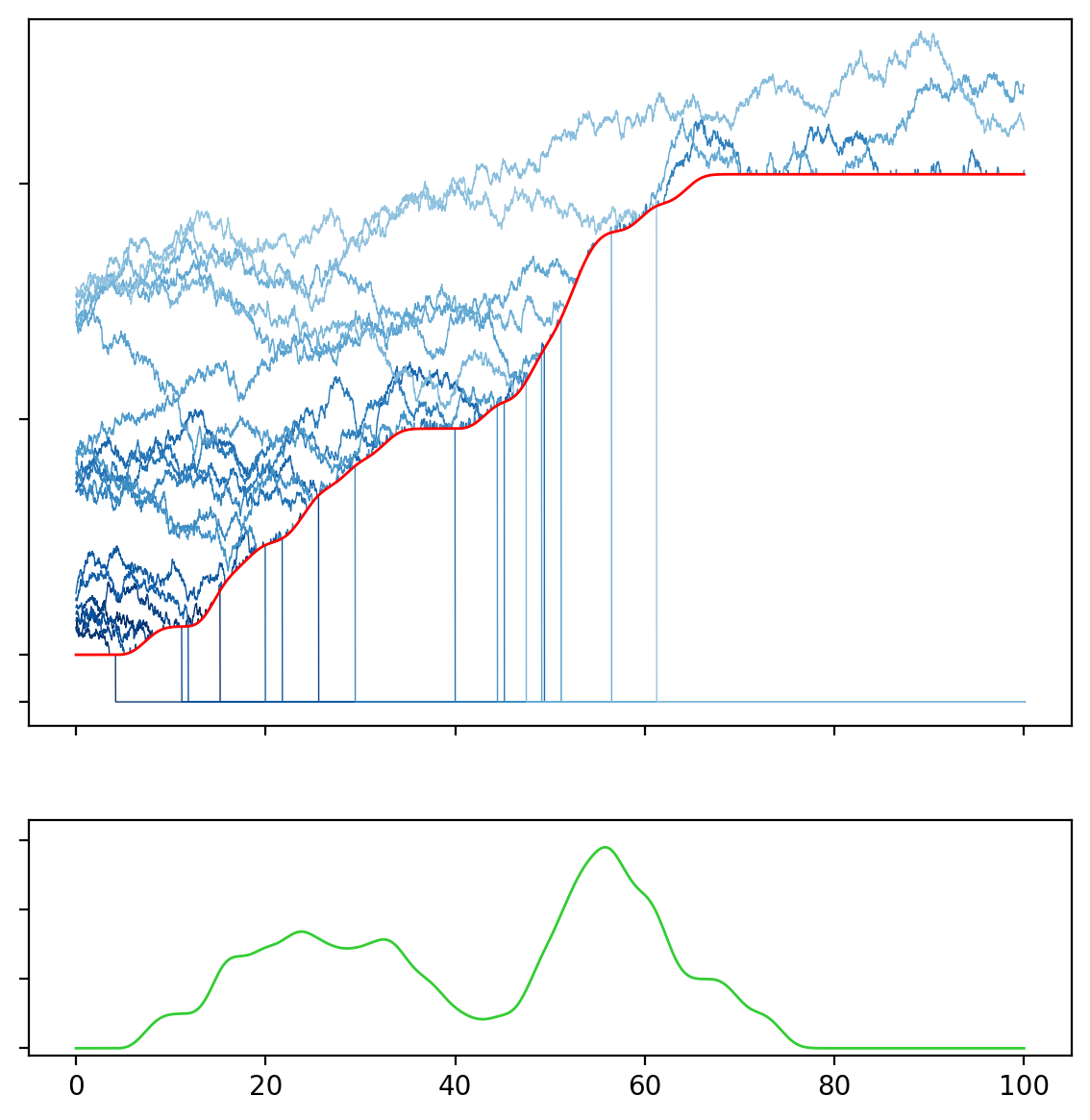}
	\caption{Dynamics of the levels of shielding in a population of 20 individuals, initially distributed according to a mixture of three disjoint uniform densities, representing low, medium and high initial shielding. Across the three plots, realisations of the $X^i$'s are fixed, while the effective rate of infection $\gamma$ is taken to be progressively more aggressive.}
	\label{fig:20_part_gamma_effect}
\end{figure}

Figure~\ref{fig:10_part} illustrates the behaviour of a system with only 10 individuals: across the three plots, all parameters are fixed, with only the realized particles' paths varying. In the first plot, the particles with the highest initial shielding level tend to diffuse away from the moving boundary (in red in the plots), and they are still safe from infection by the time the simulation ends. In contrast, the second and third plots exhibit a downward trend in the particle diffusions, leading to the infection of most, if not all, particles within the simulated time frame. Below each graph of the particle system, we plot $C^n$, the current index of contagiousness, which, recalling \eqref{eq:currently_contagious}, is concerned with the number of new infections within the previous 14 units of time (the support of $\varrho$).

In Figure~\ref{fig:20_part_gamma_effect} we plot the dynamics of the shielding levels of a population of 20 individuals. As opposed to Figure~\ref{fig:10_part}, for these simulations we fix the particle realizations and vary the parameters across the three plots: in particular, we vary $\gamma_0, k_1$ and $k_2$ and make the effective rate of infection $\gamma$ progressively more aggressive. We can observe the effect that this has on the system directly: a relatively mild $\gamma$ (as in the first plot) leads to intervals of time (for $t \in [10,40]$ and when $t\ge 70$) in which there are no new infections in the system, despite many \emph{at-risk} particles accumulating local time as they are being reflected off the boundary. A slight increase in $\gamma$ results in elevated infection frequencies starting around $t=40$ (see second plot in Figure~\ref{fig:20_part_gamma_effect}). A further increase in $\gamma$ in the third graph sees most of the population being infected by $t=70$, although some individuals still remain ``safe'' until the end of the simulation, despite having been at-risk, or at least close to the advancing epidemic front, for a stretch of time by then.

Once the general behaviour of the system is understood, one can start considering how to introduce and model external interventions. For example, with reference to the recent COVID-19 pandemic, it could be interesting to model governmental measures such as the introduction of protective equipment, immunization through vaccination, social distancing and even lock-downs. These should be reflected in the parameters, with $b$ and $\sigma$ made to vary based on the changing behaviour of the individuals, and $\gamma$ encoding the effect of these measures on the effective contagiousness of the disease. Naturally, one could also consider abrupt interventions whereby one changes the current particle positions and simulates from new initial conditions, but this is somewhat less interesting than considering gradual interventions in the model parameters in ways that respects the assumptions for the system's well-posedness.

In Figure~\ref{fig:200_part_mean_reversion}, we demonstrate a simple approach to modelling intervention in the particle dynamics. We introduce mean reversion by adding a drift term $b(t, X^{i,n}_t, X^{i,n}_0) = \theta(b_1(t, X^{i,n}_0) - X^{i,n}_t)$ to the dynamics of each particle $X^i$, where $\theta = 0.075$ controls the mean-reversion strength, and the target level $b_1(t, X^{i,n}_0)$ varies with time and depends on each particle's initial position. One should think of $b_1(t, x)$ as the level of shielding that a particle with initial shielding $x$ will tend to gravitate around at time $t$ (either for internal reasons or because of external forces seeking to achieve this). For concreteness, consider an intervention in $b_1$ implemented through a time-dependent target level of the form
$$
b_1(t, X^{i,n}_0) =
\begin{cases}
    X_0^{i,n},  & t \le t_0, \\
    X_0^{i,n} \frac{t_1 - t}{t_1 - t_0} + (X_0^{i,n} + 2) \frac{t - t_0}{t_1 - t_0}, & t \in (t_0, t_1), \\
    X_0^{i,n} + 2,  & t \ge t_1.
\end{cases}
$$
This formulation captures three distinct phases: (i) for $t \le t_0$, particles revert toward their initial positions $X_0^{i,n}$; (ii) during the intervention period $(t_0, t_1)$, the target level shifts linearly from $X_0^{i,n}$ to $X_0^{i,n} + 2$, effectively pushing particles away from the contagious front; and (iii) for $t \ge t_1$, particles maintain mean reversion toward the elevated level $X_0^{i,n} + 2$. Note that this framework falls within our theoretical analysis, as the drift is consistent with Assumption~\ref{assump:coefficient_assumptions}.

In the three plots of Figure~\ref{fig:200_part_mean_reversion} we show the evolution of a system of 200 particles. In the first plots the mean reversion level does not increase ($t_0 = \infty)$ and the particles keep reverting around their initial level: this results in two steep waves of infection. In the second and third plot, intervention happens, respectively, at time $t_0 = 50$ and $t_0 = 40$, with $t_1 = t_0 + 10$. In the second plot, the second wave of infections still happens, but it is less severe than in the first plot. In the third plot, due to earlier intervention, the second wave is almost completely avoided.

\begin{figure}[ht]
	\centering
	\includegraphics[width=0.33\linewidth]{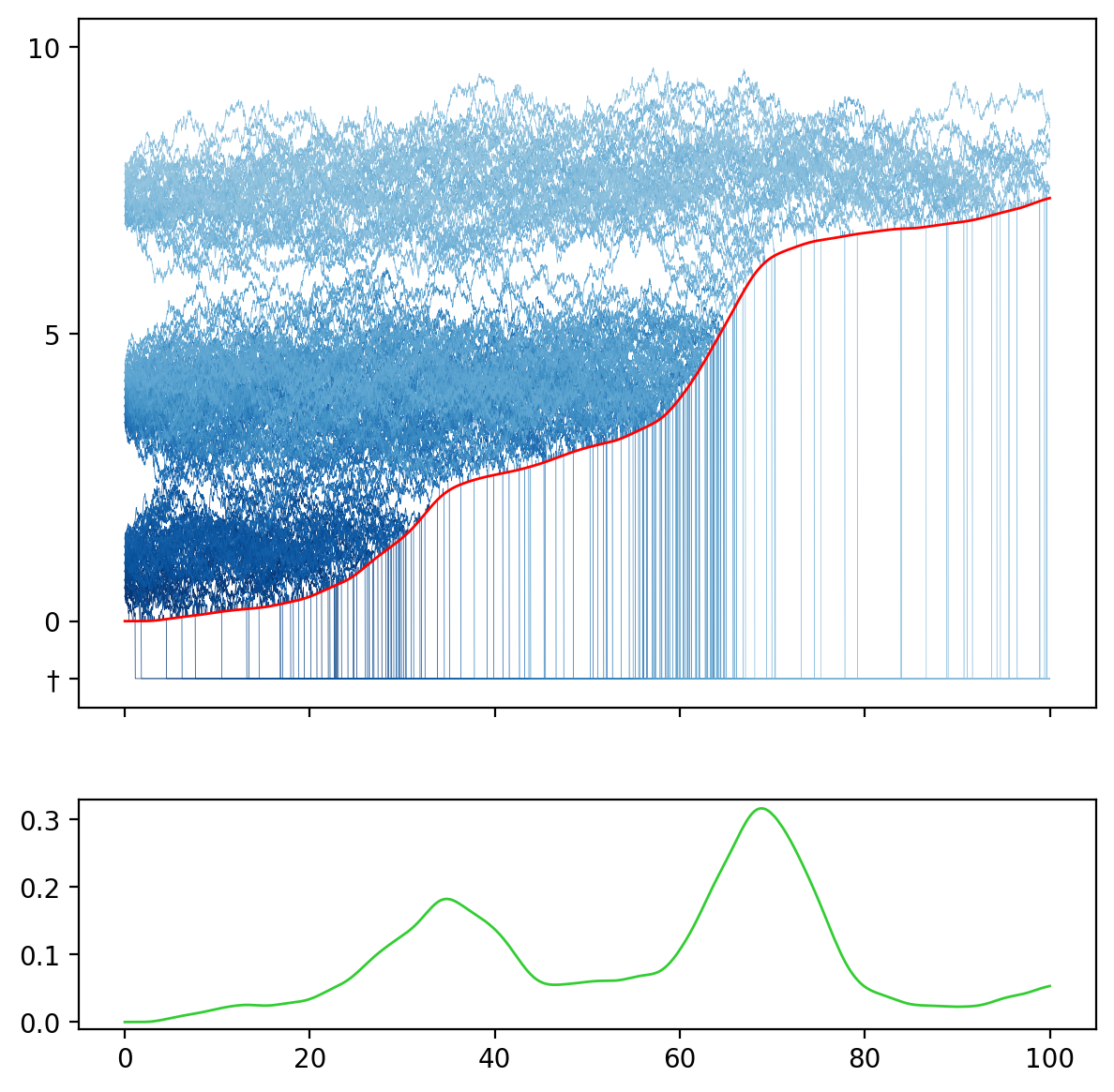}
    \includegraphics[width=0.315\linewidth]{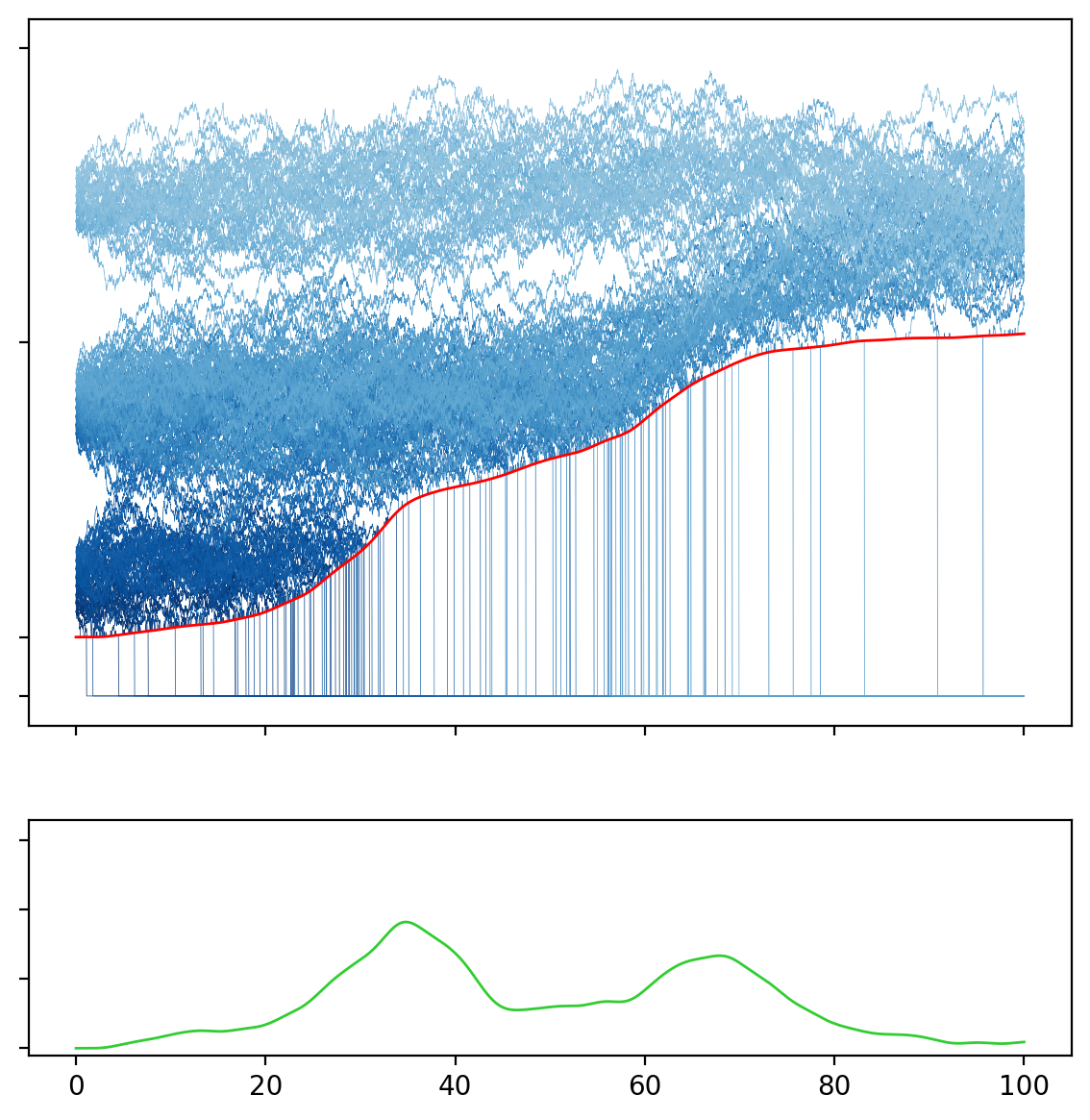}
	\includegraphics[width=0.315\linewidth]{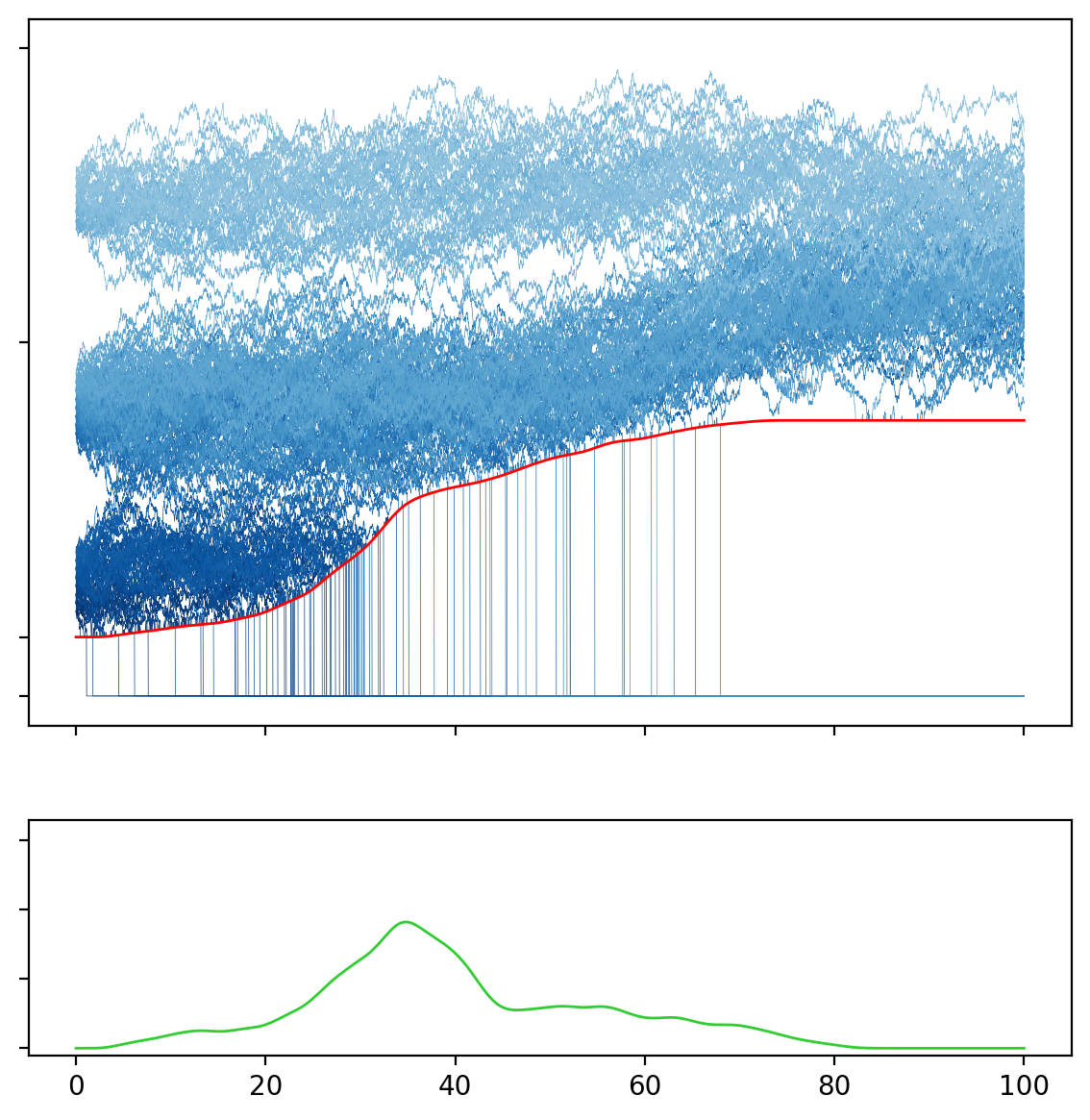}
	\caption{Dynamics of the levels of shielding in a population of 200 individuals, initially distributed according to a piecewise function with three blocks, representing low, medium and high shielding. Each of the $X^i$'s is mean reverting around its initial level $X^i_0$. In the first picture we see a first wave of infection followed by a steep second wave. In the second and third plot, at time 50 and 40 respectively, we adjust the drift dynamics of the individuals in the middle block, gradually increasing their level of mean reversion by 2. The earlier intervention significantly reduces infections.}
	\label{fig:200_part_mean_reversion}
\end{figure}

We conclude this section by discussing some interesting directions for future work in relation to the epidemic modelling. Firstly, as it is presented here, our model is intended for the short or medium term, meaning a period for which it is acceptable to perform predictions without worrying about re-infection of previously infected individuals. If one is interested in a longer term model, then this could be addressed by allowing for re-insertion of infected individuals at some later point in time when their immunity has waned. Secondly, as discussed above, it could be interesting to study interventions within the system in a disciplined way. One may simply take a simulations based approach to this, similarly to what we did in Figure~\ref{fig:200_part_mean_reversion}. Alternatively, if one wants to identify optimal strategies given certain targets and the corresponding costs, it may be possible to analyse this via the formulation of suitable stochastic control problems based on our framework. Finally, in practical applications concerned with a human epidemic, one would also need to consider outcomes such as death or hospitalisation, but here we take the point of view that the number of newly infected individuals is ultimately the key quantity to track, and one can then have a separate rule (or model) for the extent to which this leads to hospitalisations or deaths.

\subsection{Brownian particles impinging on an inert object}\label{sec:Barnes}

As discussed in the introduction, the moving boundary problems for reflected Brownian motion studied in \cite{Knight} and \cite{barnes} turn out to be closely related to our framework in a way that we now describe. To this end, consider the simplified front $A^n_t:= \int_0^t I_s^n \ds$ and replace the current contagiousness by $C^n_t:= I^n_t$. Set also $b\equiv 0$ and $\sigma \equiv 1$. Our arguments continue to apply in this setting. The second part of Theorem \ref{Thm:convergence_Barnes} then suggests that, as $n$ becomes large, the particle system will behave in approximately the same way as the following system
\begin{equation}\label{eq:sys_X_simplified}
    \left\{ \begin{array}{@{}l@{}l}
        \di  Y_t^{i,n} = \di B_t^{i,n} + \tfrac{1}{2} \di \ell_t^{i,n},\;\; t < \tau^i,\quad Y_t^{i,n}=\dagger, & \quad t \geq \tau^i, \vspace{6pt}\\
        \displaystyle \tfrac{\di }{\dt}S_t^{n} = \frac{1 }{n}\sum_{i=1}^n\int_0^t
        \mathbbm{1}_{s<\tau^i}\gamma\bigl(s, \tfrac{\di }{\ds}S_s^{n}\bigr)\di \ell_s^{i,n}, & \quad t\geq 0, 
    \end{array} \right.
\end{equation}
where $\ell^{i,n}$ is the local time of $Y^{i,n}$ along $S^n$ (with $S^n_0=0$), and where $\tau^i=\inf\{t>0:Y_t^{i,n}=\dagger\}$ is characterised as in \eqref{eq:tau_i_prob} but for the system \eqref{eq:sys_X_simplified}, informally meaning that we have
\begin{equation}\label{eq:elastic_barnes}
    \Prob\bigl(\tau^i \in(t,t+h] \mid Y^{i,n}, S^n,\, \{t<\tau^i\} \bigr) \approx \int_t^{t+h}
    \gamma\bigl(s, \tfrac{\di }{\ds}S_s^{n}\bigr)\di \ell_s^{i,n}
\end{equation}
as $h\searrow 0$. That is, the probability of absorption for the $i$'th particle is given by its contribution to the velocity of the moving boundary $S^n$.

The above should be contrasted with the particle system studied in \cite{barnes} which takes the form
\begin{equation}\label{eq:sys_X_Barnes}
    \left\{ \begin{array}{@{}l@{}l}
        \di  X_t^{i,n} = \di B_t^{i,n} + \tfrac{1}{2} \di \ell_t^{i,n}, & \quad t \geq0, \vspace{6pt}\\
        \displaystyle \tfrac{\di }{\dt}S_t^{n} = - \frac{\gamma }{n}\sum_{i=1}^n \hspace{1pt} \ell^{i,n}_t , & \quad t\geq 0, 
    \end{array} \right.
\end{equation}
for a given constant $\gamma>0$, where $\ell^{i,n}$ is the local time of $X^{i,n}$ along $S^n$ with $X^{i,n}_t\geq S^n_t$. The boundary $S^n$ describes the path of an inert object that obeys Newton's law of motion. It cannot be crossed by the Brownian particles, but is free to move, and, whenever a Brownian particle collides with it, there is a total transfer of momentum proportional to the local time $\ell^{i,n}$.

As discussed in \cite{Knight}, $S^n$ in \eqref{eq:sys_X_Barnes} can be thought of as an ideal mobile heat insulator if the $X^{i,n}$ are seen as `heat particles'. The setting of \eqref{eq:sys_X_simplified} is analogous to \eqref{eq:sys_X_Barnes} except that $S^n$ is now an \emph{imperfect} mobile heat insulator, where the heat loss is governed by \eqref{eq:elastic_barnes}. Moreover, the coefficient $\gamma$, which is constant in \eqref{eq:sys_X_Barnes}, may now depend on time and the velocity of $S^n$, and we stress that it determines both the transfer of momentum to $S^n$ and the heat loss at the boundary.

\section{Well-posedness of the particle system}\label{sec:existence_sys}

In this section, we construct the system $\mathbf{X}^n$ satisfying the formulation in Theorem \ref{thm:existence}. The main technical difficulties lie in the coupling of the infection times through the current index of contagiousness and the local time of each particle's collisions with the front. In this regard, the first key realisation is that \eqref{eq:tau_i_prob} is the correct specification of the conditional laws of the infection times to have the desired behaviour and avoid ill-posed circularities. From there, a piece-by-piece construction of the particle trajectories that goes via the construction of the auxiliary systems $\mathbf{\hat{X}}^{n,(-i)}$ breaks down the coupling of the infection times and gives the desired conditional laws.

The core idea is to partition $[0,\infty)$ into $n+1$ random time intervals, separated by the sequential infection times of the population. Within each interval, the boundary's movement and the rate of infection are fully determined by the history of the system up to that point. Through this iterative concatenation, we first establish a foundational \textit{globally reflected system} $\mathbf{\hat{X}}^n$, wherein all particles continue to reflect indefinitely after infection with the following crucial property: each particle has its own moving boundary that is advanced by the infections of the rest of the population, but remains blind to the particle's own infection. This system provides a pathwise basis for the reflected dynamics in our subsequent construction of the particle system.

By moving particles to the cemetery state at their respective infection times, we obtain the global trajectories of the particles in the \textit{true system} $\mathbf{X}^n$. To verify the conditional law \eqref{eq:tau_i_prob} of the infection times, we must isolate the dynamics of a `tagged' particle from the boundary advance caused by its own potential infection. To achieve this, for each $i \in \{1,\dots,n\}$, we supplement the construction of $\mathbf{X}^n$ with the \textit{auxiliary system} $\mathbf{\hat{X}}^{n, (-i)}$. In this system, the $i^{\text{th}}$ particle is effectively `immune': it continues to reflect off the moving boundary globally in time and exerts no influence on the boundary dynamics or the infection times of the other particles.

The recursive construction of these intermediate components, and their piecewise concatenation, becomes a little technical and the notation is quite heavy, so we defer it to Appendix \ref{appendix:particle_system}. In what follows, we present the central definitions and results that give us Theorem \ref{thm:existence}.

\subsection{Proof of Theorem \ref{thm:existence}}\label{sect:well_posed_particle_system}

We first confirm the well-posedness of the globally reflected system $\mathbf{\hat{X}}^n$, as discussed above.

\begin{prop}[Globally reflected trajectories]\label{prop:evol_x_hat}
The processes $\mathbf{\hat X}_t^n = (\hat X_t^{1,n}, \dots,\hat X_t^{n,n}) \in [\mathbf{A}^n_t, \boldsymbol{\infty})$ and $\mathbf{A}^n_t = (A_t^{1,n}, \dots, A_t^{n,n}) \in \R_{\ge 0}^n $ constructed according to \eqref{eq:general_P_defn_by_parts} in Definition \ref{defn:piecewise_process} uniquely satisfy the dynamics
\begin{equation}\label{eq:sys_X_hat}
    \left\{ \begin{array}{@{}l@{}l}
    \di \hat{X}_t^{i,n} = b(t, \hat{X}_t^{i,n}) \dt + \sigma (t, \hat{X}_t^{i,n}) \di B^i_t + \tfrac{1}{2} \di \ell_t^{A^{i,n}}(\hat{X}^{i,n}), & \quad t \ge 0, \, i = 1, \dots, n, \\ \vspace{2pt}
    A_t^{i,n} = a_0 + \alpha\int_{0}^t \varrho(t-s) I^{i,n}_{s} \ds, & \quad t \ge 0,  \\ \vspace{2pt}
    I_t^{i,n} = \frac{1}{n} \sum_{j=1}^{n-1} \mathbbm{1}_{[0,t]}(\xi^i_j), & \quad t \ge 0,
    \end{array} \right.
\end{equation}
for $\{\xi_j^i\}_{j=1}^{n-1}$ given by \eqref{eq:xi_i_true}, where $\ell^{A^{i,n}}(\hat{X}^{i,n})$ is the local time of $\hat{X}^{i,n}$ along $A^{i,n}$ in accordance with the definition \eqref{eq:local_time_defn}.
\end{prop}

\begin{proof}
    See Appendix~\ref{app:concatenation_proofs}.
\end{proof}

While $\mathbf{\hat{X}}^n$ will serve as the pathwise foundation from which we construct the true particle system, it is insufficient for pinning down the desired conditional laws of the infection times. Nevertheless, as the next result confirms, it already gives us the particle trajectories with associated killing times that satisfy the dynamics specified in \eqref{eq:sys_X}. By then also constructing the auxiliary systems $\mathbf{\hat{X}}^{n, (-i)}$ in agreement with the construction of $\mathbf{\hat{X}}^n$ (Proposition \ref{prop:artificial_tagged_i}), we will finally be able to conclude that we have a unique solution in the full sense of Theorem \ref{thm:existence} (Proposition \ref{prop:distribution_tau_i}). 

\begin{prop}[True trajectories with infection] \label{prop:evol_x}
    Let the processes $ \mathbf{X}_t^n = (X_t^{1,n}, \dots, X_t^{n,n})$ be given by \eqref{eq:general_P_defn_by_parts} in Definition \ref{defn:piecewise_process}, and define $\tau^i := \inf \{ t \ge 0 \, : \, X_t^{i,n} = \dagger  \}$ for all $i = 1, \dots, n$. Let $  I^n_t := \frac{1}{n}\sum_{j=1}^{n} \mathbbm{1}_{[0,t]}(\tau^j)$ and $A^n_t := a_0 + \alpha \int_0^t \varrho(t-s) I^n_s \ds$ for all $t\geq 0$. Then the triple $(\mathbf{X}^n,I^n,A^n)$ satisfies the system of equations \eqref{eq:sys_X} from Theorem \ref{thm:existence}.
\end{prop}

\begin{proof}
    For $i \in \{1,\dots,n\}$, the intermediate construction in Appendix \ref{app:intermediate_systems} guarantees that, pathwise, $X^{i,n}_s = \hat{X}^{i,n}_s$ for $s \in [0,\tau^i)$, where $\hat{X}^{i,n}$ satisfies \eqref{eq:sys_X_hat}. Conversely, $X^{i,n}_s = \dagger$ for all $s \ge \tau^i$. 
    
    By construction, $\tau^i = \varsigma^{(k)}$ for a unique random index $k \in \{ 1, \dots, n \}$. On the event $\{ t<\tau^i \}$, we have $\xi^i_j = \varsigma^{(j)}$ for $j = 1,\dots, k-1$ (with each $\varsigma^{(j)}$ equal to $\tau^{m}$ for some unique $m \neq i$), while $\xi^i_j > t$ for $j=k,\ldots,n$. Consequently, for $t \in [0, \tau^i)$, we have the pathwise equivalence $I_t^n = \frac{1}{n}\sum_{j=1}^{k-1} \mathbbm{1}_{[0,t]}(\xi_j^i) = I_t^{i,n}$. 
    
    It follows immediately that $A^{i,n}_t = A^{n}_t$ pathwise on $[0,\tau^i)$. Therefore, prior to infection, the local time $\ell_t^{A^{i,n}}(\hat{X}^{i,n})$ from the globally reflected system \eqref{eq:sys_X_hat} agrees pathwise with the local time of $X^{i,n}$ along  $A^{n}$ which confirms the dynamics \eqref{eq:sys_X}.
\end{proof}

To address the second part of Theorem \ref{thm:existence}, we construct, in addition to $\hat{\mathbf{X}}^n$, the auxiliary systems $\hat{\mathbf{X}}^{n, (-i)}$ where the infection of the tagged particle is dismissed entirely.

\begin{prop}[Auxiliary system dismissing infection of a tagged particle]\label{prop:artificial_tagged_i}
For each $i=1,\ldots ,n$, let the processes $ \mathbf{\hat X}_t^{n, (-i)} = (\hat X_t^{1,n, (-i)}, \dots, \hat X_t^{n,n, (-i)})$ in $(\mathbb{R} \cup \{\dagger \})^n$ be given by \eqref{eq:general_P_defn_by_parts_minus_i} of Definition~\ref{defn:piecewise_process}, and define the times $\tau^{j,(-i)} := \inf \{ t \ge 0 \, : \, \hat X_t^{j,n, (-i)} = \dagger  \}$ for all $j \neq i$. Then the triple $(\mathbf{\hat X}^{n,(-i)}, \hat I^{n, (-i)},\hat A^{n, (-i)})$ is adapted to the reduced-information filtration $\hat{\mathcal{F}}^{i,n}_t$ defined in \eqref{eq:reduced_filtration} and satisfies the system of equations
\begin{equation}\label{eq:sys_X_minus_i}
    \left\{ \begin{array}{@{}l@{}l}
    \di \hat X_t^{j,n, (-i)} = b(t, \hat X_t^{j,n, (-i)}) \dt + \sigma (t, \hat X_t^{j,n, (-i)}) \di B^j_t + \tfrac{1}{2}\di \hat \ell_t^{j,n,(-i)}, & \quad t \in [0, \tau^{j,(-i)}), \, j \neq i, \vspace{3pt}\\
    \di \hat X_t^{i,n, (-i)} = b(t, \hat X_t^{i,n, (-i)}) \dt + \sigma (t, \hat X_t^{i,n, (-i)}) \di B^i_t + \tfrac{1}{2}\di \hat \ell_t^{i,n}, & \quad t \ge 0, \vspace{3pt}\\
    \hat A^{n, (-i)}_t := a_0 + \alpha \int_0^t \varrho(t-s)  \hat I^{n, (-i)}_s \ds, & \quad t \ge 0, \vspace{4pt}\\ 
    \hat I^{n, (-i)}_t := \frac{1}{n}\sum_{j \neq i} \mathbbm{1}_{[0,t]}(\tau^{j,(-i)}), & \quad t \ge 0, \vspace{4pt}\\ 
    \hat X_{t}^{j,n,(-i)} = \dagger, &\quad t\geq \tau^{j,(-i)}, \, j \neq i,
    \end{array} \right.
\end{equation}	
where $\hat \ell^{j,n,(-i)}$ denotes the local time of the $j$'th particle $\hat X^{j,n, (-i)}$ along the front $\hat A^{n, (-i)}$. We denote $\hat \ell^{i,n} := \hat \ell^{i,n,(-i)}$ to align with the notation of Theorem \ref{thm:existence}.
\end{prop}

\begin{proof}
The proof proceeds identically to Propositions \ref{prop:evol_x_hat} and \ref{prop:evol_x}, utilizing instead the sequence of intermediate systems $\mathbf{\hat X}_t^{n, (-i),(k)}$ detailed in Appendix \ref{app:dismissed_infection}.
\end{proof}

We next confirm that the infection times $\tau^i$ constructed in Proposition \ref{prop:evol_x} are as desired.

\begin{prop}\label{prop:distribution_tau_i}
The infection times $\tau^i = \inf \{t \ge 0 \, : \, X^{i,n}_t = \dagger \}$ from Proposition \ref{prop:evol_x} satisfy
\begin{equation*}
    \Prob (\tau^i \le t \mid \hat{\mathcal{F}}^{i,n}_t) = 1 - \exp \left\{ - \int_0^t \gamma(s, \hat{C}^{n,(-i)}_s) \di \hat{\ell}_s^{i,n} \right\}, \quad \forall t \ge 0,
\end{equation*}
for $i=1,\ldots,n$, where $\hat{\mathcal{F}}^{i,n}_t $ is defined in \eqref{eq:reduced_filtration}, $\hat \ell^{i,n}$ is the local time of the process $\hat X^{i,n, (-i)}$ along the boundary $\hat A^{n, (-i)}$ from Proposition \ref{prop:artificial_tagged_i}, and $\hat{C}^{n,(-i)}_t  = \int_0^t \varrho(t-s) (\hat I_s^{n, (-i)} - \hat I_{s - \bar{d}}^{n, (-i)} ) \ds$.
\end{prop}

\begin{proof}
Fix an arbitrary index $i\in \{1,\ldots,n\}$, and recall the random times $\xi^{(-i),(k)}$, for $k = 1, \dots, n-1$ defined in the recursive construction of Appendix \ref{app:dismissed_infection}.	
We divide the probability that the $i^{\text{th}}$ particle has not yet been infected at time $t$ according to these random times, namely
\begin{align}\label{eq:i_alive_total_prob}
    \Prob (\tau^i > t \mid \hat{\mathcal{F}}^{i,n}_t )
    &=
    \sum_{k = 1}^{n-1} \Prob \left( \{ \tau^i > t \} \cap \{ t \in [\xi^{(-i),(k-1)}, \xi^{(-i),(k)}) \} \, | \, \hat{\mathcal{F}}^{i,n}_t  \right) \nonumber \\
    &\quad +
    \Prob \left( \{ \tau^i > t \} \cap \{ t \ge \xi^{(-i),(n-1)} \} \, | \, \hat{\mathcal{F}}^{i,n}_t  \right).
\end{align}
Now fix a time $t\geq0$ and consider the event $E_k : = \{ \tau^i > t \} \cap \{ t \in [\xi^{(-i),(k-1)}, \xi^{(-i),(k)}) \}$, for any given $k\in \{1,\ldots,n-1\}$. On this event, we have that particle $i$ is not yet infected, and we have that an exponential clock has rung in precisely the first $k-1$ intermediate systems in the construction of $ \mathbf{\hat X}_t^{n, (-i)}$ from Proposition \ref{prop:artificial_tagged_i}. Observe that the $i^{\text{th}}$ particle not being infected at such a time means that, in the construction of the true particle system $\mathbf{X}_t^n$ in Appendix~\ref{app:intermediate_systems}, particle $i$ has not been moved to state $\dagger$ at any of the first $j=0,\ldots,k-2$ infection times $\tilde{\varsigma}^{(j)}$ (defined in \eqref{eq:potential_infection_times_k_step}). More precisely, we can write the event $E_k$ equivalently as
\begin{equation}\label{eq:split_non-infection}
    \bigcap_{j=1}^{k-1} \left\{ X^{i,n,(j)}_{\tilde\varsigma^{(j)}} \neq \dagger \right\}
    \bigcap \left\{X^{i,n,(k)}_{t - \varsigma^{(k-1)}} \neq \dagger  \right\}
    \bigcap \left\{ t \in  [\xi^{(-i),(k-1)}, \xi^{(-i),(k)} ) \right\} \bigcap \left\{ \xi^{(-i),(k-1)} = \varsigma^{(k-1)}\right\},
\end{equation}
where the last event in the expression above highlights that $t-\varsigma^{(k-1)}\geq0$. In fact, on the event $E_k$, we must have $\tilde \varsigma^{(j)} = \tilde{\xi}^{(-i),(j)}  $ for all $j=1,\ldots,k-1$. For every $\omega \in E_k$, we use the construction \eqref{eq:sys_X_k} of $X^{i,n,(j)}$ and the definition \eqref{eq:infection_times_k_step_immune} of $\tilde{\xi}^{(-i),(j)} $, to get rid of the overlap between the first $k$ events in \eqref{eq:split_non-infection} and rewrite the full intersection as
\[
    E_k =  \bigcap_{j=1}^{k-1} \left\{ U^{i,(j)}_{\tilde\xi^{(-i),(j)}} 
    < \chi^{i,(j)} \right\} \cap \left\{ t \in [\xi^{(-i),(k-1)}, 
    \xi^{(-i),(k)}) \right\} \cap \left\{ U^{i,(k)}_{t - \xi^{(-i),(k-1)}} < \chi^{i,(k)} 
    \right\},
\]
with $U^{i,(j)}_s := \int_0^s \gamma^{i,(j)}(r) \di \ell_r^{A^{i,n,(j)}}
    (\hat X^{i,n,(j)})$. Recalling the details of the recursive construction, we can identify a conditionally independent structure with the event $\{ t\in [\xi^{(-i),(k-1}), \xi^{(-i),(k)}) \}$ being $\hat{\mathcal{F}}^{i,n}_t$-measurable while the other events are conditionally independent given $\hat{\mathcal{F}}^{i,n}_t$. Combining this with Lemma \ref{lemma:rate_integral} (where we recall that $\tilde \varsigma^{(j)} = \tilde{\xi}^{(-i),(j)}$ for all $j=1,\ldots,k-1$, for all $\omega \in E_k$), we therefore obtain that
\begin{align*}
    \Prob \left( E_k \, | \, \hat{\mathcal{F}}^{i,n}_t  \right)
    &=
    \prod_{j=1}^{k-1} \Prob \left(  \int_{ \xi^{(-i),(j-1)} }^{ \xi^{(-i),(j)} } \!\! \gamma(s, \hat{C}^{n,(-i)}_s) \di \hat{\ell}_s^{i,n} < \chi^{i, (j)}  \, | \, \hat{\mathcal{F}}^{i,n}_t  \right) \\
    &\quad \cdot \Prob \left( \int_{\xi^{(-i),(k-1)}}^{t} \!\! \gamma(s, \hat{C}^{n,(-i)}_s) \di \hat{\ell}_s^{i,n}< \chi^{i, (k)}  \, | \, \hat{\mathcal{F}}^{i,n}_t  \right)  \mathbbm{1}_{\left\{ t \in  [\xi^{(-i),(k-1)}, \xi^{(-i),(k)}) \right\}},
\end{align*}
for $\hat{C}^{n,(-i)}$ and $\hat{\ell}^{i,n}$ as in the statement of the proposition. By construction, the local time $\hat{\ell}^{i,n} =\hat \ell^{i,n,(-i)}$ and the current index of contagiousness  $\hat{C}^{n,(-i)}$ are both adapted processes for the filtration $(\hat{\mathcal{F}}^{i,n}_t)_{t\geq0}$. Likewise, for any $s\leq t$, the event $\{   \xi^{(-i),(k-1)} \leq s \}$ and the events $ \{ \xi^{(-i),(j-1)} \leq s < \xi^{(-i),(j)} \}$, for $j = 1, \dots, k-1$, are elements of $\hat{\mathcal{F}}^{i,n}_t$. Since the exponential random variables $\chi^{i,(j)}$, for $j = 1, \dots, k$, were chosen to be independent from the inputs generating $\hat{\mathcal{F}}^{i,n}_t$, we thus get
\begin{align*}
    \Prob \left( E_k \, | \, \hat{\mathcal{F}}^{i,n}_t  \right)
    &= \prod_{j=1}^{k-1} \exp \left\{ - \int_{ \xi^{(-i),(j-1)}  }^{ \xi^{(-i),(j)}  } \hspace{-10pt} \gamma(s, \hat{C}^{n,(-i)}_s) \di \hat{\ell}_s^{i,n} \right\}\\
    &\qquad \qquad\cdot 
    \exp \left\{ - \int_{ \xi^{(-i),(k-1)}  }^{t} \hspace{-12pt} \gamma(s, \hat{C}^{n,(-i)}_s) \di \hat{\ell}_s^{i,n} \right\} \mathbbm{1}_{\left\{ t \in  [\xi^{(-i),(k-1)} , \xi^{(-i),(k)} ) \right\}} 
    \\
    &= \exp \left\{ - \int_{0}^{t} \gamma(s, \hat{C}^{n,(-i)}_s) \di \hat{\ell}_s^{i,n}  \right\} \mathbbm{1}_{\left\{ t \in  [\xi^{(-i),(k-1)} , \xi^{(-i),(k)} ) \right\}},
\end{align*}
for $k=1,\ldots,n-1$. Using this expression in \eqref{eq:i_alive_total_prob}, and noting that an analogous argument applies to the final event $E_{n} = \{ \tau^i > t \} \cap \{ t \ge \xi^{(-i),(n-1)} \}$, we finally get
\begin{equation*}
    \Prob (\tau^i > t \, | \, \hat{\mathcal{F}}^{i,n}_t ) =
    \exp \left\{ - \int_{0}^{t} \gamma(s, \hat{C}^{n,(-i)}_s) \di \hat{\ell}_s^{i,n} \right\}.
\end{equation*}
which completes the proof.
\end{proof}

The above results combine to give us Theorem \ref{thm:existence}.

\begin{proof}[Proof of Theorem~\ref{thm:existence}]
The triple $(\mathbf{X}^n, I^n, A^n)$ from Proposition~\ref{prop:evol_x} satisfies \eqref{eq:sys_X}. Moreover, each auxiliary system $\hat{\mathbf{X}}^{n,(-i)}$ from Proposition~\ref{prop:artificial_tagged_i} is adapted to $\hat{\mathcal{F}}^{i,n}_t$ given by 
\eqref{eq:reduced_filtration}, and Proposition~\ref{prop:distribution_tau_i} gives that the infection times of $\mathbf{X}^n$ have the desired conditional laws \eqref{eq:tau_i_prob} in terms of $\hat{\mathcal{F}}^{i,n}_t$ and $\hat{\mathbf{X}}^{n,(-i)}$. For the filtration $\mathcal{F}^n_t$ defined in \eqref{eq:full_filtration}, it can be seen from Appendix~\ref{app:intermediate_systems} that, at each step $k$ of the recursive construction, the trajectories are progressively measurable and the random time $\varsigma^{(k)}$ is a stopping time, so we conclude that $\mathbf{X}^n$ is $\mathcal{F}^n_t$-adapted.

Regarding uniqueness, it is confirmed in Appendix~\ref{app:intermediate_systems} that, at each step, the step-$k$ reflected SDE \eqref{eq:sys_X_hat_k} admits a strong solution which is pathwise unique. Moreover, Lemma~\ref{lem:distinct_infections} ensures that the random index $j^{(k)}$ that achieves the minimum of the potential infection clocks $\tilde{\varsigma}^{i,(k)}$ in \eqref{eq:potential_infection_times_k_step} is almost surely unique. Thus, the trajectories and all the potential infection times at each step $k$ are determined uniquely by the inputs $\{X_0^j, B^j, \chi^{j,(m)} \}_{j \le n, \, m \le k}$. This, together with the construction by concatenation in \eqref{eq:general_P_defn_by_parts}, guarantees that $\mathbf{X}^n$ is almost surely given by a measurable functional of the inputs  $\{X_0^j, B^j, \chi^{j,(m)} \}_{j,m=1}^n$, where we note that $I^n$, and hence also $A^n$, can be recovered from $\mathbf{X}^n$ by definition of the infection times $\tau^i$. In particular, we have uniqueness in law for inputs satisfying Assumption~\ref{assump:inputs}. This completes the proof.
\end{proof}

We end the section by emphasising the following observation which we shall need later.

\begin{prop}[Distinct infection times]\label{prop:distinct_infections}
For any pair of indices $i\neq j$, we have
\[
    \Prob(\tau^i = \tau^j < \infty) = 0.
\]
\end{prop}
\begin{proof} This is immediate from the construction in Appendix \ref{app:intermediate_systems}. Specifically, it holds almost surely that:~(i) at most one particle is moved to $\dagger$ at each step $k$ of the construction, by Lemma \ref{lem:distinct_infections}, and (ii) there is a strictly positive time until the infection of the next step, since the local times are continuous and the exponential random variables are strictly positive.
\end{proof}

\section{Local time under a transformation of the state space}\label{sect:lamperti}

In this section we derive some general results on how the local time of a continuous semimartingale behaves under bijective transformation of the state space (Section \ref{sect:rescaling}). As a special case, we relate a wide class of real-valued reflected diffusions to reflected Brownian motion with drift on a half-line (Section \ref{sect:lamperti_reflected}). In Section \ref{sect:lamperti_for_particle_sys} we apply this to give the proof of Theorem \ref{Thm:Brownian_transformation}.

\subsection{Re-scaling the state space and the effect on local times}\label{sect:rescaling}

Let $X_t$ be a real-valued continuous semi-martingale with respect to a given filtration $(\mathcal{F}_t)_{t\geq 0}$, and consider a random map $(t,\omega, x)\mapsto \Upsilon(t,\omega,x)$ such that $(t,\omega)\mapsto \Upsilon(t,\omega,x)$ is adapted to $(\mathcal{F}_t)_{t\geq 0}$ for all $x\in\R$. Suppose also that $\Upsilon(t,\omega,\cdot): \R \rightarrow \R$ is a homeomorphism with $\lambda^\prime = \Upsilon(t,\omega,\lambda)$ for all $t\in [0,T]$ and $\omega \in \Omega$, for two given points $\lambda,\lambda^\prime \in \R$. This map can then be viewed as a random and time-dependent re-scaling of the state space on either side of $\lambda$, mapping $\lambda$ to a given point $\lambda^\prime$. We want to characterise how the local time of $X$ at $\lambda$ behaves under such a transformation.
For this, we place the following assumptions on the random map $\Upsilon(t,\omega,x)$.

\begin{assump}[The re-scaling map]\label{assump:re-scaling_map}
Fix $T>0$. In addition to the adaptedness, the re-scaling map $\Upsilon:[0,T]\times\Omega \times \R \rightarrow \R$ is required to satisfy the following regularity properties:
\begin{itemize}
    \item[(a)] Each map $x\mapsto \Upsilon(t,\omega,x)$ is strictly increasing with $\Upsilon(t,\omega,\lambda)=\lambda^\prime $, therefore invertible with strictly increasing inverse denoted $x\mapsto \Upsilon^{-1}(t,\omega,x)$ and satisfying $\Upsilon^{-1}(t,\omega,\lambda^\prime)=\lambda $.
    \item[(b)] Each map $x\mapsto \Upsilon(t,\omega,x)$ is a difference of two convex functions.
    \item[(c)] For all $\omega\in \Omega$ and $t\in[0,T]$, the map $x\mapsto \Upsilon(t,\omega,x)$ is continuously differentiable on $[\lambda,\lambda+\delta)$, for some $\delta>0$, where the right-derivative $\partial^+_x\Upsilon(t,\omega,\lambda)$ is used at $x=\lambda$.
    \item[(d)] Uniformly in $\omega\in \Omega$ and $t\in[0,T]$, we have $c \le \partial^+_x\Upsilon(t,\omega,\lambda) \le C$ and $ c \le \partial_x\Upsilon(t,\omega , x) \le  C$ for all $x\in(\lambda,\lambda+\delta)$, for a small enough $\delta>0$ and given constants $c,C>0$.
    \item[(e)] Uniformly in $\omega\in \Omega$ and $x\in(\lambda,\lambda +\delta)$, for some $\delta>0$, the map $t\mapsto \partial_x\Upsilon(t,\omega, x)$ is Lipschitz.
\end{itemize}
\end{assump}
With these assumptions, we obtain the following result on the local time after re-scaling.

\begin{prop}[Local time under re-scaling of the state space]\label{prop:local_time_transformed}
Let $X$ be a continuous semimartingale and let $\Upsilon$ be a random and time-dependent re-scaling map satisfying Assumption~\ref{assump:re-scaling_map}. Then the local time of $X$ at $\lambda$ is related to the local time of $\Upsilon_t(\omega) := \Upsilon(t,\omega , X_t(\omega))$ at $\lambda^\prime = \Upsilon(t,\omega,\lambda)$ by the expression
\[
    \ell_{t}^{\lambda^\prime}(\Upsilon)(\omega)=\int_{0}^{t} \partial_x^+\Upsilon(s,\omega,\lambda ) \di (\ell_{s}^{\lambda}(X)(\omega)),
\]
for all $t\geq 0$, almost surely.
\end{prop}

 For clarity of presentation, we postpone the proof to Section \ref{sect:proof_local_time_transformed} and go directly to a specific application which we use in the proof of Theorem \ref{Thm:Brownian_transformation}. We note that, if $\Upsilon$ does \emph{not} depend on time, then our proof drastically simplifies to a direct application of the occupation time formula with a change of variables. In this case, the statement can be found in \cite[Ex.~1.23,~Ch.~VI]{revuz_yor}. We were unable to find a more general time-dependent version in the literature. Herein, we show that Proposition \ref{prop:local_time_transformed} can still be obtained as a consequence of the occupation time formula, but several technical hurdles arise in establishing this fact.

\subsection{Lamperti transformation of reflected diffusions}\label{sect:lamperti_reflected}

We shall now make use of Proposition \ref{prop:local_time_transformed} above in a particular setting that is relevant to our particle system. Consider a reflected diffusion $X_t$ on the positive half-line with dynamics
\begin{equation}\label{eq:gen_reflected_diffusion}
    \di X_t = b(t,\omega, X_t) \dt + \sigma(t,\omega, X_t) \di W_t + \tfrac{1}{2}\di \ell_t^0(X),
\end{equation}
on a given filtered probability space. Throughout this section, we make the following assumptions on the coefficients of \eqref{eq:gen_reflected_diffusion}. Firstly, the functions $b$ and $\sigma$ are jointly measurable in $(t,\omega,x)$, and, for each $x$, adapted in $(t,\omega)$. Secondly, $\sigma$ is $W^{1,\infty}
$-weakly differentiable in $t$ and $x$, and there exist constants $c, C, K_1, K_2>0$ such that, for all $t\in[0,T]$, $x \in \R^+$ and $\omega\in \Omega$, we have $ c \leq \sigma(t,\omega, x) \leq C $, $|\partial_x \sigma(t,\omega,x)| \leq K_1$, and $|\partial_t \sigma(t,\omega,x)| \leq K_2$, for the precise representatives of the weak derivatives. Thirdly, the weak derivatives $\partial_t\sigma$ and $\partial_x\sigma$ are jointly measurable in $(t,\omega,x)$ and adapted in $(t,\omega)$ for each $x$. Finally, the drift $b$ is of at most linear growth in $x$, uniformly in $(t,\omega)$.

By analogy with the usual Lamperti transformation for real-valued Itô diffusions  (see e.g.~\cite[Sect.~3]{luschgy_pages_06}), we define the random map
\begin{equation}\label{eq:lamperti_map}
    \Upsilon(t,\omega ,y):=\int_{0}^{y}\frac{1}{\sigma(t,\omega,x)} \dx.
\end{equation}
Provided $(t,\omega)\mapsto \sigma(t, \omega ,x)$ is adapted to the given filtration, which we assume throughout, it is clear that each $(t,\omega)\mapsto \Upsilon(t,\omega ,y)$ is also adapted to this filtration. Naturally, one could consider other base points than the origin in \eqref{eq:lamperti_map}, but here we stick to the positive half-line.

\begin{lemma}\label{lemma:func_upsilon} In addition to adaptedness, the map \eqref{eq:lamperti_map} satisfies all the properties (a), (b), (c), (d), and (e) of Assumption \ref{assump:re-scaling_map}.
\end{lemma}
\begin{proof}
The adaptedness was addressed just above. Properties (a)--(d) follow immediately from the definition of $\Upsilon$ together with the non-degeneracy and boundedness of $\sigma(t,\omega,x)$. For the remaining property (e), we can observe that
\begin{align*}
    \left| \partial_x\Upsilon(t,x) - \partial_x\Upsilon(s,x) \right| 
    = \left|  \frac{1}{\sigma(t,x)} - \frac{1}{\sigma(s,x)}    \right|  
    &\le  \frac{ | \sigma(s,x) - \sigma(t,x)|}{\sigma(t,x) \sigma(s,x)}   
    \le \frac{K_2}{c^2} |t-s|,
\end{align*}
so we have the desired Lipschitzness.
\end{proof}

We can now confirm how the rescaling of the state-space by the Lamperti transformation \eqref{eq:lamperti_map} affects the local time of the reflected diffusion \eqref{eq:gen_reflected_diffusion} at the origin.

\begin{prop}\label{prop:Lamperti}
Let $X$ be given by \eqref{eq:gen_reflected_diffusion} and let $\Upsilon$ be given by \eqref{eq:lamperti_map}.
Then $\Upsilon_t := \Upsilon(t,X_t)$ defines another reflected diffusion on $\R^+$ with dynamics
\begin{equation}\label{eq:SDE_lamperti}
    \di\Upsilon_{t}=\tilde{b}(t, \omega, \Upsilon_{t}) \dt + \di W_t + \tfrac{1}{2} \di \ell_t^0(\Upsilon),
\end{equation}
where the modified drift $\tilde{b}$ is jointly measurable, adapted in $(t,\omega)$, and satisfies
\begin{equation*}
    \esssup_{(t,\omega)\in [0,T]\times \Omega}| \tilde{b}(t,\omega ,x) | \leq \kappa (1+x)\quad \text{for all} \quad x\geq0
\end{equation*}
for a given constant $\kappa >0$.
\end{prop}

\begin{proof}
By definition, we have $\Upsilon_t(\omega)=\Upsilon(t,X_t(\omega))(\omega)$, where  $\Upsilon$ is given by \eqref{eq:lamperti_map}. By the assumptions on $\sigma$, we can then apply the generalized It\^o's formula \cite[Ex.~3.12,~Ch.~4]{revuz_yor} (also generalized to weak differentiability as in \cite[Thm.~10.1, Ch.~2]{Krylov_book}) for the adapted function $(t,\omega , x )\mapsto \Upsilon(t,\omega,x)$. This is readily seen to yield
\begin{equation}\label{eq:SDE_upsilon}
    \Upsilon_t(\omega) - \Upsilon_0(\omega) = \int_0^t\bar{b}(s,\omega, X_s(\omega)) \ds + W_t(\omega) + \frac{1}{2} \int_0^t\frac{1}{\sigma(s,\omega, X_s(\omega))} \di (\ell_s^0(X)(\omega)),
\end{equation}
where we have introduced the jointly measurable and adapted function
\begin{equation}\label{eq:modified_drift}
    \bar{b}(t, \omega, x) := - \int_0^{x} \frac{\partial_t \sigma(t,\omega,y)}{\sigma(t,\omega, y)^{2}}  \dy + \frac{b(t,\omega,x)}{\sigma(t,\omega,x)} -\frac{1}{2} \partial_x \sigma(t, \omega,x).
\end{equation}
Since $ \ell^0(X)$ is carried by the set $\{ t\geq 0 \, : \, X_t = 0 \}$ \cite[Prop.~1.3]{revuz_yor}, the final term in \eqref{eq:SDE_upsilon} equals
\begin{equation}\label{eq:rescaled_local_time}
    \int_0^t\frac{1}{\sigma(s,\omega, 0)} \di (\ell_s^0(X)(\omega)).
\end{equation}
By Lemma \ref{lemma:func_upsilon}, we can apply Proposition~\ref{prop:local_time_transformed} to conclude that \eqref{eq:rescaled_local_time} is in fact equal to $\ell^0_t(\Upsilon)$. Next, we let $\Upsilon^{-1}(t,x)(\omega)$ denote the well-defined inverse of $x\mapsto \Upsilon(t,x)(\omega)$ on $\R^+$, and define
\begin{equation}\label{eq:modified_drift_tilde}
    \tilde{b}(t,\omega ,x):=\bar{b}\bigl(t,\omega, \Upsilon^{-1}(t,x)(\omega) \bigr).
\end{equation}
Then it follows that $\Upsilon$ has the dynamics
\[
    \Upsilon_t(\omega) - \Upsilon_0(\omega) = \int_0^t\tilde{b}(s,\omega, \Upsilon_s(\omega)) \ds + W_t(\omega) + \frac{1}{2} \ell^0_t(\Upsilon)(\omega).
\]
Finally, it remains to observe that, in view of how $\tilde{b}$ was defined through \eqref{eq:modified_drift} and \eqref{eq:modified_drift_tilde}, we have
\[
    |\tilde{b}(t,\omega,x)|\leq c_1 x + c_2(1+x)+c_3,
\]
by the definition of $\Upsilon$ and by the assumptions on the coefficients of \eqref{eq:gen_reflected_diffusion}.
\end{proof}

Due to Proposition~\ref{prop:local_time_transformed}, the above proof was an immediate consequence of Itô's formula. While Proposition~\ref{prop:local_time_transformed} holds for continuous semimartingales in general, here we applied it to the reflected diffusion \eqref{eq:gen_reflected_diffusion}. In this case, we note that one could also pursue an alternative proof. Very briefly, one sees that $\Upsilon_t=|\Upsilon_t|$, since $X_t$ is non-negative at all times, and then one can apply Tanaka's formula which makes the local time $\ell^0_t(\Upsilon)$ of $\Upsilon_t$ appear. Moreover, one can exploit that $\ell^0(X)$ is part of the reflected dynamics for $X$, by comparing with \eqref{eq:SDE_upsilon}. Thus, using the assumptions on our coefficients, the properties of $\ell^0(X)$, and appropriately cancelling terms, one  arrives at \eqref{eq:SDE_lamperti} with the representation \eqref{eq:rescaled_local_time} for the local time.

\subsection{Reflected Brownian motion in the frame of the moving boundary}\label{sect:lamperti_for_particle_sys}

In this section we return to the specific setting of our particle system and use Proposition \ref{prop:Lamperti}, with some minor variations, to give a proof of Theorem \ref{Thm:Brownian_transformation}.

\begin{proof}[Proof of Theorem \ref{Thm:Brownian_transformation}] 

By the construction of the particle system in Section \ref{sec:existence_sys}, and recalling the definition $\tau^i=\inf\{t>0:X_{t}^{i,n}=\dagger\}$, we can check that $X_{t}^{i,n}\in(a,b)$ if and only if we have $\hat{X}_{t}^{i,n, (-i)}\in(a,b)$ and $t<\tau^i$, since $X_{t}^{i,n}(\omega)=\hat{X}_{t}^{i,n, (-i)}(\omega)$ for all $\omega\in\{t<\tau^{i}\}$. Moreover, $\mathcal{G}_{t}^{i,n}\subseteq\mathcal{\hat{{F}}}_{t}^{i,n}$, so we can then write
\begin{align*}
    \Prob(X_{t}^{i,n}\in(a,b)\mid\mathcal{G}_{t}^{i,n})
    & =\E\bigl[\mathbbm{1}_{\{\hat{X}_{t}^{i,n, (-i)}\in(a,b)\}}\Prob(t<\tau^{i}\mid\mathcal{\hat{F}}_{t}^{i,n})\mid\mathcal{G}_{t}^{i,n}].
\end{align*}
Now observe that we can apply Fubini's theorem, a change of variables, and integration by parts, to deduce that
\begin{align*}
    \int_0^t \int_{s-\bar{d}}^s \varrho(s-r) \di \hat{I}^{n, (-i)}_r \di s &=\int_{0}^t \int_0^{t-r}  \varrho(u) \di u  \di \hat{I}^{n, (-i)}_r = -\int_{0}^t  \hat{I}^{n, (-i)}_r \di \Bigl( \int_0^{t-r}\varrho(u) \di u \Bigr)  \\
    & =  \int_{0}^t  \varrho(t-r) \hat{I}^{n, (-i)}_r \di  r = \int_{t-\bar{d}}^t  \varrho(t-r) \hat{I}^{n, (-i)}_r \di r,
\end{align*}
where we used our convention that $\hat{I}^{n, (-i)}_r=0$ for $r\leq 0$. Consequently, the advancing front $\hat{A}^{n, (-i)}$ is absolutely continuous with
\begin{equation}\label{eq:A_is_AC}
    \hat{A}^{n, (-i)}_t = a_0 + \int_0^t \partial_s\hat{A}^{n, (-i)}_s \di s,\quad  \partial_s\hat{A}^{n, (-i)}_s := \alpha \int_{s-\bar{d}}^s \varrho(s-r) \di \hat{I}^{n, (-i)}_r,
\end{equation}
As $\varrho$ is right-continuous and the total variation of $\hat{I}^{n, (-i)}_r$ is bounded by $1$, we furthermore get
\begin{equation}\label{eq:A_n_derivative_bound}
    \Vert \partial_s\hat{A}^{n, (-i)}\Vert_{L^\infty(\mathbb{R^+})} \leq \alpha \Vert \varrho \Vert_{L^\infty([0,\bar{d}])}<\infty.
\end{equation}
By considering $\hat{X}_{t}^{i,n, (-i)}$ in the frame of the associated moving boundary $\hat{A}^{n,(-i)}$, we can therefore align ourselves with the setting of Section \ref{sect:lamperti_reflected}. Indeed, setting $Y_t^{i,n}:=\hat{X}_{t}^{i,n, (-i)}-\hat{A}^{n, (-i)}_t$ we have
\begin{equation}\label{eq:shift_halfline}
    Y_t^{i,n}= Y_0^{i,n} + \int_0^t b(s,Y_s^{i,n}+\hat{A}^{n, (-i)}_s) - \partial_s\hat{A}^{n, (-i)}_s \ds +\int_0^t \sigma(s,Y_s^{i,n}+\hat{A}_s^{n, (-i)})\di B^i_s + \frac{1}{2} \ell^0_t(Y^{i,n}),
\end{equation}
so $Y_t^{i,n}$ is a reflected diffusion on the positive half-line of the form \eqref{eq:gen_reflected_diffusion}. From here, we re-scale the state-space according to the specific Lamperti transformation
\begin{equation*}
    \Upsilon^{i,n}(t,y)(\omega):= \Upsilon^{i,n}(t,\omega, y):=\int_0^y \frac{1}{\sigma(t,x+\hat{A}^{n, (-i)}_t(\omega))}\dx.
\end{equation*}
Due to \eqref{eq:A_is_AC}--\eqref{eq:A_n_derivative_bound}, the map $(t,\omega,x)\mapsto \sigma(t,x+\hat{A}^{n, (-i)}_t(\omega))$ is such that Lemma \ref{lemma:func_upsilon} holds for this definition of $\Upsilon^{i,n}$. Thus, we can apply Proposition \ref{prop:Lamperti} to see that 
\begin{equation}\label{eq:Z_Upsilon_Y}
    Z_{t}^{i,n}(\omega):=\Upsilon^{i,n}(t,\omega,Y^{i,n}_t(\omega))
\end{equation}
has the desired dynamics \eqref{eq:Z_i_n_dynamics}. Next, $\Upsilon^{i,n}(t,\cdot)$ is strictly increasing, so the image $\Upsilon^{i,n}(t,(x,y))$ is equal to the open interval from $\Upsilon^{i,n}(t,x)$ to $\Upsilon^{i,n}(t,y)$ for any $x\leq y$. Thus, by \eqref{eq:Z_Upsilon_Y}, the event that $\hat{X}_{t}^{i,n, (-i)}\in (a,b)$ is equivalent to $Z^{i,n}_t$ being in the open interval from $\Upsilon^{i,n}(t,a-\hat{A}_{t}^{n, (-i)})$ to $\Upsilon^{i,n}(t,b-\hat{A}_{t}^{n, (-i)})$, and hence
\begin{equation}
    \Prob(X_{t}^{i,n}\in(a,b)\mid\mathcal{G}_{t}^{i,n})=\E\bigl[\mathbbm{1}_{\{Z^{i,n}\in(\Upsilon^{i,n}(t,a-\hat{A}_{t}^{n, (-i)}),\Upsilon^{i,n}(t,b-\hat{A}_{t}^{n, (-i)}))\}}\Prob(t<\tau^{i}\mid\mathcal{\hat{F}}_{t}^{i,n})\mid\mathcal{G}_{t}^{i,n}].\label{eq:prob_estimate_tower_law}
\end{equation}
Finally, Proposition~\ref{prop:local_time_transformed} gives $\di \hat{\ell}^{i,n}_t = \sigma(r,\hat{A}_r^{n,(-i)})\di \ell_r^0(Z^{i,n})$, so the definition of $\tau_Z^{i,n}$ in (\ref{eq:tau_Upsilon}) yields
\begin{align*}
\Prob\bigl(t<\tau_{Z}^{i,n}\mid \mathcal{G}_{t}^{i,n}\vee\sigma(Z_{s}^{i,n}:s\in[0,t])\big) & =e^{-\int_{0}^{t}\sigma(r,\hat{A}_{r}^{n, (-i)})\gamma(r,\hat{C}_{r}^{n, (-i)})\di \ell_{r}^{0}(Z^{i,n})}
=e^{-\int_{0}^{t}\gamma(r,\hat{C}_{r}^{n, (-i)})\di \hat{\ell}_{r}^{i,n}},
\end{align*}
where the last term agrees with $\Prob(t<\tau^{i}\mid\mathcal{\hat{F}}_{t}^{i,n})$ by virtue of \eqref{eq:tau_i_prob}. Inserting this into (\ref{eq:prob_estimate_tower_law}), we have the desired expression (\ref{eq:Prob_Brownian}) for $Z^{i,n}$ and $\tau_{Z}^{i,n}$ which completes the proof.
\end{proof}

\subsection{Proof of Proposition \ref{prop:local_time_transformed}}\label{sect:proof_local_time_transformed}

For the remainder of this section, since only a linear shift is involved, there is no loss of generality in assuming $\lambda = \lambda^\prime = 0$. For notational simplicity, we also suppress the dependence of $\Upsilon$ on $\omega$ throughout the proofs. We first single out an auxiliary lemma.

\begin{lemma}\label{lemma:lip}
Let $\Upsilon$ satisfy Assumption \ref{assump:re-scaling_map},and, for every $\varepsilon > 0$, let $\varphi_{\varepsilon} \in C^{\infty}_0$ be a mollifier supported on $[0, \varepsilon]$ which smoothly approximates the Dirac mass at 0. For all $\varepsilon>0$, the function $t \mapsto \varphi_{\varepsilon}(\Upsilon(t,\omega,x)) (\partial_x \Upsilon(t,\omega,x))^2$ is then Lipschitz continuous, and the Lipschitz constant can be taken to be proportional to $\varepsilon^{-1}$ uniformly in $x \in \R^+$ and $\omega \in \Omega$.
\end{lemma}
\begin{proof}To simplify notation, we suppress the dependence on $\omega$. For each $\varepsilon > 0$, we have $\varphi_{\varepsilon} \in C^{\infty}_0$ supported on $[0, \varepsilon]$ with $\Vert \varphi_{\varepsilon} \Vert_{L^1} = 1 $. Furthermore, we can assume without loss of generality that there are constants $k_1,k_2 >0$ such that $\Vert \varphi_{\varepsilon} \Vert_{\infty} \le k_1 \varepsilon^{-1}$ and $\Vert \partial_x \varphi_{\varepsilon} \Vert_{\infty}\le k_2 \varepsilon^{-2}$ for all $\varepsilon>0$.

For every $s \ge 0$, it follows from (a) and (d) in Assumption \ref{assump:re-scaling_map} along with the fundamental theorem of calculus that $\Upsilon(s,x)=\int_0^x \partial_y \Upsilon(s,y) \dy \geq c x$ for small enough $x\geq 0$, and the map is also increasing, so $\varphi_{\varepsilon}(\Upsilon(t,x)) = \varphi_{\varepsilon}(\Upsilon(t,x)) \mathbbm{1}_{ \left\{ 0 \le x \le \varepsilon /c \right\} }$ for $\varepsilon>0$ small enough. Given this, and applying the triangle inequality, we get
\begin{align}\label{eq:lip-0}
    &\left| \varphi_{\varepsilon}(\Upsilon(t,x))(\partial_x \Upsilon(t,x))^2 - \varphi_{\varepsilon}(\Upsilon(s,x)) (\partial_x \Upsilon(s,x))^2    \right|  \nonumber \\
    &\quad \le  (\partial_x \Upsilon(s,x))^2  \left| \varphi_{\varepsilon}(\Upsilon(t,x))  - \varphi_{\varepsilon}(\Upsilon(s,x))   \right| \mathbbm{1}_{ \left\{ 0 \le x \le \varepsilon /c \right\} }  \nonumber \\
    &\qquad + \varphi_{\varepsilon}(\Upsilon(t,x)) \left|  (\partial_x \Upsilon(t,x))^2   -  (\partial_x \Upsilon(s,x))^2   \right| \mathbbm{1}_{ \left\{ 0 \le x \le \varepsilon /c \right\} } .
\end{align}
We now look for bounds for each of the two terms on the right-hand side. By the fundamental theorem of calculus for $\varphi_\varepsilon$, we get
\begin{align*}
    \left| \varphi_{\varepsilon}(\Upsilon(t,x))  - \varphi_{\varepsilon}(\Upsilon(s,x))   \right| 
    = \left| \int_{\Upsilon(s,x)}^{\Upsilon(t,x)} \partial_y \varphi_{\varepsilon}(y) \dy  \right| 
    &\le k_2 \varepsilon^{-2} \left| \Upsilon(t,x) - \Upsilon(s,x)   \right|. \nonumber 
\end{align*}
Since $\Upsilon(t,0)=\Upsilon(s,0) = 0$, the fundamental theorem of calculus and Assumption \ref{assump:re-scaling_map}(e) then give
\begin{equation*}
    \left| \varphi_{\varepsilon}(\Upsilon(t,x))  - \varphi_{\varepsilon}(\Upsilon(s,x))   \right| 
    \le k_2 \varepsilon^{-2} \int_0^x |\partial_y \Upsilon(t,y) - \partial_y \Upsilon(s,y)|   \dy \le k_2 \varepsilon^{-2} Kx|t-s|,
\end{equation*}
for $x\in [0,\varepsilon/c]$ with $\varepsilon>0$ small enough. Moreover, by Assumption \ref{assump:re-scaling_map}(d) $\partial_x \Upsilon(s,x)\le C$ for all $x\in [0,\varepsilon/c]$ with $\varepsilon>0$ small enough. Thus, the first term in \eqref{eq:lip-0} is bounded by $\tilde{C}_1 \varepsilon^{-2} x|t-s|$ on the support $x \le \varepsilon/c$, with $\tilde{C}_1 = C^2 K k_2$.

Turning to the second term on the right-hand side of \eqref{eq:lip-0}, we have $\Vert \varphi_{\varepsilon} \Vert_{\infty} \le k_1 \varepsilon^{-1}$. Furthermore, using the Lipschitz continuity and the uniform upper bound $C$:
\begin{align*}
    \left|  (\partial_x \Upsilon(t,x))^2   -  (\partial_x \Upsilon(s,x))^2   \right| 
    &= |\partial_x \Upsilon(t,x)+\partial_x \Upsilon(s,x)| |\partial_x \Upsilon(t,x) - \partial_x \Upsilon(s,x)|
    \le 2C K |t-s|.
\end{align*}
Thus, the second term is bounded by $\tilde{C}_2 \varepsilon^{-1} |t-s|$, with $\tilde{C}_2 =2 C K k_1$. Combining these bounds into \eqref{eq:lip-0}, we obtain the final Lipschitz estimate
\begin{align*}
    \left| \varphi_{\varepsilon}(\Upsilon(t,x))(\partial_x \Upsilon(t,x))^2 - \varphi_{\varepsilon}(\Upsilon(s,x)) (\partial_x \Upsilon(s,x))^2    \right|
    &\le (\tilde{C}_1 \varepsilon^{-2} x + \tilde{C}_2 \varepsilon^{-1})|t-s| \mathbbm{1}_{ \left\{ 0 \le x \le \varepsilon / c \right\}},
\end{align*}
for constants $\tilde{C}_1, \tilde{C}_2$. Since the indicator function restricts to $x\leq \varepsilon /c$, the claim follows.
\end{proof}

We now return to the proof of Proposition~\ref{prop:local_time_transformed}.

\begin{proof}[Proof of Proposition~\ref{prop:local_time_transformed}]				
Let us begin the proof by fixing a family of mollifiers $\left\{ \varphi_{\varepsilon} \right\}_{\varepsilon \in \R}$ smoothly approximating the Dirac mass at 0, as in the proof of Lemma \ref{lemma:lip}. By Assumption \ref{assump:re-scaling_map}(b)--(c), $\Upsilon_t$ is again a semimartingale with $[\Upsilon]_t=\int_0^t (\partial_x \Upsilon(s,X_s))^2 \di[X]_s$, for all $t\geq0$, as can e.g.~be seen from a time-dependent version of the Meyer--It\^o formula \cite[Ch.~IV, Thm.~70]{protter}. Applying the occupation time formula \cite[Cor.~1.6,~Ch.~6]{revuz_yor} with positive Borel function $\varphi_{\varepsilon}$, and using the almost sure right-continuity of the local time in $x$ for continuous semi-martingales \cite[Ch.~6,~Thm.~1.7]{revuz_yor},
\begin{equation*}
    \ell_t^0(\Upsilon)
    = \lim_{\varepsilon \downarrow 0} \int_{0}^t \varphi_{\varepsilon}(\Upsilon_s) \di [\Upsilon ]_s =  \lim_{\varepsilon \downarrow 0} \int_{0}^t \varphi_{\varepsilon}(\Upsilon(s, X_s)) (\partial_x \Upsilon(s,X_s))^2 \di [X]_s ,
\end{equation*}
for all $t\geq 0$ (almost surely). Next, the generalised occupation time formula for Borel functions on $[0, \infty) \times \Omega \times \R$  \cite[Ex.~1.13,~Ch.~6]{revuz_yor} then gives that
\begin{equation}\label{eq:key_local_time}
    \ell_t^0(\Upsilon)
    = \lim_{\varepsilon \downarrow 0} \int_{\R^+} \int_0^t \varphi_{\varepsilon}(\Upsilon(s, y)) (\partial_x \Upsilon(s,y))^2 \di \ell_s^y(X) \dy
\end{equation}
for all $t\geq 0$ (almost surely). From here, the claim is simply that the limit as $\varepsilon \rightarrow 0$ equals the desired expression $\int_0^t \partial^+_x \Upsilon(s,0) \di \ell_s^0(X)$. This is intuitive, but it becomes a little delicate to get to a point where we can do a change of variables $z=\Upsilon(s,y)$ to achieve the conclusion.
	
The local time of $X$ at $y$ is non-decreasing, hence of finite variation, and (almost surely) it is continuous in time (see e.g.~\cite[Ch.~6]{revuz_yor}). Thus, for all $y \in \R^+$, the aforementioned integral against $\ell^y(X)$ on $[0,t]$ is a Riemann--Stieltjes integral given by the limit of its approximating sums along arbitrary partitions $\left\{\pi_n = \{s_0, s_1, \dots, s_n\}\right\}_{n \in \N}$ of $[0,t]$. Adding and subtracting the Riemann sums for the Riemann--Stieltjes integrals $\int_0^t  \partial^+_x \Upsilon(s,0) \di \ell_s^0(X)$ and  $\int_0^t \varphi_{\varepsilon}(\Upsilon(s, y))  (\partial_x \Upsilon(s,y))^2 \di \ell_s^y(X)$, we can write down the following equality
\begin{equation}\label{eq:conv_local_time}
    \int_0^t \partial^+_x \Upsilon(s,0) \di \ell_s^0(X) = \int_{\R^+} \int_0^t \varphi_{\varepsilon}(\Upsilon(s, y)) (\partial_x \Upsilon(s,y))^2\di \ell_s^y(X) \dy + S_1(n) + S_2(n, \varepsilon) + S_3(n, \varepsilon),
\end{equation}
where we have defined
\begin{align*}
    S_1(n)
    &:= \int_0^t \partial^+_x \Upsilon(s,0) \di \ell_s^0(X) - \sum_{i =0}^{n-1} \partial^+_x \Upsilon(s_i,0)(\ell^0_{s_{i+1}} - \ell^0_{s_i}), \\
    S_2(n, \varepsilon)
    &:= \int_{\R^+} \sum_{i = 0}^{n-1} \varphi_{\varepsilon}(\Upsilon(s_i, y)) (\partial_x \Upsilon(s_i,y))^2 (\ell^y_{s_{i+1}} - \ell^y_{s_i}) \dy \\&\qquad\qquad- \int_{\R^+} \int_0^t \varphi_{\varepsilon}(\Upsilon(s, y)) (\partial_x \Upsilon(s,y))^2 \di \ell_s^y(X) \dy, \\
    S_3(n, \varepsilon)
    &:= \sum_{i=0}^{n-1} \partial^+_x \Upsilon(s_i,0)(\ell^0_{s_{i+1}} - \ell^0_{s_i}) - \int_{\R^+} \sum_{i = 0}^{n-1} \varphi_{\varepsilon}(\Upsilon(s_i, y)) (\partial_x \Upsilon(s_i,y))^2 (\ell^y_{s_{i+1}} - \ell^y_{s_i}) \dy.
\end{align*}
	
Let $\delta > 0$ be given. By definition of the Riemann--Stieltjes integral, there exists $N_1(\delta) \in \N$ such that for all $n > N_1(\delta)$ we have
\begin{equation}\label{eq:bound_A1}
    |S_1(n) | = \left| \int_0^t \partial^+_x \Upsilon(s,0) \di \ell_s^0(X) -  \sum_{i=0}^{n-1} \partial^+_x \Upsilon(s_i,0)(\ell^0_{s_{i+1}}(X) - \ell^0_{s_i}(X)) \right| \le \delta.
\end{equation}
	
Now consider $|S_2(n, \varepsilon)|$. Using linearity of the integrals, we can rewrite
\begin{align*}
    S_2(n, \varepsilon) 
    &= \int_{\R^+} \biggl(\, \sum_{i = 0}^{n-1} \varphi_{\varepsilon}(\Upsilon(s_i, y)) (\partial_x \Upsilon(s_i,y))^2 (\ell^y_{s_{i+1}} - \ell^y_{s_i}) \\ &\qquad\qquad- \int_0^t \varphi_{\varepsilon}(\Upsilon(s, y)) (\partial_x \Upsilon(s,y))^2 \di \ell_s^y(X) \biggr) \dy   \\
    &=  \int_{\R^+} \sum_{i = 0}^{n-1} \int_{s_{i}}^{s_{i+1}} \varphi_{\varepsilon}(\Upsilon(s_i, y)) (\partial_x \Upsilon(s_i,y))^2 \di \ell_s^y(X)
    \\& \qquad\qquad- \sum_{i = 0}^{n-1} \int_{s_{i}}^{s_{i+1}} \varphi_{\varepsilon}(\Upsilon(s, y)) (\partial_x \Upsilon(s,y))^2 \di \ell_s^y(X)  \dy  .
\end{align*}
Since $\ell^y_t$ is an increasing process, $\mathrm{TV}_{[0,t]}(\ell^y) = \ell^y_t$, and so we have
\begin{align*}
    |S_2(n, \varepsilon) | 
    &\le \int_{\R^+}  \sum_{i = 0}^{n-1} \int_{s_{i}}^{s_{i+1}} \left| \varphi_{\varepsilon}(\Upsilon(s_i, y)) (\partial_x \Upsilon(s_i,y))^2 
    -  \varphi_{\varepsilon}(\Upsilon(s, y)) (\partial_x \Upsilon(s,y))^2 \right| \di \ell_s^y(X)  \dy.
\end{align*}
By Lemma \ref{lemma:lip}, and recalling from its proof that $\varphi_{\varepsilon}(\Upsilon(s,y))$ is supported on $y \in [0, \varepsilon /c]$, we get
\begin{align*}
    |S_2(n, \varepsilon) | 
    &\le \int_{\R^+}  \mathbbm{1}_{[0, \varepsilon /c]}(y) \sum_{i = 0}^{n-1} \int_{s_{i}}^{s_{i+1}} \tilde{C} \varepsilon^{-1}(s - s_i) \di \ell_s^y(X)  \dy \\
    &= \tilde{C} \varepsilon^{-1} \int_0^{\varepsilon /c} \left( \int_0^t s \di \ell^y_s(X)
    - \int_0^t \sum_{i = 0}^{n-1} s_i \mathbbm{1}_{[s_i, s_{i+1}]}(s) \di \ell_s^y(X)  \right) \dy \\
    &\le \tilde{C} \varepsilon^{-1} \int_0^{\varepsilon /c}  \sup_{y \in \R^+} \ell_t^y  \, \left( \sup_{s \in [0,t]} \left| s -  \sum_{i = 0}^{n-1} s_i \mathbbm{1}_{[s_i, s_{i+1}]}(s)  \right| \right) \dy
\end{align*}
From Barlow--Yor's BDG type inequality for local times \cite[Ch.~XI,~Thm.~2.4]{revuz_yor} (see \cite[Page~199]{barlow_yor} for the case of semimartingales), the expectation of $L : = \sup_{y \in \R^+} \ell_t^y$ is finite. In particular, $L$ is finite almost surely. Moreover, we can always find a partition $\pi_n$ of $[0,t]$ fine enough that $\sup_{s \in [0,t]} | s -  \sum_{i = 0}^{n-1} s_i \mathbbm{1}_{[s_i, s_{i+1}]}(s)|$ is as small as we like: for $\delta > 0$, there exists $N_2(\delta)$ such that
\begin{equation*}
    \sup_{s \in [0,t]} \left| s -  \sum_{i = 0}^{n-1} s_i \mathbbm{1}_{[s_i, s_{i+1}]}(s)  \right| \le \delta, \quad \forall n > N_2(\delta).
\end{equation*}
Hence we conclude that, for all $n > N_2(\delta)$,
\begin{equation}\label{eq:bound_A2}
    |S_2(n, \varepsilon) | \le \tilde{C} c^{-1} L \delta,
\end{equation}
and we know the right-hand side is almost-surely finite.
	
Finally, consider $|S_3(n, \varepsilon)|$. Applying the change of variables $z = \Upsilon(s_i, y)$, and using (a) and (c) from Assumption \ref{assump:re-scaling_map} as well as the inverse function theorem, we get
\begin{equation*}
    S_3(n, \varepsilon)=  \sum_{i=0}^{n-1} \left( \partial^+_x \Upsilon(s_i,0)(\ell^0_{s_{i+1}} - \ell^0_{s_i}) -   \int_{\R^+}  \!\varphi_{\varepsilon}(z)   \partial_x \Upsilon(s_i,   \Upsilon^{-1}(s_i, z) ) (\ell^{\Upsilon^{-1}(s_i, z)}_{s_{i+1}} - \ell^{\Upsilon^{-1}(s_i, z)}_{s_i}) \di z  \right) .
\end{equation*}
Recalling that $\varphi_{\varepsilon}$ integrates to 1 and is supported only $[0, \varepsilon]$, we can then estimate
\begin{align*}
    |S_3(n, \varepsilon)| 
    &= \left| \int_0^{\varepsilon} \varphi_{\varepsilon}(z) \sum_{i=0}^{n-1} \left( \partial^+_x \Upsilon(s_i,0)(\ell^0_{s_{i+1}} - \ell^0_{s_i}) -  \partial_x \Upsilon(s_i,   \Upsilon^{-1}(s_i, z) )(\ell^{\Upsilon^{-1}(s_i, z)}_{s_{i+1}} - \ell^{\Upsilon^{-1}(s_i, z)}_{s_i}) \right) \di z   \right| \\
    &\le \int_0^{\varepsilon} \varphi_{\varepsilon}(z) \sum_{i=0}^{n-1} \left| \partial_x \Upsilon(s_i,   \Upsilon^{-1}(s_i, 0) ) \ell^{\Upsilon^{-1}(s_i, 0)}_{s_{i+1}} - \partial_x \Upsilon(s_i, \Upsilon^{-1}(s_i, z) ) \ell^{\Upsilon^{-1}(s_i, z)}_{s_{i+1}} \right| \di z \\
    &\quad + \int_0^{\varepsilon} \varphi_{\varepsilon}(z) \sum_{i=0}^{n-1} \left| \partial_x \Upsilon(s_i,   \Upsilon^{-1}(s_i, 0) ) \ell^{\Upsilon^{-1}(s_i, 0)}_{s_i} -  \partial_x \Upsilon(s_i,   \Upsilon^{-1}(s_i, z) ) \ell^{\Upsilon^{-1}(s_i, z)}_{s_i} \right| \di z
\end{align*}
Since $y \mapsto \ell_t^y(X)$ is c\`{a}dl\`{a}g, and $z \mapsto \Upsilon^{-1}(s,z)$ is continuous and strictly increasing, by Assumption \ref{assump:re-scaling_map}(a), we have that $z \mapsto \ell_t^{\Upsilon^{-1}(s,z)}(X) $ is right-continuous. Likewise, $z \mapsto \partial_x \Upsilon(s,   \Upsilon^{-1}(s, z))$ will be right-continuous on a right-neighbourhood of zero, by Assumption \ref{assump:re-scaling_map}(c). In particular, their product is right continuous, for all small enough $z$, so, given $\delta/n$, there exists $\beta(\delta) $ such that
\begin{equation*}
    |z| \le \beta(\delta) \;\;  \text{implies}\;\; |\partial_x \Upsilon(s_i,   \Upsilon^{-1}(s_i, z)) \ell_t^{\Upsilon^{-1}(s_i, z)}(X) - \partial_x \Upsilon(s_i,   \Upsilon^{-1}(s_i, 0))\ell_t^{\Upsilon^{-1}(s_i, 0)}(X)| \le \frac{\delta}{n},
\end{equation*}
for all $t\geq 0$. Then it follows that, for all $\varepsilon \in [0, \beta(\delta)]$, we have
\begin{equation}\label{eq:bound_A3}
    |S_3(n,\varepsilon) | \le 2 \delta, \quad \forall n \in \N.
\end{equation}
	
From the expression \eqref{eq:conv_local_time} and the bounds \eqref{eq:bound_A1}, \eqref{eq:bound_A2} and \eqref{eq:bound_A3} for $S_1$, $S_2$, and $S_3$, on the right-hand side, we conclude that for all $\delta > 0$ and all $n \in \N$ such that $n \ge N(\delta)= \max\{ N_1(\delta), N_2(\delta)\}$, there exists $\beta(\delta) > 0$ such that
\begin{equation*}
    \left| \int_0^t \partial^+_x \Upsilon(s,0) \di \ell_s^0(X) - \int_{\R^+} \int_0^t \varphi_{\varepsilon}(\Upsilon(s, y)) (\partial_x \Upsilon(s,y))^2 \di \ell_s^y(X) \dy \right| \le (3 + \tilde{C} c^{-1} L ) \delta \quad \forall \varepsilon \in [0, \beta(\delta)].
\end{equation*}
Taking the limit as $\delta \rightarrow 0$, and noting that $\varepsilon \rightarrow 0$ at the same time, we conclude that the right-hand side of \eqref{eq:key_local_time} converges to the intended limit (almost surely). This completes the proof. 
\end{proof}

\section{Martingale properties of the infected proportion}\label{sect:martingale_prop}

In this section, we develop the martingale machinery for the infected proportion $I^n$ and give the proof of Theorem \ref{Thm:convergence_Barnes}. Our arguments exploit the exact specification of the infection mechanism as well as our construction of the particle system and its auxiliary systems in Section \ref{sec:existence_sys}.

\subsection{Compensated martingality of the infected proportion}

Recall from Theorem \ref{thm:existence} that the infection times $\tau^{i}$ have the conditional laws
\begin{equation}\label{eq:cond_exp_dist}
    \Prob(\tau^{i}\geq s\mid\hat{\mathcal{F}}_{t}^{i,n})=1-e^{-\int_{0}^{s}\gamma(r,\hat{C}_{r}^{n,(-i)})\di \hat{\ell}_{r}^{i,n}},\quad s\leq t,
\end{equation}
for each $i=1,\ldots,n$, where we also recall that 
\begin{equation}\label{eq:recall_reduced_filtration}
    \hat{\mathcal{F}}_{t}^{i,n}=\sigma\bigl((X_{0}^{i},B_{s}^{i}),(X_{0}^{j},B_{s}^{j},\{\chi^{j,(k)}\}_{k=1}^{n}):s\in[0,t],\,j\in\{1,\ldots,n\}\!\setminus\!\{i\}\bigr)
\end{equation}
for each $i=1,\ldots,n$. In addition to this, we shall consider
the filtrations $\mathcal{I}^{i,n}$, given by
\begin{equation}\label{eq:I_filtration}
    \mathcal{I}_{t}^{i,n}:=\sigma\bigl(\{s<\tau^i\}:s\in[0,t]\bigr),\quad t\geq0,
\end{equation}
for $i=1,\ldots, n$. Note that  $\mathcal{I}_{t}^{i,n}$ reveals the infection status of the $i$'th particle up to time $t$,
but nothing else. Adding this information to $\hat{\mathcal{F}}^{i,n}$,
we define 
\begin{equation}\label{eq:F_i_filtration}    \mathcal{F}_{t}^{i,n}:=\hat{\mathcal{F}}_{t}^{i,n}\lor\mathcal{I}_{t}^{i,n},\quad t\geq0,
\end{equation}
for $i=1,\ldots, n$. Crucially, $\hat{\mathcal{F}}^{i,n}_t$ itself is not rich enough to reveal whether or not particle $i$ has been infected up to time $t$, and hence it also cannot distinguish if the other particles have been infected up that time, since they are coupled to the infection status of particle $i$.

We first derive two basic lemmas which relate the filtration $\mathcal{F}^{i,n}$ to the subfiltration $\hat{\mathcal{F}}^{i,n}$ for certain conditional expectations involving the infection times. Aside from the interacting aspects that we deal with in this section, and the corresponding form of \eqref{eq:cond_exp_dist} and the filtrations, these lemmas are analogues of the key building blocks in the theory of hazard processes from the mathematical credit risk literature; see \cite[Ch.~7]{jeanblanc} and \cite[Ch.~3]{BieleckiJeanblancRutkowski_2009}.

\begin{lemma}\label{lem:compensator1}
Fix $i\in\{1,\ldots,n\}$, and let $\mathcal{F}^{i,n}$ and $\hat{\mathcal{F}}^{i,n}$ be the filtrations defined in \eqref{eq:F_i_filtration} and \eqref{eq:recall_reduced_filtration}, respectively. For any random variable $Y\in L^{1}(\Omega,\Prob)$, defined on the same probability space as the particle system, it holds for all $s\geq0$ that
\[
    \E[Y\mathbbm{1}_{s<\tau^i}\mid\mathcal{F}_{s}^{i,n}]=e^{\int_{0}^{s}\gamma(r,\hat{C}_{r}^{n,(-i)})\di \hat{\ell}_{r}^{i,n}}\E\bigl[\,Y\mathbbm{1}_{s<\tau^i}\mid\hat{\mathcal{F}}_{s}^{i,n}\bigr]\mathbbm{1}_{s<\tau^i}.
\] 
\end{lemma}
\begin{proof}
Firstly, the event $\{s<\tau^i\}$ is an element of $\mathcal{I}_{s}^{i,n}$, so we have 
\[
    \E[Y\mathbbm{1}_{s<\tau^i}\mid\mathcal{F}_{s}^{i,n}]=\E[Y\mathbbm{1}_{s<\tau^i}\mid\mathcal{F}_{s}^{i,n}]\mathbbm{1}_{s<\tau^i}
\]
Next, noting that the $\sigma$-algebra $\mathcal{F}_{s}^{i,n}$ is generated by events of the form $A\cap E$ for $A\in\hat{\mathcal{F}}_{s}^{i,n}$ and $E\in\mathcal{I}_{s}^{i,n}$, and noting that any such intersection satisfies $A\cap E\cap\{s<\tau^i\}=\emptyset$ or $A\cap E\cap\{s<\tau^i\}=A\cap\{s<\tau^i\}$, we deduce that the restriction $\E[Y\mathbbm{1}_{s<\tau^i}\mid\mathcal{F}_{s}^{i,n}]\!\!\upharpoonright_{\{s<\tau^i\}}$ is measurable for the restricted $\sigma$-algebra $\mathcal{\hat{F}}_{s}^{i,n}\!\!\upharpoonright_{\{s<\tau^i\}}$.
From this, it is easy to verify that $\E[Y\mathbbm{1}_{s<\tau^i}\mid\mathcal{F}_{s}^{i,n}]\!\!\upharpoonright_{\{s<\tau^i\}}$ satisfies the definition of $\E[Y\!\!\upharpoonright_{\{s<\tau^i\}}\mid\mathcal{\hat{F}}_{s}^{i,n}\!\!\upharpoonright_{\{s<\tau^i\}}]$, and hence we get
\begin{align*}
    \E[Y\mathbbm{1}_{s<\tau^i}\mid\mathcal{F}_{s}^{i,n}] & =\E[Y\!\!\upharpoonright_{\{s<\tau^i\}}\mid\mathcal{\hat{F}}_{s}^{i,n}\!\!\upharpoonright_{\{s<\tau^i\}}]\mathbbm{1}_{s<\tau^i}\\
    & =\frac{\E[Y\mathbbm{1}_{s<\tau^i}\mid\mathcal{\hat{F}}_{s}^{i,n}]}{\Prob(s<\tau^i\mid\mathcal{\hat{F}}_{s}^{i,n})}\mathbbm{1}_{s<\tau^i}=e^{\int_{0}^{s}\gamma(r,\hat{C}_{r}^{n,(-i)})\di \hat{\ell}_{r}^{i,n}}\E[Y\mathbbm{1}_{s<\tau^i}\mid\mathcal{\hat{F}}_{s}^{i,n}]\mathbbm{1}_{s<\tau^i}
\end{align*}
where the last equality comes from \eqref{eq:cond_exp_dist}. This completes the proof.
\end{proof}

We have the following consequence for the random variable obtained by stopping a stochastic process at one of the infection times.
\begin{lemma}\label{lem:compensator2}
Fix $i\in\{1,\ldots,n\}$. Let $\mathcal{F}^{i,n}$ and $\hat{\mathcal{F}}^{i,n}$ be given by  \eqref{eq:F_i_filtration} and \eqref{eq:recall_reduced_filtration}, and let $(Y_{t})_{t \geq 0}$ be a left-continuous process adapted to $(\hat{\mathcal{F}}_{t}^{i,n})_{t\geq 0}$. For any $0\leq s \leq t$, if $\E[\sup_{r\in [s,t]}|Y_{r}|]<\infty$, then 
\[
    \E[Y_{\tau^i\land t}\mathbbm{1}_{s<\tau^i}\mid\mathcal{F}_{s}^{i,n}]=\E\left[\int_{s}^{t}Y_{r}\gamma(r,\hat{C}_{r}^{n,(-i)})e^{-\int_{s}^{r}\gamma(\theta,\hat{C}_{\theta}^{n,(-i)})\di \hat{\ell}_{\theta}^{i,n}}\di \hat{\ell}_{r}^{i,n} + \tilde{Y}_t \mid \hat{\mathcal{F}}_{s}^{i,n}\right]\mathbbm{1}_{s<\tau^i},
\]
where $\tilde{Y}_t:=Y_t\,e^{-\int_{s}^{t}\gamma(\theta,\hat{C}_{\theta}^{n,(-i)})\di \hat{\ell}_{\theta}^{i,n}}$.
\end{lemma}
\begin{proof}
By Lemma \ref{lem:compensator1} applied to the random variable $Y_{\tau^i\land t}\in L^{1}(\Omega,\Prob)$, we have
\begin{equation}\label{eq:predict_marting}
    \E[Y_{\tau^i\land t}\mathbbm{1}_{s<\tau^i}\mid\mathcal{F}_{s}^{i,n}]=e^{\int_{0}^{s}\gamma(r,\hat{C}_{r}^{n,(-i)})\di \hat{\ell}_{r}^{i,n}}\E\bigl[Y_{\tau^i\land t}\mathbbm{1}_{s<\tau^i}\mid\hat{\mathcal{F}}_{s}^{i,n}\bigr]\mathbbm{1}_{s<\tau^i}.
\end{equation}
By left-continuity and adaptedness, we can approximate $Y$ by simple left-continuous processes
\[
    Y_{r}^{(k)}:=\sum_{l=1}^{k}A_{l}^{(k)}
    \mathbbm{1}_{s_{l}<r\leq s_{l+1}},
\]
on $[s,t]$, where each $A_{l}^{(k)}$ is $\hat{\mathcal{F}}_{s_{l}}^{i,n}$-measurable with $s_1=s$ and $s_{k+1}=t$. The tower law gives
    \[	\E\left[Y_{\tau^i\land t}^{(k)}\mathbbm{1}_{s<\tau^i}\mid\hat{\mathcal{F}}_{s}^{i,n}\right]=\sum_{l=1}^{k-1}\E\left[A_{l}^{(k)}\mathbb{E}\bigl[\mathbbm{1}_{s_{l}<\tau^i \leq s_{l+1}}\mid\hat{\mathcal{F}}_{s_{l}}^{i,n}\bigr]\mid\hat{\mathcal{F}}_{s}^{i,n}\right]+\E\left[A_{k}^{(k)}\mathbb{E}\bigl[ \mathbbm{1}_{s_k<\tau^i}\mid\hat{\mathcal{F}}_{s_k}^{i,n}\bigr]\mid\hat{\mathcal{F}}_{s}^{i,n}\right].
\]
Exploiting the tower law again, \eqref{eq:cond_exp_dist} and the chain rule for Stieltjes integrals gives
\begin{align*}
    \mathbb{P}( s_{l}<\tau^i \leq s_{l+1}\mid\hat{\mathcal{F}}_{s_{l}}^{i,n}) & =\E\left[e^{-\int_{0}^{s_{l}}\gamma(\theta,\hat{C}_{\theta}^{n,(-i)})\di \hat{\ell}_{\theta}^{i,n}} - e^{-\int_{0}^{s_{l+1}}\gamma(\theta,\hat{C}_{\theta}^{n,(-i)})\di \hat{\ell}_{\theta}^{i,n}}\mid\hat{\mathcal{F}}_{s_{l}}^{i,n}\right]
    \\
    &=\E\left[\int_{s_l}^{s_{l+1}}\gamma(r,\hat{C}_{r}^{n,(-i)}) e^{-\int_{0}^{r}\gamma(\theta,\hat{C}_{\theta}^{n,(-i)})\di \hat{\ell}_{\theta}^{i,n}} \di \hat{\ell}_{r}^{i,n}\mid\hat{\mathcal{F}}_{s_{l}}^{i,n}\right].
\end{align*}
Since each $A_{l}$ is $\hat{\mathcal{F}}_{s_{l}}^{i,n}$-measurable,
we thus arrive at
\begin{align*}
\E\left[Y_{\tau^i\land t}^{(k)}\mathbbm{1}_{s<\tau^i}\mid\hat{\mathcal{F}}_{s}^{i,n}\right]
    =\E\left[\int_{s}^{s_k}Y_{r}^{(k)}\gamma(r,\hat{C}_{r}^{n,(-i)})e^{-\int_{0}^{r}\gamma(\theta,\hat{C}_{\theta}^{n,(-i)})\di \hat{\ell}_{\theta}^{i,n}}\di \hat{\ell}_{r}^{i,n} + \tilde{A}^{(k)}_k \mid\hat{\mathcal{F}}_{s}^{i,n}\right]
\end{align*}
with $\tilde{A}^{(k)}_k:=A^{(k)}_{k}\,e^{-\int_{0}^{s_{k}}\gamma(\theta,\hat{C}_{\theta}^{n,(-i)})\di \hat{\ell}_{\theta}^{i,n}}$. By $\E[\sup_{r\leq t}|Y_{r}|]<\infty$, we can apply dominated convergence as $k\rightarrow \infty$ on both sides of the above expression. From this and \eqref{eq:predict_marting}, the conclusion follows.
\end{proof}
We are now in a position to verify the martingale property of the infected proportion. Combining the individual infection filtrations from \eqref{eq:I_filtration}, we define
\[
    \mathcal{I}_{t}^{n}=\sigma(\{s<\tau^{j}\}:s\in[0,t],\,j\in\{1,\ldots,n\}).
\]
Let $\mathcal{D}^{n}$ be the filtration generated by the initial points and the Brownian drivers, namely
\[
    \mathcal{D}_{t}^{n}:=\sigma(X_{0}^{j},B_{s}^{j}:s\in[0,t],\,j\in\{1,\ldots,n\}).
\]
We then define the combined filtration
\begin{equation}
    \mathcal{\bar{{F}}}_{t}^{n}:=\mathcal{D}_{t}^{n}\lor\mathcal{I}_{t}^{n},\label{eq:Filtration_F_bar}
\end{equation}
which we note has enough information to reconstruct the particle system up to time $t$, without revealing anything about the future. The following result is the first part of Theorem \ref{Thm:convergence_Barnes}.
\begin{prop}[Martingale property]\label{prop:Infected_proportion_martingale}
For all $n\geq1$, the stochastic process $(M_{t}^{n})_{t\geq0}$ given by
\[
    M_{t}^{n}:=I_{t}^{n}-V_t^n,\quad V_t^n=\frac{1}{n}\sum_{i=1}^{n}\int_{0}^{t}\mathbbm{1}_{s<\tau^i}\gamma(s,C_{s}^{n})\di \ell_{s}^{i,n},
\]
is a martingale with respect to the filtration $(\bar{\mathcal{F}}_{t}^{n})_{t\geq0}$ defined in \eqref{eq:Filtration_F_bar}.
\end{prop}      
\begin{proof}
Notice first that $M^{n}$ is indeed adapted to $(\mathcal{\bar{F}}_{t}^{n})_{t\geq0}$, since (i) the individual events $\{s<\tau^i\}$, for $s\leq t$, are in $\mathcal{\bar{F}}_{t}^{n}$ by virtue of $\mathcal{I}_{t}^{n}$, and (ii) combining these events with the information in $\mathcal{\hat{F}}_{t}^{n}$ is sufficient to reconstruct the full particle system with infection up to time $t$, so the processes $\ell^{i,n}$, $I^{n}$, and $C^n$ are also adapted, for $i=1,\ldots,n$. Next, fixing any $i\in\{1,\ldots,n\}$, we note that the filtration $\mathcal{F}_{t}^{i,n}$ can in fact be seen to contain all the events $\{s<\tau^{j}\}$, for $s\in[0,t]$ and $j\in\{1,\ldots,n\}$, so we deduce that $\bar{\mathcal{F}}_{t}^{n}$ is contained in $\mathcal{F}_{t}^{i,n}$ for each $i=1,\ldots,n$.
Write
\[
    V_t^n=\frac{1}{n}\sum_{i=1}^n V_t^{i,n} ,\quad  V_{t}^{i,n}:=\int_{0}^{t\land\tau^i}\gamma(s,C_{s}^{n})\di \ell_{s}^{i,n}\quad\text{for}\quad i=1,\ldots,n,
\]
we therefore have
\begin{equation}
    \E[M_{t}^{n}\mid\mathcal{\bar{F}}_{s}^{n}]=\E\Bigl[\frac{1}{n}\sum_{i=1}^{n}\E[\mathbbm{1}_{t\geq\tau^i}-V_{t}^{i,n}\mid\mathcal{F}_{s}^{i,n}]\mid\bar{\mathcal{F}}_{s}^{n}\Bigr]\label{eq:I_martin}
\end{equation}
for all times $s,t\geq0$. Now fix an arbitrary pair of times $s<t$. By writing
\begin{align*}
    \E[\mathbbm{1}_{t\geq\tau^i}-\mathbbm{1}_{s\geq\tau^i}\mid\mathcal{F}_{s}^{i,n}]
    & =\Prob(s<\tau^i\leq t\mid\mathcal{F}_{s}^{i,n})=\E[\mathbbm{1}_{s<\tau^i}\mathbbm{1}_{t\geq\tau^i}\mid\mathcal{F}_{s}^{i,n}],
\end{align*}
we see that Lemma \ref{lem:compensator1} gives
\begin{align*}
    \E[\mathbbm{1}_{t\geq\tau^i}-\mathbbm{1}_{s\geq\tau^i}\mid\mathcal{F}_{s}^{i,n}]
    & =\mathbbm{1}_{s<\tau^i}e^{\int_{0}^{s}\gamma(r,\hat{C}_{r}^{n,(-i)})\di \hat{\ell}_{r}^{i,n}}\Prob(s<\tau^i,\;t\geq\tau^i\mid\hat{\mathcal{F}}_{s}^{i,n})\\
    & =\mathbbm{1}_{s<\tau^i}e^{\int_{0}^{s}\gamma(r,\hat{C}_{r}^{n,(-i)})\di \hat{\ell}_{r}^{i,n}}\bigl(\Prob(t\geq\tau^i\mid\hat{\mathcal{F}}_{s}^{i,n})-\Prob(s\geq\tau^i\mid\hat{\mathcal{F}}_{s}^{i,n})\bigr)
\end{align*}
From here, we use the conditional law of $\tau^i$ from (\ref{eq:cond_exp_dist}) to write
\begin{align*}
    \Prob(t\geq\tau^i\mid\hat{\mathcal{F}}_{s}^{i,n})
    & =\E\bigl[\Prob(t\geq\tau^i\mid\hat{\mathcal{F}}_{t}^{i,n})\mid\hat{\mathcal{F}}_{s}^{i,n}\bigr]=1-\E[e^{-\int_{0}^{t}\gamma(r,\hat{C}_{r}^{n,(-i)})\di \hat{\ell}_{r}^{i,n}}\mid\hat{\mathcal{F}}_{s}^{i,n}],
\end{align*}
and so we arrive at
\begin{align}
    \E[\mathbbm{1}_{t\geq\tau^i}-\mathbbm{1}_{s\geq\tau^i}\mid\mathcal{F}_{s}^{i,n}]
    & =\mathbbm{1}_{s<\tau^i}e^{U^{i,n}_{0,s}}  \Bigl( e^{-U^{i,n}_{0,s}} - \E[e^{-U^{i,n}_{0,t}}  \mid\hat{\mathcal{F}}_{s}^{i,n}] \Bigr)\nonumber \\
    & =\mathbbm{1}_{s<\tau^i}\bigl(1-\E[e^{-U^{i,n}_t}\mid\hat{\mathcal{F}}_{s}^{i,n}]\bigr),\label{eq:indicator_marting}
\end{align}
where we have defined
\begin{equation*}
    U^{i,n}_{q,r} := \int_{q}^{r}\gamma(\theta,\hat{C}_{\theta}^{n,(-i)})\di \hat{\ell}_{\theta}^{i,n}, \quad r \ge q,\quad \text{and}\quad  U^{i,n}_r :=  U^{i,n}_{s,r}, \quad r \ge s.
\end{equation*}
From the construction of the particle system in Section \ref{sec:existence_sys}, we can observe that $\hat{C}_{\theta}^{n,(-i)}=C_{\theta}^{n}$ and $\hat{\ell}_{\theta}^{i,n}=\ell_{\theta}^{i,n}$ pathwise on the event $\{\theta<\tau^i\}$. Thus, we can write
\begin{align*}
    V_{t}^{i,n}-V_{s}^{i,n}
    &= U^{i,n}_{0, t\land\tau^i} - U^{i,n}_{0, s\land\tau^i} = \mathbbm{1}_{s<\tau^i} U^{i,n}_{ \tau^i\land t}.
\end{align*}
Applying Lemma \ref{lem:compensator2} to the continuous $\hat{\mathcal{F}}^{i,n}$-adapted process $Y=U^{i,n}$, we get
\begin{equation}\label{eq:V_U}
    \E\left[ V_{t}^{i,n}-V_{s}^{i,n} \mid \mathcal{F}_{s}^{i,n} \right] 
    = \mathbbm{1}_{s<\tau^i}\E\left[\int_{s}^{t} U^{i,n}_{r} e^{-U^{i,n}_r} \di U^{i,n}_r +  U^{i,n}_{t}\,e^{-U^{i,n}_t} \mid\hat{\mathcal{F}}_{s}^{i,n}\right].
\end{equation}
Since $U^{i,n}$ is continuous non-decreasing on $[s,t]$ with $U^{i,n}_s=0$, we see that
\[
    \int_s^t U^{i,n}_r\,e^{-U^{i,n}_r}\,\di U^{i,n}_r = \int_0^{U^{i,n}_t} u\,e^{-u}\,\di u = 1 - e^{-U^{i,n}_t} - U^{i,n}_t\,e^{-U^{i,n}_t}.
\]
so it follows from \eqref{eq:indicator_marting} and \eqref{eq:V_U} that
\[
    \E[\mathbbm{1}_{t\geq\tau^i}-V_{t}^{i,n}\mid\mathcal{F}_{s}^{i,n}]=\mathbbm{1}_{s\geq\tau^i}-V_{s}^{i,n}.
\]
Since $i$ was arbitrary, we can plug this back into the tower law computation \eqref{eq:I_martin} to conclude that
\[
    \E[M_{t}^{n}\mid\mathcal{\bar{F}}_{s}^{n}]=\E[M_{s}^{n}\mid\bar{\mathcal{F}}_{s}^{n}]=M_{s}^{n},
\]
using also that $M_{s}^{n}$ is $\bar{\mathcal{F}}_{s}^{n}$-measurable. As the times $s<t$ were arbitrary, this shows that $(M_{t}^{n})_{t\geq0}$ is a martingale with respect to the filtration $(\mathcal{\bar{F}}_{t}^{n})_{t\geq0}$.
\end{proof}

\subsection{Limiting behaviour as the population size tends to infinity}\label{subsect:proof_Barnes_thm}

It remains to confirm the asymptotic part of Theorem~\ref{Thm:convergence_Barnes}. To this end, we exploit the martingale property established above along with the fact that, with probability 1, no two infections can occur at the same time by Proposition \ref{prop:distinct_infections}.

\begin{proof}[Proof of Theorem \ref{Thm:convergence_Barnes}]
By Proposition \ref{prop:Infected_proportion_martingale}, the difference $I^n_t-V^n_t$ is a martingale in the filtration $(\bar{\mathcal{F}}_{t}^{n})_{t\geq0}$ from (\ref{eq:Filtration_F_bar}). Since the local times are continuous and of finite variation, we have that the quadratic variation is
\[
    [I^n-V^n]_t= \sum_{0<s \leq t} (\Delta I_s^n)^2 = \frac{1}{n^2}\sum_{i,j=1}^n \mathbbm{1}_{\{\tau^i=\tau^j\leq t\}}
\]
Applying the Burkholder--Davis--Gundy inequality, it therefore follows that, for all $t\geq0$,
\[
    \E \bigl[  \sup_{s\leq t}|I^n_s - V^n_s|^2 \bigr] \leq \frac{4}{n^2} \sum_{i,j=1}^n  \Prob (\tau^i=\tau^j \leq t ),\quad \text{for all} \quad n\geq 1.
\]
In view of Proposition \ref{prop:distinct_infections}, the off-diagonal terms in the sum are all zero, so we get
\[
    \E \bigl[  \sup_{s\leq t}|I^n_s - V^n_s|^2 \bigr] \leq \frac{C}{n},\quad \text{for all}\quad n\geq 1.
\]
Finally, Markov's inequality gives that there is then also convergence to zero uniformly on compacts in probability. This completes the proof.
\end{proof}

\begin{appendices}

\section{Recursive construction of the particle system}\label{appendix:particle_system}

This appendix provides the step-by-step construction of the particle system $\mathbf{X}^n$ satisfying Theorem~\ref{thm:existence} and the auxiliary particle systems $\mathbf{\hat{X}}^{n, (-i)}$ for $i = 1, \dots, n$ introduced in Section \ref{sec:existence_sys}. We first define a sequence of intermediate systems, which are then concatenated to form the global trajectories. The intermediate systems defined in Appendix~\ref{app:intermediate_systems} are utilized to construct both the true system $\mathbf{X}^n$ and the globally reflected system $\mathbf{\hat{X}}^n$, while those in Appendix~\ref{app:dismissed_infection} build the auxiliary systems $\mathbf{\hat{X}}^{n, (-i)}$. Finally, Appendix~\ref{app:concatenation_proofs} implements the precise concatenation of these piecewise components and confirms that this produces the correct local times.

\subsection{Intermediate systems}\label{app:intermediate_systems}
We construct the intermediate systems iteratively.
Let $\varsigma^{(k)}$ denote the $k^{\text{th}}$ infection time in the system $\mathbf{X}^n$. Setting $\varsigma^{(0)} := 0$, we partition the time horizon $[0,\infty)$ into a sequence of $n+1$  random intervals $[0, \varsigma^{(1)}), [\varsigma^{(1)}, \varsigma^{(2)}), \dots, [\varsigma^{(n)}, \infty)$.
For each of these $n+1$ time-intervals we define two intermediate systems, $\mathbf{X}^{n,(k)}$ and $\mathbf{\hat{X}}^{n,(k)}$ (where the index $k$ represents the $k^{th}$ time-interval, and also the $k^{th}$ step of the recursive construction). In the systems $\mathbf{X}^{n,(k)}$, infected particles are moved to the cemetery state $\dagger$, while in the systems $\mathbf{\hat{X}}^{n,(k)}$, particles, after infection, continue with fully reflected dynamics indefinitely. Both these intermediate systems are necessary to construct the global system $\mathbf{X}^n$.

\vspace{0.5cm}

\noindent\textbf{Step 1.} Define the system $\mathbf{\hat{X}}^{n, (1)} = (\hat{X}^{1,n, (1)}, \dots, \hat{X}^{n,n,(1)}) \in [A^{1,n,(1)}, \infty) \times \dots \times [A^{n,n,(1)}, \infty)$ as follows:
\begin{equation*}
    \left\{ \begin{array}{@{}l@{}l}
    \di \hat{X}_t^{i,n,(1)} = b^{(1)}(t, \hat{X}_t^{i,n,(1)}) \dt + \sigma^{(1)} (t, \hat{X}_t^{i,n,(1)}) \di B^{i,(1)}_t + \tfrac{1}{2}\di \ell_t^{A^{i,n,(1)}}(\hat{X}^{i,n,(1)}), & \quad t \ge 0, \\ \vspace{2pt}
    A_t^{i,n, (1)} = a_0,  & \quad t \ge 0, \\ \vspace{2pt}
    \hat{X}_0^{i,n,(1)} = X^i_0, &
    \end{array} \right.
\end{equation*}
for $i=1,\ldots,n$, where $X^{i}_0 >a_0$ and $a_0 \in \R_{\ge 0}$ are the given initial conditions; $b^{(1)}(t, x) := b(t, x)$ and $\sigma^{(1)} (t, x) := \sigma(t,x)$; and $B_t^{i,(1)} = B_t^i$. By Assumptions \ref{assump:coefficient_assumptions}--\ref{assump:inputs}, it follows from \cite[Thm.~1.2.1]{Pilipenko_book} that we have a unique strong solution $\mathbf{\hat{X}}^{n, (1)}$, noting that the result readily generalises to allow for the (suppressed) dependence on $X_0^i$ in $b^{(1)}$ and $\sigma^{(1)}$. Set $\mathbf{A}^{n, (1)} = (A^{1,n,(1)}, \dots, A^{n,n,(1)})$ and $\mathbf{B}^{(1)} = (B^{1,(1)}, \dots, B^{n,(1)} )$.

For all $i \in \{1, \dots, n\}$, we introduce elastic absorption at the boundary $A^{i,n,(1)}_t$ at rate $\gamma^{i,(1)}(t) : = \gamma(t,0)$. Let $\{\chi^{i,(1)}\}^{n}_{i=1}$ be a family of i.i.d.~exponential random variables, independent of all system components. Define the \textit{potential infection times}
\begin{equation*}
    \tilde{\varsigma}^{\,i,(1)} : = \inf\left\{ t\ge 0 \, : \, \int_0^t \gamma^{i,(1)}(s) \di \ell^{a_0}_s(\hat X^{i,n,(1)}) \ge \chi^{i,(1)} \right\},
\end{equation*}
so that the actual \textit{time of the first infection} is $\tilde{\varsigma}^{(1)} : = \min \{ \tilde\varsigma^{\, i,(1)}: i = 1, \dots, n \}$. Let $j^{(1)} \in \{1, \dots, n\}$ denote the random index of the particle that achieves this minimum if it is finite. Lemma \ref{lem:distinct_infections} below confirms that this index is uniquely defined with probability 1.

The intermediate system $\mathbf{X}^{n, (1)} = (X^{1,n, (1)}, \dots, X^{n,n, (1)}) \in [\mathbf{A}^{n, (1)}, \boldsymbol{\infty}) \cup \{\boldsymbol{\dagger}\}$ is then defined by setting the infected particle to the cemetery state. More precisely,
\begin{equation}\label{eq:sys_X_1}
    \left\{ \begin{array}{@{}l@{}l}
    X^{i,n,(1)}_t = \hat{X}^{i,n,(1)}_t, & \quad t \ge 0, \, i= 1,\dots, n, \, i \neq j^{(1)}, \\ \vspace{2pt}
    X^{j^{(1)},n,(1)}_t =\hat{X}^{j^{(1)},n,(1)}_t, & \quad t \in [0, \tilde{\varsigma}^{(1)}), \\ \vspace{2pt}
    X^{j^{(1)},n,(1)}_t = \dagger, & \quad t \ge \tilde{\varsigma}^{(1)}.
    \end{array} \right.
\end{equation}
By construction, the trajectories of \eqref{eq:sys_X_1} exactly coincide with the desired global system \eqref{eq:sys_X} on the interval $[0, \tilde{\varsigma}^{(1)})$. At the moment of this first infection, the system state must update, prompting the second step of the construction. To formalize this transition, we set $\varsigma^{(1)} := \tilde{\varsigma}^{(1)}$ (noting that $\varsigma^{(1)} = \tau^{j^{(1)}}$) and initialize the intermediate contagiousness index as $\hat{C}_t^{i,n,(1)} = 0$. 

\vspace{0.5cm}

\noindent\textbf{Step $\boldsymbol{k}$ ($\boldsymbol{k=2, \dots, n}$).} 
Let $J^{(k-1)} = \left\{ j^{(m)} \, : \, m =  1, \dots, k-1  \right\}$ denote the random set of indices for the particles that have been infected up to step $k-1$, with infection times $\tilde{\varsigma}^{(m)}$ within each step and cumulative infection time $\varsigma^{(k-1)}  := \varsigma^{(k-2)} + \tilde \varsigma^{(k-1)}$ defined recursively through the following construction. Consider the globally reflected intermediate system $\mathbf{\hat{X}}^{n,(k)} \in [\mathbf{A}^{n, (k)}, \boldsymbol{\infty})$ with dynamics:
\begin{equation}\label{eq:sys_X_hat_k}
    \left\{ \begin{array}{@{}l@{}l}
    \di \hat{X}_t^{i,n,(k)} = b^{(k)}(t, \hat{X}_t^{i,n,(k)}) \dt + \sigma^{(k)} (t, \hat{X}_t^{i,n,(k)}) \di B^{i,(k)}_t + \tfrac{1}{2}\di \ell_t^{A^{i,n,(k)}}(\hat{X}^{i,n,(k)}), & \quad  \forall i,  \\[4pt]
    \displaystyle A_t^{i,n, (k)} = A_{\tilde{\varsigma}^{(k-1)}}^{i,n,(k-1)}
    + \frac{\alpha}{n} \sum_{m=1}^{k-2} \int_{\varsigma^{(m)}}^{t+\varsigma^{(m)}} \hspace{-15pt} \varrho(t + \varsigma^{(k-1)}-s) \ds
    + \frac{\alpha}{n} \int_{0}^{t} \hspace{-5 pt} \varrho(t-s) \ds,  & \quad  i \notin J^{(k-1)}, \\[3pt]
    \displaystyle A_t^{i,n, (k)} = A_{\tilde{\varsigma}^{(k-1)}}^{i,n,(k-1)}
    + \frac{\alpha}{n} \sum_{m=1, \, j^{(m)} \neq i}^{k-2} \int_{\varsigma^{(m)}}^{t+\varsigma^{(m)}} \hspace{-15pt} \varrho(t + \varsigma^{(k-1)}-s) \ds
    + \frac{\alpha}{n} \int_{0}^{t} \hspace{-5 pt} \varrho(t-s) \ds,  & \quad i \in J^{(k-2)}, \\[3pt]
    \displaystyle A_t^{i,n, (k)} = A_{\tilde{\varsigma}^{(k-1)}}^{i,n,(k-1)}
    + \frac{\alpha}{n} \sum_{m=1}^{k-2} \int_{\varsigma^{(m)}}^{t+\varsigma^{(m)}} \hspace{-15pt} \varrho(t + \varsigma^{(k-1)}-s) \ds,  & \quad i=j^{(k-1)}, \\[2pt]
    \hat{X}_0^{i,n,(k)} = \hat{X}^{i,n,(k-1)}_{\tilde{\varsigma}^{(k-1)}}, & \quad \forall i,
    \end{array} \right.
\end{equation}
where the coefficients and the Brownian motions are given by $b^{(k)}(t, x) := b(t+\varsigma^{(k-1)}, x)$, $\sigma^{(k)} (t, x) := \sigma(t+\varsigma^{(k-1)}, x)$, and $B^{i,(k)}_t := B^i_{t + \varsigma^{(k-1)}} - B^i_{\varsigma^{(k-1)}}$.  Crucially, the way $\tilde{\varsigma}^{(k-1)}$ and $\varsigma^{(k-1)}$ were generated in the previous step (as per the above for $k=2$ and the constructions outlined in the next two paragraphs for $k\geq 3$) ensures that $\varsigma^{(k-1)}$ is a stopping time for $\mathcal{F}^n_t$ and that the boundary $A^{i,n,(k)}$ and the starting point $\hat{X}^{i,n,(k-1)}_{\tilde{\varsigma}^{(k-1)}}$ are measurable functions of the inputs in the previous steps. Shifting to the frame of the moving boundary as in \eqref{eq:A_is_AC}--\eqref{eq:shift_halfline}, our assumptions thus enable us to apply \cite[Thm.~1.2.1]{Pilipenko_book} on the positive half-line, noting that the result generalises to allow for the randomness in our coefficients and $A^{i,n,(k)}$. This gives us a unique strong solution $\mathbf{\hat{X}}^{n,(k)}$.

To evaluate the infection rate, we define the intermediate contagiousness index $C_t^{i,n,(k)}$, accounting for the time-shifted boundaries across previous steps:
\begin{equation}\label{eq:contagious_prop_k}
    C_t^{i,n,(k)} := \frac{1}{\alpha} \Big( A_t^{i,n,(k)} 
    - A_{t - \bar{d}}^{i,n,(k)} \mathbbm{1}_{\{\bar{d}\, \le \, t\}}
    - \sum_{m = -1}^{k-2} A_{t + \varsigma^{(k-1)}- \varsigma^{(m)} - \bar{d}}^{i,n,(m+1)} \mathbbm{1}_{ \{  \varsigma^{(m)} \,\le \, t + \varsigma^{(k-1)} - \bar{d}\, < \, \varsigma^{(m+1)} \}} \Big),
\end{equation}
where we recall that $\bar{d}$ is the duration of the infection, and we have introduced the convention $\varsigma^{(-1)} := - \infty$ and $A^{i,n,(0)} \equiv a_0$. The intermediate infection rate is then given by $\gamma^{i,(k)}(t) := \gamma(t + \varsigma^{(k-1)}, C_t^{i,n,(k)})$.

Let $\{\chi^{j,(k)}\}^{n}_{j=1}$ be the family of i.i.d.~exponential random variables from Assumption \ref{assump:inputs}, independent of all variables generated up to step $k-1$, and independent of the system dynamics at step $k$. For each $i \notin J^{(k-1)}$, define the \textit{potential next infection times}:
\begin{equation}\label{eq:potential_infection_times_k_step}
    \tilde{\varsigma}^{\, i,(k)} : = \inf\left\{ t\ge 0 \, : \, \int_0^t \gamma^{i,(k)}(s) \di \ell^{A^{i,n, (k)}}_s(\hat X^{i,n,(k)}) \ge \chi^{i,(k)} \right\}.
\end{equation}
We define the \textit{time of the $k^{\text{th}}$ infection} as $\tilde \varsigma^{(k)} := \min \left\{ \tilde{\varsigma}^{\, i,(k)}: i \notin J^{(k-1)} \right\}$, and let $j^{(k)} \notin J^{(k-1)}$ denote the random index of the particle that realizes this minimum if it is finite. Similarly to step 1, the index is uniquely defined with probability 1 by Lemma \ref{lem:distinct_infections}.

Finally, recalling that particles indexed in $J^{(k-1)}$ are already infected, the true intermediate system $\mathbf{X}^{n, (k)} \in [\mathbf{A}^{n, (k)}, \boldsymbol{\infty}) \cup \{\boldsymbol{\dagger}\}$ is given by:
\begin{equation}\label{eq:sys_X_k}
    \left\{ \begin{array}{@{}l@{}l}
    X^{i,n,(k)}_t = \hat{X}^{i,n,(k)}_t, & \quad t \ge 0, \, i \notin J^{(k-1)}, \, i \neq j^{(k)}, \\ \vspace{2pt}
    X^{i,n,(k)}_t = \dagger, & \quad t \ge 0, \, i \in J^{(k-1)}, \\ \vspace{2pt}
    X^{j^{(k)},n,(k)}_t =\hat{X}^{j^{(k)},n,(k)}_t, & \quad t \in [0, \tilde{\varsigma}^{(k)}), \\ \vspace{2pt}
    X^{j^{(k)},n,(k)}_t = \dagger, & \quad t \ge \tilde{\varsigma}^{(k)}.
\end{array} \right.
\end{equation}

\vspace{0.05cm}

\noindent\textbf{Step $\boldsymbol{n+1}$.} After the $n^{\text{th}}$ infection, all particles are infected. For the true system, we simply set $\mathbf{X}^{n, (n+1)}_t := \boldsymbol{\dagger}$ for $t \ge \varsigma^{(n)}$. For the reflected system, the globally reflected trajectories $\mathbf{\hat{X}}^{n,(n+1)} \in [\mathbf{A}^{n, (n+1)}, \boldsymbol{\infty})$ continue to evolve according to the dynamics in \eqref{eq:sys_X_hat_k} (with coefficients and boundaries shifted by $\varsigma^{(n)}$ analogously to previous steps), but without any further infection stopping times.

\begin{rmk}
    It is clear from our construction that the dynamics of particle $i$ in the systems $\mathbf{X}^{n, (k)}$ and $\mathbf{\hat{X}}^{n,(k)}$ are equivalent up until its infection time $\tau^i$. After infection, particle $i$ is moved to the `infected state' $\dagger$ in the systems $\mathbf{X}^{n, (k)}$, while it goes back to its reflected dynamics in $\mathbf{\hat{X}}^{n,(k)}$. Then, the $\mathbf{\hat{X}}^{n,(k)}$ are particle systems with global-in-time diffusive dynamics, where the interaction between the particles is still accounted for in the movement of the boundary.
\end{rmk}

\begin{lemma}[Unique infected particle at each step]\label{lem:distinct_infections}
At every step $k$ in the above construction, for any $i\neq j$, we have
\begin{equation}\label{eq:no_tie_step_k}
    \Prob\bigl(\tilde\varsigma^{i,(k)} = \tilde\varsigma^{j,(k)} = \tilde\varsigma^{(k)} < \infty,\; i,j \notin J^{(k-1)}\bigr) = 0.
\end{equation}
In particular, the random indices $j^{(k)}$, $k=1,\ldots,n$, are uniquely determined with probability 1.
\end{lemma}
\begin{proof}
For each step $k\in\{1,\ldots,n\}$ and particle $l\in\{1,\ldots,n\}$, we write 
\[
    U^{l,(k)}_s :=\displaystyle \int_0^s \gamma^{l,(k)}(r) \di \ell^{A^{l,n,(k)}}_r(\hat X^{l,n,(k)}),\quad l \notin J^{(k-1)},
\]
and set $U^{l,(k)}_s:=0$ for $l \in J^{(k-1)}$. Set also $\tilde\varsigma^{l,(k)}:=\infty$ for $l\in J^{(k-1)}$. With these conventions, \eqref{eq:potential_infection_times_k_step} reads as $\tilde\varsigma^{l,(k)}=\inf\{s\geq0: U^{l,(k)}_s\geq\chi^{l,(k)}\}$, where $U^{l,(k)}$ is continuous and non-decreasing for every $l$ and $k$. By the continuity and monotonicity of $U^{i,(k)}$, on $\{\tilde\varsigma^{i,(k)}<\infty\}$ we have $U^{i,(k)}_{\tilde\varsigma^{i,(k)}}=\chi^{i,(k)}$. Since $\tilde\varsigma^{j,(k)}=\tilde\varsigma^{i,(k)}$ on the tie event in \eqref{eq:no_tie_step_k}, this gives
\begin{equation}\label{eq:distinct_times_inclusion}
    \bigl\{\tilde\varsigma^{i,(k)}=\tilde\varsigma^{j,(k)}=\tilde\varsigma^{(k)}<\infty,\; i,j\notin J^{(k-1)}\bigr\} \subseteq \bigl\{U^{i,(k)}_{\tilde\varsigma^{j,(k)}}=\chi^{i,(k)}\bigr\}.
\end{equation}
Defining
\[
    \mathcal{D}^{(k-1)} := \sigma\bigl(X_0^l,B^l_s,\chi^{l,(m)} : s\geq 0, \, l\in\{1,\ldots,n\},\, m\leq k-1\bigr),
\]
we see that the step-$k$ trajectories $\hat X^{l,n,(k)}$, the boundaries $A^{l,n,(k)}$, and the rate coefficients $\gamma^{l,(k)}$, as well as the random set $J^{(k-1)}$ are $ \mathcal{D}^{(k-1)}$-measurable.  Consequently, each $U^{l,(k)}$ is $\mathcal{D}^{(k-1)}$-measurable, and the event $\{i,j\notin J^{(k-1)}\}$ is in $\mathcal{D}^{(k-1)}$. With $U^{i,(k)}_\infty:=\sup_{s\geq0} U^{i,(k)}_s$, $\tilde\varsigma^{j,(k)}$ and $U^{i,(k)}_{\tilde\varsigma^{j,(k)}}$ are therefore $\mathcal{D}^{(k-1)}\vee\sigma(\chi^{j,(k)})$-measurable, so we conclude that
\begin{equation*}
    \Prob\bigl(U^{i,(k)}_{\tilde\varsigma^{j,(k)}} = \chi^{i,(k)}\bigr) = 0, 
\end{equation*}
since, by construction, $\chi^{i,(k)}$ is an Exponential random variable independent of $\mathcal{D}^{(k-1)}\vee\sigma(\chi^{j,(k)})$. In view of \eqref{eq:distinct_times_inclusion}, this completes the proofs.
\end{proof}

\subsection{Dismissing infection of a tagged particle}\label{app:dismissed_infection}

The intermediate systems needed for the construction of the auxiliary systems $\mathbf{\hat{X}}^{n, (-i)}$ for each $i \in \{1, \dots, n\}$ are also defined recursively, as in the previous subsection. System $\mathbf{\hat{X}}^{n, (-i)}$ and the true system $\mathbf{X}^n$ are pathwise identical up to the infection time of particle $i$ in $\mathbf{X}^n$, meaning their intermediate constructions coincide up until that point. However, because particle $i$ is immune in $\mathbf{\hat{X}}^{n, (-i)}$, the construction must diverge as soon as the infection of particle $i$ happens in $\mathbf{X}^n$. As before, each step of the construction of $\mathbf{\hat{X}}^{n, (-i)}$ ends when a new infection happens: since $i$ is immune, we need only account for $n-1$ infections, so we have a total of $n$ steps, instead of $n+1$. We denote by $\mathbf{\hat{X}}^{n, (-i),(k)}$ the intermediate systems for each step of this construction.

Fix $i \in \{1, \dots, n\}$. First, we define a sequence of times representing when particles \textit{other than $i$} are infected in the \textit{true} intermediate systems $\mathbf{X}^{n, (k)}$. If particle $i$ is the $m^{\text{th}}$ particle to get infected in the true system (i.e., $i = j^{(m)}$), we define
\begin{equation}\label{eq:xi_i_true}
    \xi^{i}_k := \varsigma^{(k)} \quad \mathrm{for} \: k=1,\ldots,m-1, \qquad \mathrm{and} \qquad  \xi^{i}_k := \varsigma^{(k+1)} \quad \mathrm{for} \: k=m,\ldots,n-1,
\end{equation}
where the $\varsigma^{(k)}$ are defined as in Appendix~\ref{app:intermediate_systems}. The $\xi^i_k$ are used in the definition of the system specified in Proposition~\ref{prop:evol_x_hat}.

We now move to the construction of the intermediate systems $\mathbf{\hat{X}}^{n, (-i),(k)} \in [ \mathbf{\hat{A}}^{n, (-i), (k)}, \boldsymbol{\infty}) \cup \{\boldsymbol{\dagger}\}$. For $k = 1, \ldots, n-1$,
set the \textit{potential infection times}
\[
    \tilde\varsigma^{j,(k),(-i)}
    \;:=\; \inf\bigl\{ t \ge 0 \,:\, \textstyle\int_0^t \gamma^{i,(k),(-i)}(s) \di\ell_s^{\hat A^{j,n,(-i),(k)}}(\hat X^{j,n,(-i),(k)}) \ge \chi^{j,(k)}\bigr\}, \quad j \ne i,
\]
and let the \textit{time of the $k^{th}$ infection} be
\begin{equation}\label{eq:infection_times_k_step_immune}
    \tilde\xi^{(-i),(k)} \;:=\; \min\bigl\{\tilde\varsigma^{j,(k),(-i)} \,:\, j \in J^{(-i),(k-1)}\bigr\},
\end{equation}
where $J^{(-i),(k-1)} \subseteq \{1, \ldots, n\} \setminus \{i\}$ is the set of surviving particles (diferent from $i$) entering step $k$. We initialize $J^{(-i),(0)} = \{1,\ldots,n\}\setminus\{i\}$, and at each step the minimiser $j^{(-i),(k)}$ is removed, noting that it is unique by Lemma \ref{lem:distinct_infections}. We let the \textit{cumulative infection times} be defined as $\xi^{(-i),(k)} = \xi^{(-i),(k-1)} + \tilde\xi^{(-i),(k)}$, with $\xi^{(-i),(0)} := 0 $.  For all $j \in \{1, \dots, n\}$, the boundary $\hat A^{j, n,(-i),(k)}$ and contagiousness index $\hat C^{j, n,(-i),(k)}$ follow the same recursive form as in \eqref{eq:sys_X_hat_k} and \eqref{eq:contagious_prop_k}, with the random increments $\tilde\xi^{(-i),(\cdot)}$ and cumulative infection times $\xi^{(-i),(\cdot)}$ in place of $\tilde\varsigma^{(\cdot)}$ and $\varsigma^{(\cdot)}$. With these changes, and applying all other required notational replacements wherever necessary, the construction of $\mathbf{\hat{X}}^{n, (-i),(k)}$ now proceeds as that of $\mathbf{X}^{n,(k)}$ in Appendix~\ref{app:intermediate_systems}, for $k = 1, \dots, n-1$. Finally, at step $n$ of the construction, we are left with particle $i$ reflecting off the boundary indefinitely, while all other $n-1$ particles have been removed to $\dagger$.

\begin{rmk}
We emphasize the distinction between ${\xi}^{(-i),(k)}$ and $\xi^i_k$: while $\xi^i_k$ from \eqref{eq:xi_i_true} merely acts as a shifted index that ignores the $m^{\text{th}}$ stopping time $\varsigma^{(m)}$, the times ${\xi}^{(-i),(k)}$ defined above for $k \ge m$ represent a genuine stochastic divergence in the pathwise evolution of the auxiliary system. We have $\xi^i_k = \xi^{(-i), (k)}$ for $k = 1, \dots, m-1$. Similarly, the intermediate systems $\mathbf{\hat{X}}^{n, (-i),(k)}$ and $\mathbf{X}^{n,(k)}$, and the moving boundaries $\mathbf{\hat{A}}^{n, (-i), (k)}$ and $\mathbf{A}^{n, (k)}$ are pathwise equivalent for $k=1, \dots, m-1$, but diverge for $k \ge m$, since particle $i$ has never contributed to the boundary advancement in this construction, which in turn impacts the particles' dynamics.
\end{rmk}

\subsection{Concatenation and proofs}\label{app:concatenation_proofs}

We now carry out the concatenation of the piecewise intermediate systems from Appendix~\ref{app:intermediate_systems} and Appendix~\ref{app:dismissed_infection} to construct the global trajectories. This yields the true system $\mathbf{X}^n$ satisfying Theorem~\ref{thm:existence} together with the corresponding globally reflected system $\mathbf{\hat{X}}^n$ and the tagged auxiliary systems $\mathbf{\hat{X}}^{n, (-i)}$ for each particle $i \in \{ 1, \dots, n\}$.

\begin{defn}[Piecewise construction]\label{defn:piecewise_process}
Let $\{ \mathfrak{P}^{(k)}_\cdot \}_{k=1}^{n+1} $ be one of the collections of processes $\{ \hat X^{j,n,(k)}_\cdot \}_k$, $\{ X^{j,n,(k)}_\cdot \}_k$, or $\{ A^{j,n,(k)}_\cdot \}_k  $ defined in Appendix~\ref{app:intermediate_systems}. For all $t \ge 0$, we construct their piecewise concatenation $\mathfrak{P}_t$ by
\begin{equation}
    \mathfrak{P}_t:= \mathfrak{P}^{(1)}_t \mathbbm{1}_{t \in [0, \varsigma^{(1)})}
    + \sum_{k=2}^{n} \mathfrak{P}^{(k)}_{t-\varsigma^{(k-1)}} \mathbbm{1}_{t \in [\varsigma^{(k-1)},  \varsigma^{(k)})}
    + \mathfrak{P}^{(n+1)}_{t-\varsigma^{(n)}} \mathbbm{1}_{t  \ge \varsigma^{(n)}}. \label{eq:general_P_defn_by_parts}
\end{equation}
Similarly, for any fixed $i \in \{1, \dots, n\}$, let $\{ \mathfrak{Y}^{(-i),(k)}_\cdot \}_{k=1}^{n} $ denote a collection of the processes $\{ \hat X^{j,n, (-i),(k)}_\cdot \}_k$ or $\{ \hat A^{j,n, (-i),(k)}_\cdot \}_k$ from Appendix~\ref{app:dismissed_infection}. Using the partition times $\xi^{(-i),(k)}$, we construct their global concatenation $\mathfrak{Y}^{(-i)}_t$ for all $t \ge 0$ by:
\begin{equation} \label{eq:general_P_defn_by_parts_minus_i}
    \mathfrak{Y}^{(-i)}_t := \mathfrak{Y}^{(-i),(1)}_t \mathbbm{1}_{t \in [0, \xi^{(-i),(1)})}
    + \sum_{k=2}^{n-1} \mathfrak{Y}^{(-i),(k)}_{t-\xi^{(-i),(k-1)}} \mathbbm{1}_{t \in [\xi^{(-i),(k-1)},  \xi^{(-i),(k)})}
    + \mathfrak{Y}^{(-i),(n)}_{t-\xi^{(-i),(n-1)}} \mathbbm{1}_{t  \ge \xi^{(-i),(n-1)}}.
\end{equation}
\end{defn}

The remainder of this section is dedicated to the proof of Proposition \ref{prop:evol_x_hat} along with some technical lemmas utilised here and in Section \ref{sec:existence_sys}.

\begin{proof}[Proof of Proposition~\ref{prop:evol_x_hat}]
Recall the definitions of $\hat{X}_t^{i,n,(k)}$ for $k = 1, \dots, n+1$ in Appendix~\ref{app:intermediate_systems}. According to Definition~\ref{defn:piecewise_process}, we rewrite their dynamics in integral form, so that for $t \ge 0$ we have
\begin{align}\label{eq:X_hat_first}
    \hat{X}_t^{i,n} 
    & = \left(  X_0^{i,n}
    + \int_0^{t} b^{(1)}(s, \hat{X}_s^{i,n,(1)}) \ds
    + \int_0^{t} \sigma^{(1)} (s, \hat{X}_s^{i,n,(1)}) \di B^{i,(1)}_s \right) \mathbbm{1}_{t \in [0, \varsigma^{(1)})} \nonumber \\
    &\quad + \sum_{k=2}^{n} \left(  \hat{X}_0^{i,n,(k)}
    + \int_{0}^{t-\varsigma^{(k-1)}} \hspace{-20pt} b^{(k)}(s, \hat{X}_s^{i,n,(k)}) \ds
    + \int_{0}^{t - \varsigma^{(k-1)}} \hspace{-20pt} \sigma^{(k)} (s, \hat{X}_s^{i,n,(k)}) \di B^{i,(k)}_s \right)
    \mathbbm{1}_{t \in [\varsigma^{(k-1)},  \varsigma^{(k)})} \nonumber \\
    &\quad + \left(  \hat{X}_0^{i,n,(n+1)}
    + \int_{0}^{t-\varsigma^{(n)}} \hspace{-20pt} b^{(n+1)}(s, \hat{X}_s^{i,n,(n+1)}) \ds
    + \int_{0}^{t - \varsigma^{(n)}} \hspace{-20pt} \sigma^{(n+1)} (s, \hat{X}_s^{i,n,(n+1)}) \di B^{i,(n+1)}_s \right) \mathbbm{1}_{t  \ge \varsigma^{(n)}}  \nonumber \\
    &\quad + \int_0^{t} \tfrac{1}{2}\di \ell_{s}^{A^{i,n,(1)}}(\hat{X}^{i,n,(1)}) \, \mathbbm{1}_{t \in [0,\varsigma^{(1)})}
    + \sum_{k=2}^{n} \int_0^{t - \varsigma^{(k-1)}} \hspace{-30pt} \tfrac{1}{2}\di \ell_{s}^{A^{i,n, (k)}}(\hat{X}^{i,n,(k)}) \, \mathbbm{1}_{t \in [\varsigma^{(k-1)}, \varsigma^{(k)})} \nonumber \\
    &\quad + \int_0^{t - \varsigma^{(n)}} \hspace{-20pt} \tfrac{1}{2}\di \ell_{s}^{A^{i,n, (n+1)}}(\hat{X}^{i,n,(n+1)}) \, \mathbbm{1}_{t  \ge \varsigma^{(n)}}.
\end{align}
From the definitions of $X_0^{i,n,(k)}, \tilde \varsigma^{(k)}$ and $\varsigma^{(k)}$ (recalling that $\varsigma^{(0)} =0$), we see that
\begin{align}\label{eq:X_hat_k_0_formula_1}
    \hat{X}_0^{i,n,(k)}
    &= \hat{X}_{\tilde{\varsigma}^{(k-1)}}^{i,n,(k-1)}  = \hat{X}_{0}^{i,n,(k-1)}  + \int_0^{\tilde{\varsigma}^{(k-1)}} \hspace{-20pt} \di \hat{X}_{s}^{i,n,(k-1)} 
    = X_{0}^{i,n}  +  \sum_{m = 1}^{k-1} \int_0^{\tilde{\varsigma}^{(m)}} \hspace{-20pt} \di \hat{X}_{s}^{i,n,(m)} \nonumber \\
    &= X_{0}^{i,n} + \sum_{m = 1}^{k-1} \left( \int_{\varsigma^{(m-1)}}^{\varsigma^{(m)}} \hspace{-10pt} b^{(m)}(s - \varsigma^{(m-1)}, \hat{X}_{s - \varsigma^{(m-1)}}^{i,n,(m)}) \ds
    \right. \nonumber \\
    &\quad + \left. \int_{\varsigma^{(m-1)}}^{\varsigma^{(m)}} \hspace{-10pt} \sigma^{(m)} (s - \varsigma^{(m-1)}, \hat{X}_{s - \varsigma^{(m-1)}}^{i,n,(m)}) \di B^{i,(m)}_{s - \varsigma^{(m-1)}} +\tfrac{1}{2} \ell_{\tilde \varsigma^{(m)}}^{A^{i,n,(m)}} (\hat{X}^{i,n,(m)}) \right),
\end{align}
where we have applied the change of variables $s \mapsto s + \varsigma^{(m-1)}$ to get the last equality. By construction (see \textit{Step $k$} in Appendix~\ref{app:intermediate_systems}), we have that $b^{(m)}(s-\varsigma^{(m-1)},Y_s) = b(s, Y_s)$ and $\sigma^{(m)}(s-\varsigma^{(m-1)},Y_s) = \sigma (s, Y_s)$. Moreover, the stochastic increments $B^{i,(m)}_{s_{j+1} - \varsigma^{(m-1)}} - B^{i,(m)}_{s_{j} - \varsigma^{(m-1)}}$ are equal to $B^i_{s_{j+1}} - B^i_{s_j}$ for any discretization $\{ s_j \}$ of the time-interval $(\varsigma^{(m-1)}, \varsigma^{(m)})$. Finally, noting that $\hat{X}^{i,n}_s \,  \mathbbm{1}_{s \in [\varsigma^{(m-1)}  , \varsigma^{(m)})} =  \hat{X}_{s - \varsigma^{(m-1)}}^{i,n,(m)}$ by \eqref{eq:general_P_defn_by_parts}, we can simplify \eqref{eq:X_hat_k_0_formula_1} as
\begin{align}\label{eq:X_hat_k_0_formula}
    \hat{X}_0^{i,n,(k)}
    &= X_{0}^{i,n} +
    \int_0^{\varsigma^{(k-1)}} \hspace{-10pt} b (s, \hat{X}_s^{i,n}) \ds
    + \int_0^{\varsigma^{(k-1)}} \hspace{-10pt} \sigma (s, \hat{X}_s^{i,n}) \di B^i_s
    + \sum_{m = 1}^{k-1} \tfrac{1}{2}\ell_{\tilde \varsigma^{(m)}}^{A^{i,n,(m)}} (\hat{X}^{i,n,(m)}).
\end{align}	
By similar computations, and using \eqref{eq:X_hat_k_0_formula} into \eqref{eq:X_hat_first}, we finally have:
\begin{align*}
    \hat{X}_t^{i,n} 
    &= X_0^{i,n} + \int_0^{t} b(s, \hat{X}_s^{i,n}) \ds
    + \int_0^{t} \sigma(s, \hat{X}_s^{i,n}) \di B^{i}_s
    + \ell_{t}^{A^{i,n,(1)}}(\hat{X}^{i,n,(1)}) \, \mathbbm{1}_{t \in [0,\varsigma^{(1)})} \\
    &\quad + \frac{1}{2} \sum_{k=2}^{n} \left( \sum_{m = 1}^{k-1} \ell_{\tilde \varsigma^{(m)}}^{A^{i,n,(m)}} (\hat{X}^{i,n,(m)}) + \ell_{t - \varsigma^{(k-1)} }^{A^{i,n, (k)}}(\hat{X}^{i,n,(k)}) \right) \mathbbm{1}_{t \in [\varsigma^{(k-1)}, \varsigma^{(k)})} \\
    &\quad + \frac{1}{2}\left( \sum_{m = 1}^{n} \ell_{\tilde \varsigma^{(m)}}^{A^{i,n,(m)}} (\hat{X}^{i,n,(m)}) + \ell_{t - \varsigma^{(n)}}^{A^{i,n, (n+1)}}(\hat{X}^{i,n,(n+1)}) \right) \mathbbm{1}_{t  \ge \varsigma^{(n)}}.
\end{align*}
Lemma~\ref{lemma:local_time} and Corollary~\ref{cor:local_time} below deal with the local time terms, so that the above further simplifies to
\begin{equation*}
    \hat{X}_t^{i,n} 
    = X_0^{i,n} + \int_0^{t} b(s, \hat{X}_s^{i,n}) \ds
    + \int_0^{t} \sigma(s, \hat{X}_s^{i,n}) \di B^{i}_s
    + \tfrac{1}{2}\ell_{t}^{A^{i,n}}(\hat{X}^{i,n}).
\end{equation*}

We now move on to showing that the evolution equation of the process $A_t^{i,n}$ constructed according to \eqref{eq:general_P_defn_by_parts} satisfies the required equation in \eqref{eq:sys_X_hat}. Recalling the notation used in Appendix \ref{app:intermediate_systems}, we have that $j^{(k)}$ denotes the index of the particle infected at the $k^{th}$ step of the construction. Since by the end of our construction all $n$ particles have been infected, we can assume without loss of generality that $i = j^{(m)}$ for some $m \in \left\{ 1, \dots, n \right\}$. We compute $A^{i,n,(k)}_{t - \varsigma^{(k-1)}}$ for $t \in [\varsigma^{(k-1)},\varsigma^{(k)})$ distinguishing three cases: $k = 1, \dots, m$, $k = m+1$ and $k = m+2, \dots, n+1$.
\\
For $k = 1, \dots, m$ one can easily check by induction that
\begin{equation*}
    A_{t - \varsigma^{(k-1)}}^{i,n,(k)} = a_0 + \frac{\alpha}{n} \sum_{j = 1}^{k-1} \int_{\varsigma^{(j)}}^t \varrho(t-s) \ds, \quad \mathrm{for} \: t \in [\varsigma^{(k-1)},\varsigma^{(k)}),
\end{equation*}
using the change of variables $s \mapsto s+\varsigma^{(k-1)}$ and $s \mapsto s +t - \varsigma^{(k-1)}$ when appropriate.
\\
For the case $k=m+1$, we have that
\begin{align*}
    A_{t - \varsigma^{(m)}}^{i,n,(m+1)}
    &= A_{\tilde\varsigma^{(m)}}^{i,n,(m)}
    + \frac{\alpha}{n} \sum_{j = 1}^{m-1} \int_{\varsigma^{(j)}}^{t + \varsigma^{(j)} - \varsigma^{(m)}} \hspace{-35pt} \varrho(t-s) \ds \\
    &= a_0 + \frac{\alpha}{n} \sum_{j = 1}^{m-1} \left( \int_{\varsigma^{(j)}}^{\varsigma^{(m)}} \hspace{-15pt} \varrho(\varsigma^{(m)}-s) \ds + \int_{\varsigma^{(j)}}^{t + \varsigma^{(j)} - \varsigma^{(m)}} \hspace{-35pt} \varrho(t-s) \ds \right) \\
    &= a_0 + \frac{\alpha}{n} \sum_{j = 1}^{m-1} \int_{\varsigma^{(j)}}^t \varrho(t-s) \ds,
    \quad \mathrm{for} \: t \in [\varsigma^{(m)},\varsigma^{(m+1)}),
\end{align*}
where we have applied the change of variables $s \mapsto s+t-\varsigma^{(m)}$ to get the last equality. Finally, for $k = m+2, \dots, n$ again by induction we can show that
\begin{equation*}
    A_{t - \varsigma^{(k-1)}}^{i,n,(k)} = a_0 + \frac{\alpha}{n} \sum_{j = 1}^{m-1} \int_{\varsigma^{(j)}}^t \varrho(t-s) \ds + \frac{\alpha}{n} \sum_{j = m+1}^{k-1} \int_{\varsigma^{(j)}}^t \varrho(t-s) \ds, \quad \mathrm{for} \: t \in [\varsigma^{(k-1)},\varsigma^{(k)}).
\end{equation*}
Similarly for the final interval $t \ge \varsigma^{(n)}$, we have $A_{t - \varsigma^{(n)}}^{i,n,(n+1)} = a_0 + \frac{\alpha}{n} \sum_{j \neq m}^{n} \int_{\varsigma^{(j)}}^t \varrho(t-s) \ds$.

Recalling \eqref{eq:xi_i_true} for the definition of the random times $\xi^i_k$, we have that $\xi_k^i = \varsigma^{(k)}$ for $k = 1, \dots, m-1$ and $\xi_k^i = \varsigma^{(k+1)}$ for $k = m, \dots, n-1$. This sequence explicitly omits $\varsigma^{(m)}$. Therefore, for all $k = 1, \dots, n+1$, we can express the evaluated boundaries in a unified manner in terms of $\left\{\xi^i_k \right\}$ as
\begin{equation*}
    A_{t - \varsigma^{(k-1)}}^{i,n,(k)}
    = a_0 + \frac{\alpha}{n} \sum_{j = 1, j \neq m}^{n} \int_{0}^t \varrho(t-s) \mathbbm{1}_{s \ge \varsigma^{(j)}} \ds
    = a_0 + \frac{\alpha}{n} \sum_{j = 1}^{n-1} \int_{0}^t \varrho(t-s) \mathbbm{1}_{s \ge \xi^i_j} \ds, \quad \mathrm{for} \: t \in [\varsigma^{(k-1)},\varsigma^{(k)}).
\end{equation*}
Then by the piecewise concatenation \eqref{eq:general_P_defn_by_parts}, we have that, for all $t \ge 0$,
\begin{equation*}
    A^{i,n}_{t} = a_0 + \frac{\alpha}{n} \sum_{j = 1}^{n-1} \int_{0}^t \varrho(t-s) \mathbbm{1}_{s \ge \xi^i_j} \ds
    = a_0 + \alpha \int_{0}^t \varrho(t-s) I_s^{i,n} \ds, 
\end{equation*}
where we define the process $I_t^{i,n} = \frac{1}{n} \sum_{j=1}^{n-1} \mathbbm{1}_{[0,t]}(\xi^i_j)$. This matches the required dynamics in \eqref{eq:sys_X_hat}, which concludes the proof of Proposition~\ref{prop:evol_x_hat}.
\end{proof}

For the above proof, we relied on the following observations, which we now prove.

\begin{lemma}\label{lemma:local_time}
For $k=2, \dots, n+1$, and $t \in [\varsigma^{(k-1)}, \varsigma^{(k)})$,
\begin{equation*}
    \sum_{m = 1}^{k-1} \ell_{\tilde \varsigma^{(m)}}^{A^{i,n,(m)}} (\hat{X}^{i,n,(m)}) + \ell_{t - \varsigma^{(k-1)} }^{A^{i,n, (k)}}(\hat{X}^{i,n,(k)}) = \ell_{t}^{A^{i,n}}(\hat{X}^{i,n}), \quad \forall i = 1, \dots, n.
\end{equation*}
\end{lemma}

\begin{corollary}\label{cor:local_time}
For $t < \varsigma^{(1)}$,
\begin{equation*}
    \ell_{t}^{A^{i,n,(1)}}(\hat{X}^{i,n,(1)}) = \ell_{t}^{A^{i,n}}(\hat{X}^{i,n}), \quad \forall i = 1, \dots n.
\end{equation*}
For $t \ge \varsigma^{(n)}$,
\begin{equation*}
    \sum_{m = 1}^{n} \ell_{\tilde \varsigma^{(m)}}^{A^{i,n,(m)}} (\hat{X}^{i,n,(m)}) + \ell_{t - \varsigma^{(n)}}^{A^{i,n, (n+1)}}(\hat{X}^{i,n,(n+1)})
    = \ell_{t}^{A^{i,n}}(\hat{X}^{i,n}), \quad \forall i = 1, \dots n.
\end{equation*}
\end{corollary}

\begin{proof}[Proof of Lemma \ref{lemma:local_time}]
Consider $\ell_{t-\varsigma^{(k-1)}}^{A^{i,n, (k)}}(\hat X^{i,n,(k)})$ for $t \in [\varsigma^{(k-1)}, \varsigma^{(k)})$. By the definition of the local time in \eqref{eq:local_time_defn},
\begin{align*}
    \ell_{t-\varsigma^{(k-1)}}^{A^{i,n, (k)}}(\hat X^{i,n,(k)})
    &= \lim_{\varepsilon \rightarrow 0} \frac{1}{\varepsilon} \int_{0}^{t-\varsigma^{(k-1)}} \mathbbm{1}_{[A^{i,n,(k)}_s, A^{i,n,(k)}_s + \varepsilon)} (\hat X_s^{i,n,(k)}) \bigl( \sigma^{(k)}(s, \hat X_s^{i,n,(k)}) \bigr)^2 \hspace{-7pt} \ds  \\
    &= \lim_{\varepsilon \rightarrow 0} \frac{1}{\varepsilon} \int_{\varsigma^{(k-1)}}^{t} \mathbbm{1}_{[A^{i,n,(k)}_{s-\varsigma^{(k-1)}}, A^{i,n,(k)}_{s-\varsigma^{(k-1)}} + \varepsilon)} (\hat X_{s-\varsigma^{(k-1)}}^{i,n,(k)})\bigl( \sigma^{(k)}(s-\varsigma^{(k-1)}, \hat X_{s-\varsigma^{(k-1)}}^{i,n,(k)}) \bigr)^2 \hspace{-7pt} \ds  \\
    &= \lim_{\varepsilon \rightarrow 0} \frac{1}{\varepsilon} \int_{\varsigma^{(k-1)}}^{t} \mathbbm{1}_{[ A^{i,n}_s, A^{i,n}_s + \varepsilon)}(\hat X^{i,n}_s) \di \langle  \hat X^{i,n} \rangle_s,
\end{align*}
where we have applied the usual change of variables $s \mapsto s + \varsigma^{(k-1)}$ and used that, for
    $s \in [\varsigma^{(k-1)}, \varsigma^{(k)})$,
    $\hat X^{i,n,(k)}_{s-\varsigma^{(k-1)}} = \hat X^{i,n}_s$ and $A_{s-\varsigma^{(k-1)}}^{i,n,(k)} =  A_s^{i,n}$ by Definition~\ref{defn:piecewise_process}.

Similarly, for $m = 1, \dots, k-1$, we have that 
\begin{equation*}
    \ell_{\tilde \varsigma^{(m)}}^{A^{i,n,(m)}} (\hat{X}^{i,n,(m)})
    = \lim_{\varepsilon \rightarrow 0} \frac{1}{\varepsilon} \int_{\varsigma^{(m-1)}}^{\varsigma^{(m)}} \mathbbm{1}_{[ A^{i,n}_s, A^{i,n}_s + \varepsilon]}( \hat X_s^{i,n}) \di \langle  \hat X^{i,n} \rangle_s.
\end{equation*}
Summing these expressions together, we conclude that
\begin{equation*}
    \sum_{m = 1}^{k-1} \ell_{\tilde \varsigma^{(m)}}^{A^{i,n,(m)}} (\hat{X}^{i,n,(m)}) + \ell_{t - \varsigma^{(k-1)} }^{A^{n, (k)}}(\hat{X}^{i,n,(k)})
    = \lim_{\varepsilon \rightarrow 0} \frac{1}{\varepsilon} \int_{0}^{t} \mathbbm{1}_{[ A^{i,n}_s, A^{i,n}_s + \varepsilon]}(\hat X^{i,n}_s) \di \langle  \hat X^{i,n} \rangle_s
    = \ell_{t}^{A^{i,n}}(\hat{X}^{i,n}),
\end{equation*}
for all $t \in [\varsigma^{(k-1)}, \varsigma^{(k)})$.
\end{proof}
Finally, we used the following lemma in the proof of Proposition~\ref{prop:distribution_tau_i} in Section~\ref{sec:existence_sys}.

\begin{lemma}\label{lemma:rate_integral}
For all $i = 1, \dots, n$ and $k = 1, \dots, n$,
\begin{equation*}
    \int_0^{\tilde\varsigma^{(k)}} \hspace{-8pt} \gamma^{i,(k)}(s) \di \ell_s^{A^{i,n,(k)}} (\hat X^{i,n,(k)}) = 
    \int_{\varsigma^{(k-1)}}^{\varsigma^{(k)}} \hspace{-12pt} \gamma(s, C^{i,n}_s) \di \ell_s^{A^{i,n}} (\hat X^{i,n}).
\end{equation*}
\end{lemma}
\begin{proof}
Define a partition $\mathsf{P}_m$ of $[\varsigma^{(k-1)},\varsigma^{(k)}]$; applying the change of variables $s \mapsto s + \varsigma^{(k-1)}$ and writing down the random Stieltjes integral as an infinite sum, we have:
\begin{align*}
    \int_0^{\tilde\varsigma^{(k)}} \hspace{-8pt} \gamma^{i,(k)}(s) &\di \ell_s^{A^{i,n,(k)}} (\hat X^{i,n,(k)})
    = \int_{\varsigma^{(k-1)}}^{\varsigma^{(k)}} \hspace{-8pt} \gamma^{i,(k)}(s-\varsigma^{(k-1)}) \di \ell_{s-\varsigma^{(k-1)}}^{A^{i,n,(k)}} (\hat X^{i,n,(k)}) \\
    &= \lim_{m \rightarrow \infty}
    \sum_{s_j \in \mathsf{P}_m} \gamma^{i,(k)}(s_j-\varsigma^{(k-1)}) \left(
    \ell_{s_{j+1} -\varsigma^{(k-1)}}^{A^{i,n,(k)}}(\hat{X}^{i,n,(k)})
    - \ell_{s_j - \varsigma^{(k-1)}}^{A^{i,n,(k)}}(\hat{X}^{i,n,(k)}) \right).
\end{align*}
By similar computations to those in the proof of Lemma \ref{lemma:local_time}, we have that
\begin{align*}
    \ell_{s_{j+1} -\varsigma^{(k-1)}}^{A^{i,n,(k)}}(\hat{X}^{i,n,(k)})
    - \ell_{s_j - \varsigma^{(k-1)}}^{A^{i,n,(k)}}(\hat{X}^{i,n,(k)})
    &= \lim_{\varepsilon \rightarrow 0} \frac{1}{\varepsilon} \int_{s_{j}}^{s_{j+1}}
    \mathbbm{1}_{[A_s^{i,n}, A_s^{i,n} + \varepsilon]} (\hat X^{i,n}) \di \langle \hat X^{i,n} \rangle_s \\
    &= \ell_{s_{j+1}}^{A^{i,n}} (\hat X^{i,n}) - \ell_{s_j}^{A^{i,n}} (\hat X^{i,n}),
\end{align*}
and in turn also
\begin{equation*}
    \int_0^{\tilde\varsigma^{(k)}} \hspace{-8pt} \gamma^{i,(k)}(s) \di \ell_s^{A^{i,n,(k)}} (\hat X^{i,n,(k)})
    = \int_{\varsigma^{(k-1)}}^{\varsigma^{(k)}} \hspace{-8pt} \gamma^{i,(k)}(s-\varsigma^{(k-1)}) \di \ell_{s}^{A^{i,n}} (\hat X^{i,n}).
\end{equation*}
By the definition of $\gamma^{i,(k)}(t)$ and $C^{i,n, (k)}_t$ in Appendix~\ref{app:intermediate_systems} (see \eqref{eq:contagious_prop_k}), the claim follows.
\end{proof}
\end{appendices}


\bibliographystyle{alpha}
\bibliography{biblio_epidemic}

\end{document}